\documentclass[a4paper, 11pt, twoside, leqno]{article}

\usepackage[T1]{fontenc}
\usepackage[utf8x]{inputenc}
\usepackage[english]{babel}
\usepackage{lmodern}               
\usepackage{amsmath,amsthm,amssymb}
\usepackage{mathrsfs,esint,bbm,stmaryrd}
\usepackage[shortlabels]{enumitem}
\usepackage{graphicx,xcolor,subcaption}
\usepackage[textwidth=16.1cm,centering,headheight=24pt]{geometry}
\usepackage{authblk}

\newtheorem{theorem}{Theorem}           
\newtheorem{corollary}[theorem]{Corollary}
\newtheorem{lemma}[theorem]{Lemma}
\newtheorem{prop}[theorem]{Proposition}

\newtheorem{mainthm}{Theorem}           

\theoremstyle{definition}              

\newtheorem*{notation}{Notation}

\theoremstyle{remark}                  
\newtheorem{step}{Step}

\newtheorem{remark}{Remark}

\DeclareMathOperator{\Id}{Id}                                       
\DeclareMathOperator{\tr}{tr}                                       
\DeclareMathOperator{\dist}{dist}                                   
\DeclareMathOperator{\spt}{spt}                                     
\DeclareMathOperator{\curl}{curl}                                  
\let\div\relax
\DeclareMathOperator{\div}{div}                                     

\DeclareMathOperator{\SBV}{SBV}

\DeclareMathOperator{\loc}{loc}

\newcommand{\abs}[1]{\left| #1 \right|}                             
\newcommand{\norm}[1]{\left\| #1 \right\|}                          
\newcommand{\one}{\mathbbm{1}}                                      

\newcommand{\csubset}{\subset\!\subset}                             

\DeclareMathAlphabet{\mathpzc}{OT1}{pzc}{m}{it}

\newcommand{\D}{\mathrm{D}}       

\renewcommand{\d}{\mathrm{d}}
\renewcommand{\o}{\mathrm{o}}

\newcommand{\R}{\mathbb{R}}
\newcommand{\Z}{\mathbb{Z}}
\newcommand{\C}{\mathbb{C}}
\renewcommand{\SS}{\mathbb{S}}
\renewcommand{\P}{\mathbf{P}}    
\newcommand{\Q}{\mathbf{Q}}
\newcommand{\M}{\mathbf{M}}
\newcommand{\I}{\mathbf{I}}
\newcommand{\X}{\mathbf{X}}
\newcommand{\N}{\mathbf{N}}
\renewcommand{\j}{\mathbf{j}}
\newcommand{\n}{\mathbf{n}}
\newcommand{\m}{\mathbf{m}}

\newcommand{\q}{\mathbf{q}}

\renewcommand{\u}{\mathbf{u}}
\newcommand{\NN}{\mathscr{N}}     

\newcommand{\F}{\mathscr{F}}
\renewcommand{\H}{\mathscr{H}}

\newcommand{\eps}{\varepsilon}

\newcommand{\Cpot}{{\rm C}_{\rm pot}}
\newcommand{\RP}{\mathbb{R}\mathrm{P}}

\SetSymbolFont{stmry}{bold}{U}{stmry}{m}{n}  
\newcommand{\nnu}{{\boldsymbol{\nu}}}

\newcommand{\ttau}{{\boldsymbol{\tau}}}

\newcommand{\Sz}{\mathscr{S}_0^{2\times 2}}
\newcommand{\Qb}{\Q_{\mathrm{bd}}}
\newcommand{\Mb}{\M_{\mathrm{bd}}}

\renewcommand{\S}{\mathrm{S}}

\title{The formation of gradient-driven singular structures of
codimension one and two in two-dimensions: The case study of ferronematics \\
Part~I: Energy estimates and compactness results}

\date{}
\author{Giacomo~Canevari, Federico~Luigi~Dipasquale, Bianca~Stroffolini}

\newcommand{\Addresses}{{
  \bigskip
  \footnotesize

  Giacomo~Canevari \\
  \textsc{Universit\`{a} di Verona} \\
  Strada Le Grazie 15, 37134 Verona, Italy \\
  \textit{E-mail address}: \texttt{giacomo.canevari@univr.it}

  \medskip

  Federico~Luigi~Dipasquale \\
  \textsc{Scuola Superiore Meridionale}\\
  Via Mezzocannone 4, 80138 Napoli, Italy\\
  \textit{E-mail address}: \texttt{f.dipasquale@ssmeridionale.it}

    \medskip

    Bianca~Stroffolini \\
    \textsc{Dipartimento di Matematica e Applicazioni ``Renato Caccioppoli'',}\\
    \textsc{Universit\`{a} degli studi di Napoli ``Federico II''}\\
    Via Cintia, Monte S. Angelo, 80126 Napoli, Italy\\ 
  \textit{E-mail address}: \texttt{bstroffo@unina.it}
}}

\begin{document}

\maketitle

\begin{abstract}
We study a two-dimensional variational model for ferronematics --- composite materials formed by dispersing magnetic nanoparticles into a liquid crystal matrix. The model features two coupled order parameters: a Landau-de Gennes~$\Q$-tensor for the liquid crystal component and a magnetisation vector field~$\M$, both of them governed by a Ginzburg-Landau-type energy. The energy, the largest part of which is carried by the $\Q$-component, includes a singular coupling term favouring alignment between~$\Q$ and~$\M$. In this article and in the companion paper~\cite{CDS2}, we analyse the asymptotic behaviour of (not necessarily minimizing) critical points as a small parameter~$\eps$ tends to zero. 
In this paper, we prove that the (rescaled) energy density for the $\Q$-component,  concentrates, to leading order, on a finite number of singular points. Moreover, we prove energy estimates and compactness results that will be crucially used in~\cite{CDS2} to determine the structure of the energy concentration set for the $\M$-component as well as the relationship between the two singular sets.



\end{abstract}

The formation of gradient-driven singular structures of codimension one and two has been the object of intensive investigation over the past few decades. Singular structures of codimension one, i.e.~interfaces, arise in both scalar and vectorial problems, with notable examples in the literature
on diffusion-reaction equations, such as the Allen-Cahn equations for instance.
While the mathematical theory on the scalar Allen-Cahn equation is well-advanced, much less is known in the vectorial case. The available results on Allen-Cahn systems focus on the asymptotic analysis, via~$\Gamma$-convergence, of minimising solutions~\cite{Baldo, FonsecaTartar} --- with the exception of a recent paper~\cite{Bethuel-AC-Acta}, which obtains convergence results for non-minimising solutions of Allen-Cahn systems in dimension two.
On the other hand, singular objects of codimension two arise only in vectorial problems. A prototypical example is the Ginzburg-Landau functional for superconductivity, which has been extensively studied since the seminal monograph~\cite{BBH}. There is now quite a well-established theory on Ginzburg-Landau solutions, encompassing both minimising and non-minimising ones as well as solutions to evolution problems and, to a lesser degree, generalisations to different functionals with a similar structure.

In this work, we investigate a two-dimensional problem
where singularities of both codimension one and two emerge,
due to the coupling between two different components of the problem.
The physical motivation for this model
is the study of \emph{ferronematics},
a class of composite materials obtained as suspensions
of magnetic nanoparticles in a nematic liquid crystal host~\cite{Brochard,MLDC}.
According to the approach proposed in~\cite{bisht2019}, ferronematics are described by two order parameters.
The orientation of the liquid crystal molecules
is described by the Landau-de Gennes $\Q$-tensor,
which is a map from the physical domain~$\Omega\subseteq\R^2$
to the space~$\Sz$ of~$2\times 2$, symmetric, real matrices
with trace equal to zero. Nonzero values of~$\Q$ correspond
to liquid crystal configurations with a well-defined direction of
molecular alignment, represented by the
eigenspace of~$\Q$ associated with the positive eigenvalue,
while~$\Q = 0$ indicates an isotropic state,
where all the directions of molecular alignment are equally likely.
The distribution of magnetic nanoparticles is
described by the average magnetisation vector,
$\M\colon\Omega\to\R^2$. The system is governed by a
free energy functional which depends on both~$\Q$ and~$\M$:
\begin{equation}\label{eq:Feps}
	\F_\eps(\Q, \M) :=  \int_G \left\{ \frac{1}{2} \abs{\nabla \Q}^2 + \frac{\eps}{2} \abs{\nabla \M}^2 + \frac{1}{\eps^2} f_\eps(\Q,\M)\right\} {\d}x,
\end{equation}
where~$\eps$ is a non-dimensional parameter.
The interaction between the liquid crystal host and
the magnetic inclusions is mediated by the potential~$f_\eps$,
which takes the form
\begin{equation}\label{eq:potential}
	f_\eps(\Q, \M) := \frac{1}{4}\left(1-\abs{\Q}^2\right)^2
	+ \frac{\eps}{4}\left(1-\abs{\M}^2\right)^2 - \eps \beta \, \Q \M \cdot \M +  \kappa_\eps.
\end{equation}
Here~$\beta> 0$ is given and~$\kappa_\eps$ is an additive
constant, depending on~$\eps$ and~$\beta$ only,
uniquely determined by imposing that~$\inf f_\eps = 0$.
Potential energies similar to~\eqref{eq:potential} have been derived via suitable homogenisation limits~\cite{Caldereretal}.
For positive values of~$\beta$, the potential~$f_\eps$
promotes alignment between the liquid crystal molecules
and the magnetisation vector.
Indeed, the potential~$f_\eps(\Q, \, \M)$ is minimised
by pairs~$(\Q_{\eps}^{\mathrm{pot}}, \, \M_{\eps}^{\mathrm{pot}})$
that satisfy the conditions
\begin{equation} \label{minpot}
	\abs{\M_{\eps}^{\mathrm{pot}}} = \lambda_{\eps,  \beta}, \qquad
	\Q_{\eps}^{\mathrm{pot}} = \sqrt{2} s_{\eps, \beta}\left( \frac{\M_{\eps}^{\mathrm{pot}} \otimes \M_{\eps}^{\mathrm{pot}}}{\lambda_{\eps,\beta}^2} - \frac{\I}{2}  \right)
\end{equation}
where~$\mathbf{I}$ is the~$2\times 2$ identity matrix
and~$\lambda_{\eps, \beta}$, $s_{\eps,\beta}$ are
positive constants, uniquely determined by~$\eps$ and~$\beta$,
such that
\[
 \lambda_{\eps,\beta} \to\left(\sqrt{2}\beta + 1\right)^{1/2},
 \qquad s_{\eps,\beta} \to 1
\]
as~$\eps\to 0$ (see~\cite[Lemma~B.2]{CanevariMajumdarStroffoliniWang}).
We are interested in the limit as~$\eps\to 0$,
which is physically motivated as the ``large domain limit''
(i.e., the size of the domain is much larger than
the typical correlation length for the liquid crystal molecules).
The asymptotic analysis of the Ginzburg-Landau functional
shows that the energy for the~$\Q$-component concentrates
around a finite number of points,
which correspond to \emph{non-orientable} singularities
of the limiting liquid crystal configuration.
In particular, the $\Q$-component of minimisers
converges to a matrix-valued map that admits no
continuous orthonormal frame of eigenvectors.
However, as the coupling promotes alignment between~$\M$
and the eigenvectors of~$\Q$, the energy for the~$\M$-component
should concentrate on singular lines, corresponding to
jumps in the eigenvector frame.

The asymptotic analysis of minimisers for~$\F_\eps$,
subject to Dirichlet boundary conditions,
has been carried out
in the paper~\cite{CanevariMajumdarStroffoliniWang}.
In the present paper, instead, we consider general
\emph{critical points}~$(\Q_\eps, \, \M_\eps)$
of the functional~$\F_\eps$, that is,
finite-energy solutions of the
Euler-Lagrange system of equations
\begin{align}
  -&\Delta\Q_\eps + \dfrac{1}{\eps^2}(\abs{\Q_\eps}^2 - 1)\Q_\eps
  - \dfrac{\beta}{\eps}\left(\M_\eps\otimes\M_\eps
  - \dfrac{\abs{\M_\eps}^2}{2}\I\right)  = 0  \label{EL-Q} \\
  -&\Delta\M_\eps + \dfrac{1}{\eps^2}(\abs{\M_\eps}^2 - 1)\M_\eps
  - \dfrac{2\beta}{\eps^2}\Q_\eps\M_\eps = 0. \label{EL-M}
\end{align}
At first sight, both~\eqref{EL-Q} and~\eqref{EL-M} look like
perturbed Ginzburg-Landau systems. However, 
\eqref{EL-M} can be regarded as the Euler-Lagrange equation of
a `perturbed' vectorial Allen-Cahn-type energy functional
with wells depending on~$\Q_\eps$.
Indeed, the coupling term favours~$\M_\eps$ to align with the eigenvector of~$\Q_\eps$ corresponding to the positive eigenvalue, so that for each nozero value of~$\Q_\eps$, there are exactly two values of~$\M_\eps$ that minimise the potential.
We will consider Dirichlet boundary conditions
for~$\Q_\eps$ and either Dirichlet or Neumann boundary conditions for~$\M_\eps$ (see~\eqref{hp:bc}, \eqref{hp:bcbis} below for details).

There are physical motivations behind the analysis of general critical points, including both stable and unstable ones.
In setting up variational models for physical systems, there is some
consensus that observable configurations (often also referred to as \emph{stable
configurations}, according to Gibbs) should correspond
to minimisers or local minimisers
of the energy, subject to boundary conditions as appropriate.
However, there are experimentally and numerically observed
cases \cite{Cahn, CahnHilliard, GurtinMatano} in which, before
decaying into a state of minimal or locally minimal energy
an unperturbed system remains in a state of \emph{unstable equilibrium} (also called a
\emph{metastable phase}) for quite a long time, with respect to the
typical time-scale of the overall dynamics.
It is been argued in~\cite{GurtinMatano} that
such unstable equilibria should correspond
to general critical points, so that their
mathematical analysis is physically justified.
These general remarks are particularly relevant
to the functional~\eqref{eq:Feps} we are considering,
since numerical results~\cite{bisht2019, CanevariMajumdarStroffoliniWang} show
an abundance of critical points, stable and unstable.

In the study of the system~\eqref{EL-Q},~\eqref{EL-M}, a major difficulty is
provided by the coupling term.
Roughly speaking, while we expect the integral of~$\frac{\beta}{\eps}\Q_\eps\M_\eps\cdot\M_\eps$
in the free energy functional to be relatively small
because of the alignment between~$\M_\eps$ and~$\Q_\eps$,
the coupling terms in the equations~\eqref{EL-Q}, \eqref{EL-M}
have large prefactors, of order~$\frac{1}{\eps}$
or~$\frac{1}{\eps^2}$. To deal with these issues, 
here and in the companion paper~\cite{CDS2}, we
adapt and combine methods from the literature
on the Ginzburg-Landau functional~\cite{BBH, BBO}
with the approach introduced in the recent
paper~\cite{Bethuel-AC-Acta} for non-minimising solutions
of the vectorial Allen-Cahn problem.

\subsection*{Assumptions and setting of the problem}

We consider two alternative sets of boundary conditions
for~$\Q_\eps$ and~$\M_\eps$.
One option is to impose
Dirichlet boundary conditions for both~$\Q_\eps$ and~$\M_\eps$:
\begin{equation} \label{bc}
 \Q_\eps = \Qb, \quad \M_\eps = \Mb \qquad \textrm{on } \partial\Omega.
\end{equation}
We assume the boundary data~$\Qb\in C^1(\partial\Omega, \, \Sz)$,
$\Mb\in C^1(\partial\Omega, \, \R^2)$ are~$\eps$-independent
maps that satisfy
\begin{equation} \label{hp:bc}
 \abs{\Mb} = \left(\sqrt{2}\beta + 1\right)^{1/2},
 \qquad \Qb = \sqrt{2}\left(\frac{\Mb\otimes\Mb}{\sqrt{2}\beta + 1}
- \frac{\I}{2}\right)
\end{equation}
at any point of~$\partial\Omega$.
This particular form of the boundary datum
approximates the set of minimisers~\eqref{minpot} for the potential and
ensures that the potential energy is small,
i.e.~$f_\eps(\Qb, \, \Mb) \leq C\eps^2$
for some $\eps$-independent constant~$C$
(see Lemma~\ref{lemma:feps} below).
Alternatively to~\eqref{bc}, we consider
`mixed' boundary conditions, i.e. Dirichlet boundary
conditions for~$\Q_\eps$ and homogeneous Neumann boundary
conditions for~$\M_\eps$:
\begin{equation} \label{bcbis}
 \Q_\eps = \Qb, \quad \partial_\nnu\M_\eps = 0
 \qquad \textrm{on } \partial\Omega,
\end{equation}
where~$\nnu$ is the exterior unit normal to~$\partial\Omega$.
We then assume the boundary datum~$\Qb$
(does not depend on~$\eps$ and) takes the form
\begin{equation} \label{hp:bcbis}
 \Qb = \sqrt{2}\left(\n_{\mathrm{bd}}\otimes\n_{\mathrm{bd}} - \frac{\mathbf{I}}{2}\right)
\end{equation}
on~$\partial\Omega$, for some
map~$\n_{\mathrm{bd}}\in C^1(\partial\Omega, \, \R^2)$
which, a priori, is completely independent of the
values of~$\M_\eps$ on the boundary.

Regardless of our choice of boundary conditions,
we will always \emph{assume} that there exists a
constant~${\rm C}_{\rm pot} > 0$, independent of~$\eps$, such that
\begin{equation} \label{hp:potential_bound}
 \frac{1}{\eps^2}\int_\Omega f_\eps(\Q_\eps, \, \M_\eps) \leq {\rm C}_{\rm pot}
\end{equation}
for any~$\eps$ small enough.

\begin{remark}
	If $\{\left(\Q_\eps^\star, \M_\eps^\star\right)\}$ is a sequence of
	\emph{minimisers} of $\F_\eps$, subject to
	either Dirichlet~\eqref{bc}--\eqref{hp:bc}  or mixed boundary conditions~\eqref{bcbis}--\eqref{hp:bcbis},
	then the assumption~\eqref{hp:potential_bound} is automatically satisfied
	(for the pure Dirichlet boundary conditions~\eqref{bc}--\eqref{hp:bc},
	this follows from \cite[Lemma~4.3 and Lemma~4.10]{CanevariMajumdarStroffoliniWang};
	for the `mixed' boundary conditions~\eqref{bcbis}--\eqref{hp:bcbis}, 
	the arguments are completely analogous, see~\cite{CDS3}).
\end{remark}

\begin{remark}
	If the domain~$\Omega$ is star-shaped, then solutions
	to~\eqref{EL-Q}--\eqref{EL-M} subject to
	Dirichlet boundary conditions as in~\eqref{bc}--\eqref{hp:bc}
	necessarily satisfy the condition~\eqref{hp:potential_bound}.
	This is a consequence of a Pohozaev identity
	(see Lemma~\ref{lemma:pohozaev} below),
	exactly as in~\cite{BBH}. By contrast,
	solutions to~\eqref{EL-Q}--\eqref{EL-M} subject to
	\emph{mixed} boundary conditions need
	not satisfy~\eqref{hp:potential_bound}.
	Indeed, if~$\Q_\eps$ is a critical point of the
	(uncoupled) Ginzburg-Landau functional subject
	to appropriate Dirichlet conditions, then~$(\Q_\eps, \, 0)$
	is a solution to~\eqref{EL-Q}--\eqref{EL-M}
	with boundary conditions~\eqref{bcbis}, yet
	\[
	 \frac{1}{\eps^2} \int_\Omega f_\eps(\Q_\eps, \, 0)
	 = \left(\frac{1}{4\eps} + \frac{\kappa_\eps}{\eps^2}\right) \abs{\Omega}
	 = \mathrm{O}\left(\frac{1}{\eps}\right) \! ,
	\]
	for Lemma~\ref{lemma:feps} below
	implies that~$\kappa_\eps = \mathrm{O}(\eps)$ as~$\eps\to 0$.
\end{remark}

\begin{remark}\label{rk:consequences-Cpot}
Under the assumption~\eqref{hp:potential_bound}
(and for boundary data as in~\eqref{bc}--\eqref{hp:bc}
or~\eqref{bcbis}--\eqref{hp:bcbis}), we will later \emph{prove} that
 \begin{equation}\label{eq:consequences-Cpot}
  \int_{\Omega}\abs{\nabla\Q_\eps}^2 \leq C \abs{\log\eps},
  \qquad \eps \int_{\Omega}\abs{\nabla\M_\eps}^2 \leq C
 \end{equation}
 for some constant~$C$ that does not depend on~$\eps$
 (see Proposition~\ref{prop:Q_bound} and~\ref{prop:M_bound},
 respectively) depending only on $\beta$, ${\rm C}_{\rm pot}$,
 and the boundary data.
\end{remark}

We consider the functions
\begin{align}
	\mu_\eps &:= \frac{1}{\abs{\log\eps}}
	\left( \frac{1}{2}\abs{\nabla \Q_\eps}^2 + \frac{\eps}{2} \abs{\nabla \M_\eps}^2 + \frac{1}{\eps^2}f_\eps(\Q_\eps,\,\M_\eps) \right) \label{eq:def-mu-eps} \\
	\nu_\eps &:=  \frac{\eps}{2} \abs{\nabla \M_\eps}^2 + \frac{1}{\eps^2}f_\eps(\Q_\eps,\,\M_\eps) \label{eq:def-nu-eps}
\end{align}
Under the assumption~\eqref{hp:potential_bound}
and for boundary data as in~\eqref{bc}--\eqref{hp:bc}
or~\eqref{bcbis}--\eqref{hp:bcbis}, by Remark~\ref{rk:consequences-Cpot}
it follows that
that the families~$(\mu_\eps)_{\eps > 0}$, $(\nu_\eps)_{\eps > 0}$
are bounded in~$L^1(\Omega)$.
Given a map~$\Q\in W^{1,2}(\Omega, \, \Sz)$,
we define the pre-Jacobian of~$\Q$ as
a vector-field~$j(\Q)\colon\Omega\to\R^2$,
given component-wise as
\[
 j(\Q) := \left(Q_{11} \partial_1 Q_{12} - Q_{12} \partial_1 Q_{11}, \, Q_{11} \partial_2 Q_{12} - Q_{12} \partial_2 Q_{11}\right)
\]
(see~\eqref{preJac} below for more details).

\subsubsection*{Convergence results: Theorem~\ref{mainthm:asymp}}
The main result of this paper is the following theorem.
\begin{mainthm}\label{mainthm:asymp}
	Let~$\Omega\subseteq\R^2$ be a bounded, simply connected domain of class~$C^2$.
	Let $\{(\Q_\eps, \,\M_\eps)\}\subset W^{1,2}(\Omega, \, \Sz)\times W^{1,2}(\Omega, \, \R^2)$ be a sequence of critical
	points of $\F_\eps$ subject to either~\eqref{bc}--\eqref{hp:bc}
	or to~\eqref{bcbis}--\eqref{hp:bcbis}.
	Assume that the condition~\eqref{hp:potential_bound}
	is satisfied. Then, as $\eps \to 0$,
	\begin{enumerate}[(i)]
		\item\label{item:mainthm-asymp-conv-QM} $\Q_\eps \to \Q_\star$ strongly in $W^{1,p}(\Omega)$ for any $p < 2$
		and $\M_\eps \to \M_\star$ strongly in $L^p(\Omega)$ for any $p < +\infty$.
		\item\label{item:mainthm-asymp-Q*M*-j} The limiting maps~$\Q_\star$, $\M_\star$
		satisfy
		\[
			\abs{\Q_\star(x)} = 1, \quad
			\abs{\M_\star(x)} = \left(\sqrt{2}\beta + 1\right)^{1/2}, \quad
			\Q_\star(x) = \sqrt{2} \left( \frac{\M_\star(x) \otimes \M_\star(x)}{\sqrt{2}\beta+1} - \frac{\I}{2} \right)
		\]
		for a.e.~$x\in\Omega$ and
		\begin{equation}\label{eq:Q*-harm-intro}
			\div j(\Q_\star) = 0
		\end{equation}
		in the sense of distributions in~$\Omega$.
		\item\label{item:mainthm-asymp-conv-mu*nu*} $\mu_\eps \rightharpoonup \mu_\star$ and
		$\nu_\eps \rightharpoonup \nu_\star$ weakly*
		as measures in $\Omega$.
		\item\label{item:mainthm-asymp-supports} $\spt\mu_\star$ is finite. 
		\item\label{item:mainthm-asymp-strong-conv} $\Q_\eps \to \Q_\star$ strongly in
		$W^{1,p}_{\rm loc}(\Omega \setminus \spt\mu_\star)$ for any $p < +\infty$.
	\end{enumerate}
\end{mainthm}

Theorem~\ref{mainthm:asymp} shows, in particular, that the energy of a sequence
of critical points concentrates on a finite set~$\spt\mu_\star$
of singular points, arising from the Ginzburg-Landau terms
in~$\Q_\eps$. 
We also emphasise that, by~\eqref{eq:Q*-harm-intro} 
and item~\ref{item:mainthm-asymp-strong-conv},  
$\Q_\star$ is a smooth harmonic map in $\Omega \setminus \spt\mu_\star$.

As we are going to discuss in more detail 
later on, the proof of Theorem~\ref{mainthm:asymp} is 
fundamentally rooted in the classical Ginzburg-Landau theory,  
as developed in \cite{BBH, BBO},  
although the presence of the coupling terms in 
the Euler-Lagrange equations~\eqref{EL-Q},~\eqref{EL-M} calls 
for special care and non-trivial refinements of the arguments 
in \cite{BBH, BBO}.
The situation changes completely when coming to the analysis 
of the energy densities $\nu_\eps$, the reason being that their 
behaviour is basically controlled by the system~\eqref{EL-M}, 
which, as mentioned above, should be 
properly considered as a perturbed Allen-Cahn system.
%
%
%
Correspondingly, the arguments  
should be completely different from 
those for the measures $\mu_\eps$, and follow instead the lines of 
the asymptotic analysis of critical points of the vectorial 
Allen-Cahn functional, as developed in the recent 
breakthrough paper~\cite{Bethuel-AC-Acta}. 
Nonetheless, since the arguments in~\cite{Bethuel-AC-Acta} are 
ultimately based on refined energy estimates and since the 
systems~\eqref{EL-Q},~\eqref{EL-M} are coupled, 
the estimates and compactness properties obtained  
in the present paper will be crucial to cope with the consequences of 
the presence of the coupling term in~\eqref{EL-M}. 
However, because of the significantly different framework, we deferred the 
study of the asymptotic behaviour of energy densities $\nu_\eps$ 
to the companion paper~\cite{CDS2}.

%

\subsection*{Structure of the proof}
The proof of Theorem~\ref{mainthm:asymp} is substantially more complex than the analogous results for minimising solutions. 
Many of the arguments of~\cite{CanevariMajumdarStroffoliniWang, CDS3} do not carry over to this context, because they are based 
on energy minimality.
Instead, our proofs combine different elements. 
In order to obtain compactness for the~$\Q_\eps$-component, we adapt arguments from the Ginzburg-Landau literature, mostly from~\cite{BBH} and~\cite{BBO}, with some modifications due to the presence of coupling terms in equations~\eqref{EL-Q} and~\eqref{EL-M}. The same arguments also show that~$\mu_\eps$ concentrates on a finite number of singular points. To prove compactness for the~$\M_\eps$-component, we perform the same change of variables that is used in~\cite{CanevariMajumdarStroffoliniWang} (see Section~\ref{sec:u}), which enables us to apply compactness results for the vectorial Modica-Mortola problem~\cite{FonsecaTartar}.
 
The proof of the strong convergence $\Q_\eps \to \Q_\star$ 
in $W^{1,2}_{\rm loc}(\Omega \setminus \spt\mu_\star)$ is 
more delicate and, once again, the main source of difficulty is the coupling term. 
The heart of the matter consists in proving that, for any ball 
$B \csubset \Omega \setminus \spt\mu_\star$, there holds 
\begin{equation}\label{eq:crucial-lemma-intro}
	\int_B \left\{ \abs{\nabla\abs{\Q_\eps}}^2 + \left(\frac{\abs{\Q_\eps}-1}{\eps} - \kappa_\star\right)^2 \right\}\,{\d}x \to 0, \qquad \mbox{as } \eps \to 0,
\end{equation}
where $\kappa_\star = \frac{\beta}{2\sqrt{2}} \left( \sqrt{2}\beta + 1 \right)$ is 
a constant representing the ground level of the `modified' Ginzburg-Landau potential 
arising because of the interaction between the $\Q$ and $\M$ (i.e., the term in 
round brackets in the left-hand side of~\eqref{eq:crucial-lemma-intro}). 
The analogous statement in the classical Ginzburg-Landau theory, 
formally corresponding to $\beta = 0$, is well-known, see~\cite[Theorem~X.2]{BBH}. 
However, the argument of~\cite{BBH} does not work directly in our 
case, exactly because of the coupling term. The proof of~\eqref{eq:crucial-lemma-intro}, 
detailed in Lemma~\ref{lemma:con-loc-mod+pot} below,  
is therefore trickier than its counterpart~\cite[Equation~(127)]{BBH}. Once~\eqref{eq:crucial-lemma-intro} 
is proved, the rest of the argument in~\cite{BBH}, based on a suitable 
Hodge decomposition of the pre-Jacobian $j(\Q_\eps)$ of $\Q_\eps$, carries on  
with small modifications (see Proposition~\ref{prop:strong-conv-Qeps} 
for full details). 

The importance of~\eqref{eq:crucial-lemma-intro} is not limited to 
the proof of the locally strong convergence away from $\spt\mu_\star$. 
As detailed in the companion article~\cite{CDS2}, it is also a key 
ingredient in the asymptotic analysis of the energy measures $\nu_\eps$ as well as 
in the description of the support of the corresponding limiting measure and of its 
relationship with the support of $\mu_\star$.

\subsection*{Plan of the paper}
The paper is organised as follows. Section~\ref{sect:prelim} contains some notation and preliminary results on the space of~$\Q$-tensors, on the potential~$f_\eps$, and some first consequences of the equations~\eqref{EL-Q}--\eqref{EL-M} (such as the maximum principle, the Pohozaev identity, and a ``clearing-out'' property~\cite{BBH}).
Section~\ref{sect:compQ} is devoted to uniform estimates on both~$\Q_\eps$ and~$\M_\eps$. For instance, we prove uniform~$W^{1,p}(\Omega)$-estimates on~$\Q_\eps$, for~$p<2$, as well as uniform~$L^1(\Omega)$-estimates for~$\mu_\eps$ and~$\nu_\eps$. The main estimates obtained through the section are all summarised in
Theorem~\ref{lemma:energy-est}.
In Section~\ref{sect:compactness}, we will apply these estimates to prove compactness for both~$\Q_\eps$ and~$\M_\eps$. As a key tool, we use a change of coordinates,
already introduced in \cite{CanevariMajumdarStroffoliniWang}, allowing us to write the functional as an almost decoupled one: a Ginzburg-Landau type term,
a vectorial Allen-Cahn term, and a small error (vanishing with $\eps$). 
The proof of Theorem~\ref{mainthm:asymp} is presented in Section~\ref{sect:mainthm}.
Finally, Appendix~\ref{app:dir} 
completes the paper, with the proof of some technical results.

\begin{notation}
We use the symbol $\{X_\eps\}$ to denote a family of
objects indexed by the parameter $\eps > 0$, almost always
used as a shorthand to denote
the sequence $\left\{X_{\eps_k}\right\}_k$, where $\eps_k \to 0$ as
$k \to +\infty$.
Usually, for the sake of a lighter notation, we do not relabel subsequences.

In inequalities like $A \lesssim B$, the symbol $\lesssim$ means that there
exists a constant $C$, independent of $A$ and $B$, such that $A \leq C B$.
In particular,
dealing with sequences indexed by $\eps$, we use
$\lesssim$ to denote inequality up to a constant independent of $\eps$.
Whenever it is relevant, we keep track of the dependences
of the implicit constants. Whenever possible without inducing 
ambiguities, for the sake of a lighter notation, we avoid writing 
explicitly the measure of integration in the integrals.
\end{notation}

\tableofcontents

\section{Preliminary estimates}
\label{sect:prelim}

\setcounter{equation}{0}
\numberwithin{equation}{section}
\numberwithin{definition}{section}
\numberwithin{theorem}{section}
\numberwithin{remark}{section}
\numberwithin{example}{section}

\subsection{The space of~$\Q$-tensors}\label{sec:Q-tensors}

In the Introduction, we defined the space of $\Q$-tensors
$\Sz$ as
\[
	\Sz := \{ \Q \in \mathbb{M}_{2\times 2}(\R) \,:\, \Q = \Q^{\rm t},\,\, \tr(\Q) = 0 \}.
\]
It follows immediately that $\Sz$ is a linear space of dimension $2$ whose elements
are $2 \times 2$, real matrices of the form
\[
	\Q =
	\begin{pmatrix}
		Q_{11} & Q_{12} \\
		Q_{12} & -Q_{11}
	\end{pmatrix}.
\]
We endow $\Sz$ with the usual Frobenius norm of matrices, i.e.,
\[
	\abs{\Q} := \sqrt{\tr(\Q^{\rm t}\Q)}, \qquad
	\mbox{for any } \Q \in \Sz.
\]
The map
\begin{equation}\label{eq:small-q}
	\Sz \ni \Q \mapsto \q := \sqrt{2}\left( Q_{11},\, Q_{12} \right) \in \R^2
\end{equation}
establishes an isometric isomorphism between $\Sz$ and $\R^2$
(the latter provided with the standard scalar product).
Of course, we can also identify any $\Q \in \Sz$ with the complex number
$Q_{11} + i Q_{12}$, and thus we have the further identification
$\Sz \simeq \C$.

We denote as $\NN$ the unit sphere in $\Sz$, i.e.,
\[
	\NN := \left\{ \Q \in \Sz \,:\, \abs{\Q} = 1 \right\},
\]
which can be identified with the unit circle $\SS^1$ in $\R^2$
(or, equivalently, in $\C$). Moreover, in view of~\eqref{eq:small-q},
we may also identify, equivalently,
\[
	\NN = \left\{ \sqrt{2}\left( \n \otimes \n - \frac{\I}{2} \right) \,:\, \n \in \SS^1 \right\}.
\]

\paragraph*{Representing $\Q$-tensors in ``polar coordinates''.}
Let~$G\subseteq\Omega$ be a simply connected,
smooth subdomain, and let~$\Q\colon G\to\Sz$
be a smooth map such that~$\abs{\Q} > 0$ on~$G$.
The spectral theorem implies that~$\Q$
can be written in the form~$\sqrt{2}\Q = \abs{\Q} (\n\otimes\n- \m\otimes\m)$,
where~$(\n, \, \m)$ is an orthonormal basis of eigenvectors of~$\Q$.
Both~$\n$ and~$\m$ are uniquely determined by~$\Q$,
up to their sign. (In other words, either $\m = \n^\perp$ or
$\m = - \n^\perp$.)
Similarly, $\n$ can be written as~$\n = (\cos\varphi, \, \sin\varphi)$
for some scalar function~$\varphi\colon G\to\R$.
As a consequence, setting~$\rho:=\abs{\Q}$, we can write
\begin{equation} \label{Qpolar}
 \Q = \frac{\rho}{\sqrt{2}} \left(
  \begin{matrix}
   \cos(2\varphi) & \phantom{-}\sin(2\varphi) \\
   \sin(2\varphi) & -\cos(2\varphi)
  \end{matrix}\right)
\end{equation}
The function~$\varphi$ is uniquely determined by~$\Q$,
up to a constant multiple of~$\pi$. An explicit computation
shows that
\begin{equation} \label{gradQpolar}
 \abs{\nabla\Q}^2 = \abs{\nabla\rho}^2 + 4 \rho^2 \abs{\nabla\varphi}^2
\end{equation}
Moreover, if we only assume that~$\Q$ is bounded
away from zero on~$\partial G$, but not necessarily
in the interior of~$G$, we can still write~$\Q$
as in~\eqref{Qpolar}, except that the function~$\varphi$
can be discontinuous. However, we can always choose~$\varphi$
to be continuous on~$\partial G$, except for one
jump discontinuity at most (see, e.g., \cite[Proof of Lemma~1.3 and the remarks thereafter]{BrezisMironescu}). If~$x_0\in\partial G$
is the unique jump point of~$G$
and~$\varphi(x_0^+)$, $\varphi(x_0^-)$ are the
traces of~$\varphi$ on either side of~$x_0$
(oriented in the anticlockwise direction), then we define

\begin{equation} \label{degreephi}
 \deg(\Q, \, \partial G)
 := \frac{1}{2\pi}\left(\varphi(x_0^+) - \varphi(x_0^-)\right)
\end{equation}
The number~$\deg(\Q, \, \partial G)$ is the topological
degree of~$\Q$ on~$\partial G$, and is a half-integer.
If~$\Q$ admits a continuous lifting on~$\partial G$,
then~$\deg(\Q, \, \partial G) = 0$ (and vice-versa).
Similar considerations apply when~$\Q\in W^{1,2}(\Omega, \, \Sz)$,
so long as~$\rho$ is bounded away from zero
(by lifting results for $\Q$-tensors --- see, e.g., \cite{BallZarnescu, BethuelChiron} --- and
degree theory in $W^{1/2,2}$ --- see, e.g., \cite[Proposition~12.1]{BrezisMironescu}).

\paragraph*{Pre-Jacobians.}
We recall some notation from~\cite{CanevariMajumdarStroffoliniWang}.
We define the vector product of
two matrices~$\Q\in\Sz$, $\P\in\Sz$ as
\begin{equation} \label{cross}
 \Q\times\P :=  Q_{11} P_{12} - Q_{12} P_{11} + Q_{21} P_{22} - Q_{22} P_{21}
 = 2\left(Q_{11} P_{12} - Q_{12} P_{11}\right)
\end{equation}
For any~$\Q\in (L^\infty\cap W^{1,1})\left(\Omega, \, \Sz\right)$,
we define the vector field~$j(\Q)\colon\Omega\to\R^2$ as
\begin{equation} \label{preJac}
 j(\Q) := \frac{1}{2}\left(\Q\times\partial_1\Q, \,
  \Q\times\partial_2\Q \right) .
\end{equation}
For notational convenience, we shall often write~\eqref{preJac}
in the form
\begin{equation}\label{preJac-bis}
	j(\Q) = \frac{1}{2} \Q \times \nabla \Q.
\end{equation}
By analogy with the Ginzburg-Landau literature, $j(\Q)$
could be called the ``pre-Jacobian'' of~$\Q$.
This terminology is justified because, for smooth maps~$\Q$,
we have
\begin{equation} \label{Jac}
 \frac{1}{2} \curl j(\Q)
 = \partial_1 Q_{11} \, \partial_2 Q_{12} - \partial_2 Q_{11} \, \partial_1 Q_{12}
\end{equation}
(and the same remains true if~$\Q\in W^{1,2}(\Omega, \, \Sz)$,
by a density argument).
In particular, for any $\Q$-tensor field
$\Q \in \left(L^\infty \cap W^{1,1}\right)\left(\Omega,\,\Sz\right)$
it makes sense to define
\begin{equation}\label{eq:Jac-bis}
	J(\Q) := \curl j(\Q),
\end{equation}
if we agree that the derivatives are taken in the sense of distributions.
The distribution $J(\Q)$ is called the \emph{distributional Jacobian}
of $\Q$.

If~$\Q$ can be represented in the form~\eqref{Qpolar}, then
\begin{equation} \label{preJacpolar}
 j(\Q) = \rho^2 \, \nabla\varphi.
\end{equation}
As a consequence, for any smooth~$\Q\colon\Omega\to\Sz$, there holds
\begin{equation} \label{preJacnabla}
 \abs{\Q}^2 \abs{\nabla\Q}^2
   = \abs{\Q}^2 \abs{\nabla(\abs{\Q})}^2 + 4\abs{j(\Q)}^2
\end{equation}
The equality~\eqref{preJacnabla} holds at any point
of~$\{\Q\neq 0\}$ because of~\eqref{gradQpolar}, \eqref{preJacpolar},
and is satisfied trivially at almost any point of~$\{\Q = 0\}$.
Another consequence of~\eqref{preJacpolar} is the equality
\begin{equation} \label{preJaccurl}
 \curl\left(\frac{j(\Q)}{\abs{\Q}^2}\right) = 0,
\end{equation}
which holds at any point~$x\in\Omega$ such that~$\Q(x)\neq 0$.
Moreover, if~$G\subseteq\Omega$ is a simply connected,
smooth subdomain such that~$\abs{\Q} > 0$ on~$\partial G$
and~$\ttau$ is the unit tangent vector to~$\partial G$
oriented anti-clockwise direction, then
\begin{equation} \label{preJacdegree}
 \int_{\partial G} \frac{j(\Q)\cdot\ttau}{\abs{\Q}^2} \,\d s
  = \pi \deg(\Q, \, \partial G)
\end{equation}
because of~\eqref{degreephi} and~\eqref{preJacpolar}.

\begin{remark}\label{rk:continuity-Jac}
	The distributional pre-Jacobian and 
	Jacobian are continuous with respect to the
	weak $W^{1,p}$-convergence, for any $p > 1$. In other words,
	given a sequence of uniformly bounded $\Q$-tensor fields
	$\Q_\eps \in W^{1,p}\left(\Omega,\,\Sz\right)$
	which converges weakly in $W^{1,p}\left(\Omega,\,\Sz\right)$ to
	some $\Q$-tensor field $\Q_\star$ as $\eps \to 0$, then
	$j(\Q_\eps) \to j(\Q_\star)$ and
	$J(\Q_\eps) \to J(\Q_\star)$ as $\eps \to 0$ in the sense of
	distributions in $\Omega$. (Cf., e.g., \cite[Remark~{3.6}(iii)]{Alberti}.)
\end{remark}

\begin{notation}
	Besides the pre-Jacobian of $\Q$-tensor fields,
	occasionally we shall also consider the pre-Jacobian of a vector
	field $\u \in \left(L^\infty \cap W^{1,1}\right)(\Omega, \,\R^2)$.
	Denoted $j(\u)$, such a pre-Jacobian can be identified with the
	vector field with components
	\begin{equation}\label{eq:def-prejac-vect}
		\left(j(\u)\right)_k = \left(u_1 \partial_k u_2 - u_2 \partial_k u_1\right)_k,
		\qquad \mbox{for } k \in \{1,\,2\},
	\end{equation}
	i.e.,
	\[
		j(\u) = \u \times \nabla \u,
	\]
	exactly as in the standard Ginzburg-Landau theory.
\end{notation}

\subsection{Properties of the potential~$f_\eps$.}

We summarise below some properties of the potential~$f_\eps$.
As in \cite[Equation~{(3.1)}]{CanevariMajumdarStroffoliniWang},
we define
\begin{equation}\label{eq:k*}
	\kappa_\star := \frac{1}{2 \sqrt{2}}\beta\left( \sqrt{2}\beta + 1 \right).
\end{equation}

\begin{lemma}[{\cite[Lemma~B.3]{CanevariMajumdarStroffoliniWang}}]
\label{lemma:feps}
 The potential~$f_\eps$ satisfies the following properties.
 \begin{enumerate}[label=(\roman*)]
  \item The constant~$\kappa_\eps$ in~\eqref{eq:potential},
  uniquely defined by imposing
  the condition~$\inf f_\eps = 0$, satisfies
  \[
   \kappa_\eps = \frac{1}{2} \left(\beta^2 + \sqrt{2} \beta\right) \eps
   + \kappa_\star^2 \eps^2 + \o(\eps^2)
  \]
  as $\eps \to 0$.

  \item\label{item:good-data}
  If~$(\Q, \, \M)\in\Sz\times\R^2$ is such that
  \[
   \abs{\M} = \left(\sqrt{2}\beta + 1\right)^{1/2},
  \qquad \Q = \sqrt{2}\left(\frac{\M\otimes\M}{\sqrt{2}\beta + 1}
   - \frac{\I}{2}\right)
  \]
  then $f_\eps(\Q, \, \M) = \kappa_\star^2 \, \eps^2 + \o(\eps^2)$
  as~$\eps\to 0$.

  \item If~$\eps$ is sufficiently small, then
  for any~$(\Q, \, \M)\in\Sz\times\R^2$ we have
  \begin{equation} \label{potential_comparison}
   \frac{1}{\eps^2} f_\eps(\Q, \, \M)
    \geq \frac{1}{8\eps^2}\left(\abs{\Q}^2 - 1\right)^2
    - \beta^2\abs{\M}^4 \! .
  \end{equation}
 \end{enumerate}
\end{lemma}

For any~$(\Q, \, \M)\in\Sz\times\R^2$, we can write
\begin{equation} \label{f-ell}
 \frac{1}{\eps^2} f_\eps(\Q, \, \M)
  = \frac{1}{4\eps^2}\left(\abs{\Q}^2 - 1\right)^2
   + \frac{1}{\eps} \ell(\Q, \, \M) + \chi_\eps,
\end{equation}
where
\begin{equation} \label{ell_eps}
 \ell(\Q, \, \M) := \frac{1}{4}\left(\abs{\M}^2 - 1\right)^2
  - \beta \, \Q \M \cdot \M
  + \frac{1}{2} \left(\beta^2 + \sqrt{2} \beta\right)
\end{equation}
and
\begin{equation} \label{chi_eps}
 \chi_\eps := \frac{\kappa_\eps}{\eps^2}
  - \frac{1}{2\eps} \left(\beta^2 + \sqrt{2} \beta\right) \! .
\end{equation}
Note that
\begin{equation}\label{eq:chi-to-k*}
	\chi_\eps \to \kappa_\star^2 \qquad \mbox{as } \eps \to 0.
\end{equation}
In the next lemma, we consider minimisers of the
function~$\ell(\Q, \, \cdot\,)$, for a given~$\Q$.
We recall that any~$\Q\in\Sz\setminus\{0\}$
has exactly one positive and one negative eigenvalues,
because of the constraint~$\tr\Q = 0$.

\begin{lemma} \label{lemma:f-fixedQ}
 For any~$\Q\in\Sz\setminus\{0\}$, the function~$\ell(\Q, \, \cdot\,)$
 has exactly two minimisers, given by
 \begin{equation} \label{minimisers_l}
  \M_{\pm} := \pm\left(\sqrt{2}\beta\rho + 1\right)^{1/2} \n
 \end{equation}
 where~$\rho:= \abs{\Q}$ and~$\n$ is 
 a unit eigenvector of~$\Q$ corresponding to its positive eigenvalue.
 Moreover, there holds
 \begin{equation} \label{ell-minimum}
  \begin{split}
   \min\ell(\Q, \, \cdot\,)
   = \ell(\Q, \, \M_{\pm})
   = \frac{\beta}{2} (1 - \rho)\left(\sqrt{2} + \beta + \beta\rho\right) \! .
  \end{split}
 \end{equation}
\end{lemma}

\begin{proof}
 The gradient and the Hessian of~$\ell$ with respect to
 the~$\M$-variable are given by
 \begin{align*}
  \nabla_{\M} \ell(\Q, \, \M)
   &= \left(\abs{\M}^2 - 1\right) \M
   - 2\beta \Q\M \\
  \D^2_{\M} \ell(\Q, \, \M)
   &= \left(\abs{\M}^2 - 1\right) \I
   + 2\M\otimes\M - 2\beta\Q.
 \end{align*}
 We immediately see that any critical point of~$\ell(\Q, \, \cdot\,)$
 must be an eigenvector of~$\Q$.
 Since~$\tr\Q = 0$ and~$\tr(\Q^2) = \rho^2$,
 the eigenvalues of~$\Q$ must be equal to~$\pm\rho/\sqrt{2}$.
 Let~$(\n, \, \m)$ be an orthonormal basis
 of eigenvectors of~$\Q$, such that~$\n$ is associated with
 the positive eigenvalue of~$\Q$ and~$\m$ is associated with
 the negative one. Then, substituting~$\M = \lambda\,\n$,
 $\M = \mu\,\m$ in the condition~$\nabla_{\M} \ell = 0$, we find
 that critical points (with respect to the~$\M$-variable only)
 occur for
 \[
  \M = 0, \qquad \M = \M_{\pm}
  := \pm\left(1 + \sqrt{2}\beta\rho\right)^{1/2} \n, \qquad
  \M = \widetilde{\M}_{\pm} := \pm\left(1 - \sqrt{2}\beta\rho\right)^{1/2} \m
 \]
 --- the latter pair, $\widetilde{\M}_{\pm}$,
 being defined only when~$\sqrt{2}\beta\rho \leq 1$.
 However, $\D^2_{\M} \ell(\Q, \, 0) = -\I - 2\beta\Q$ cannot
 be positive semi-definite, because its trace is equal to~$-2$.
 Moreover, we have
 $\ell(\Q, \, \widetilde{\M}_+) = \ell(\Q, \, \widetilde{\M}_-)$,
 $\ell(\Q, \, \M_+) = \ell(\Q, \, \M_-)$ and
 \[
  \ell(\Q, \, \widetilde{\M}_\pm) - \ell(\Q, \, \M_\pm)
  = \sqrt{2}\beta \rho > 0.
 \]
 Therefore, the minimisers of~$\ell(\Q, \, \cdot\,)$ are exactly those given
 in~\eqref{minimisers_l}, and~\eqref{ell-minimum} follows
 by an explicit computation.
\end{proof}

\begin{corollary}\label{cor:ell-rephrased}
	For any~$\Q\in\Sz\setminus\{0\}$ and any $\M \in \R^2$,
	there holds
	\begin{equation}\label{eq:ell-rephrased}
		\ell(\Q,\,\M) - \ell(\Q,\,\M_\pm) = \frac{1}{4}\abs{\M - \M_+}^2\abs{\M-\M_-}^2
		+ \sqrt{2}\beta \rho(\M \cdot \m)^2,
	\end{equation}
	where $(\n,\,\m)$ is an eigenframe for $\Q$ and $\n$ is a unit
	eigenvector of $\Q$ corresponding to its positive eigenvalue.
\end{corollary}

\begin{proof}
	Since $\rho = \abs{\Q} > 0$, we can write
	$\Q = \frac{\rho}{\sqrt{2}} (\n \otimes \n - \m \otimes \m)$, where
	$\n$ is an eigenvector of $\Q$ corresponding to its positive eigenvalue
	and $\m$ is orthogonal to $\n$. Furthermore, we can decompose $\M$
	along $(\n,\,\m)$, i.e, we can write
	\[
		\M = (\M \cdot \n) \n + (\M \cdot \m) \m,
	\]
	and it follows that
	\[
		\Q \M \cdot \M = \frac{\rho}{\sqrt{2}}\left( (\M \cdot \n)^2 - (\M \cdot \m)^2 \right)
	\]
	On the other hand, we may obviously rephrase
	\[
		\left( \abs{\M}^2 - 1 \right)^2 = \left( \abs{\M}^2 - \left(1 + \sqrt{2}\beta \rho\right) \right)^2 + 2 \sqrt{2} \beta \rho \left(\abs{\M}^2 - 1 -\sqrt{2}\beta\rho\right) + 2 \beta^2 \rho^2.
	\]
	Combining the above formulas with the definition~\eqref{ell_eps} of $\ell$
	and with~\eqref{ell-minimum}, we obtain
	\[
	\begin{split}
		\ell(\Q,\,\M) &= \frac{1}{4}\left( \abs{\M}^2 - \left(1+\sqrt{2}\beta \rho\right) \right)^2 + \frac{\sqrt{2}}{2}\beta \rho\left( 2(\M \cdot \m)^2 - \left(1+\sqrt{2}\beta \rho\right) \right) \\
		&\quad+ \frac{1}{2}\beta^2\rho^2 + \frac{1}{2}\left( \beta^2 + \sqrt{2}\beta \right)\\
		&= \ell(\Q,\,\M_\pm) + \frac{1}{4}\left( \abs{\M}^2 - \left(1+\sqrt{2}\beta\rho\right) \right)^2 + \sqrt{2}\beta \rho(\M \cdot \m)^2.
	\end{split}
	\]
	After rearrangement, the conclusion follows by noticing that,
	in view of~\eqref{minimisers_l},
	\[
		\left(\abs{\M}^2 - \left(1+\sqrt{2}\beta\rho\right)\right)^2
		= \abs{\M - \M_+}^2 \abs{\M - \M_-}^2.
	\]
\end{proof}

For~$\Q\in\Sz\setminus\{0\}$, let~$\Sigma(\Q)
:= \{\M_+, \, \M_-\}\subseteq\R^2$ be the set of minimisers
of~$\ell(\Q, \, \cdot\,)$, as given by~\eqref{minimisers_l}.
If~$\Q\neq 0$ and~$\dist(\M, \, \Sigma(\Q)) \leq 1$, we define
\begin{equation} \label{projection}
 \pi(\Q, \, \M) :=
 \begin{cases}
  \M_+ &\textrm{if } \abs{\M - \M_+} \leq 1 \\
  \M_- &\textrm{if } \abs{\M - \M_-} \leq 1.
 \end{cases}
\end{equation}
$\pi(\Q, \, \M)$ is the closest projection of~$\M$
to~$\Sigma(\Q)$. It is defined in a non-ambiguous way,
because~$\abs{\M_+ - \M_-} = 2(\sqrt{2}\beta\rho + 1)^{1/2} > 2$.

\begin{remark}\label{rk:lower-bound-distance}
	Let $\Q \in \Sz \setminus \{0\}$ be arbitrary.
	As an obvious consequence of~\eqref{eq:ell-rephrased},
	and since $\abs{\M_+ - \M_-} > 2$, letting $\N = \pi(\Q,\,\M)$,
	for any $\M \in \R^2$ such that $\dist(\M,\,\Sigma(\Q)) \leq 1$,
	there holds
	\begin{equation}\label{eq:lower-bound-distance}
		\ell(\Q,\,\M) - \ell(\Q,\,\N) \geq \frac{1}{4} \abs{\M - \N}^2.
	\end{equation}
\end{remark}

\begin{lemma} \label{lemma:ell2}
 For any~$\Lambda > 1$, there exist positive
 constants~$\delta_0(\Lambda)$ and~$c_0(\Lambda)$
 (depending on~$\Lambda$ and~$\beta$ only)
 such that the following statement holds:
 for any~$(\Q, \, \M)\in\Sz\times\R^2$ such
 that~$\Lambda^{-1} \leq \abs{\Q} \leq \Lambda$
 and~$\dist(\M, \, \Sigma(\Q)) \leq \delta_0(\Lambda)$, there holds
 \[
  \nabla_{\M} \ell(\Q, \, \M)\cdot(\M - \pi(\Q, \, \M))
   \geq c_0(\Lambda)\Big(\ell(\Q, \, \M) - \ell(\Q, \, \pi(\Q, \, \M))\Big)
   \geq 0.
 \]
\end{lemma}
\begin{proof}
 Let~$X(\Lambda)\subseteq\Sz\times\R^2$ be the set of pairs~$(\Q, \, \M)$
 such that~$\Lambda^{-1}\leq \abs{\Q} \leq \Lambda$
 and $\dist(\M, \, \Sigma(\Q)) \leq 1$.
 Let~$(\Q, \, \M)\in X(\Lambda)$ and~$\N := \pi(\Q, \, \M)$.
 Since $\abs{\Q} > 0$, we may write~$\sqrt{2}\Q = \rho(\n\otimes\n-\m\otimes\m)$,
 where~$\rho := \abs{\Q}$ and~$(\n, \, \m)$
 is an orthonormal set of eigenvectors, with $\n$ corresponding to the positive
 eigenvalue of $\Q$.
 In view of Corollary~\ref{cor:ell-rephrased}, of the immediate inequality
 $\abs{\M - \N}^2 \geq (\M \cdot \m)^2$, and of the
 assumption $\dist(\M,\,\Sigma(\Q)) \leq 1$, we obtain
 \begin{equation} \label{ell21}
  0 \leq \ell(\Q, \, \M) - \ell(\Q, \, \N)
   \leq L(\Lambda) \abs{\M - \N}^2,
 \end{equation}
 where the first inequality holds because $\N$ is a minimiser of $\ell(\Q,\,\cdot\,)$,
 and~$L(\Lambda)$ can be taken of the form
 \begin{equation}\label{ell22}
 	L(\Lambda) = \sqrt{2}\beta \Lambda + C,
 \end{equation}
 form some $C$  that does not depend on~$\Lambda$, nor on~$\beta$.

 On the other hand, by Taylor-expanding~$\nabla_{\M}\ell$
 around the point~$(\Q, \, \N)$, we obtain
 \begin{equation*}
  \begin{split}
   \nabla_{\M} \ell(\Q, \, \M)\cdot(\M - \N)
   &\geq \D^2_{\M} \ell(\Q, \, \N)(\M - \N)\cdot(\M - \N)
    - L^\prime\abs{\M - \N}^3,
  \end{split}
 \end{equation*}
 where~$L^\prime$ is an upper bound for the third
 derivatives of~$\ell(\Q, \, \cdot\,)$ on~$X(\Lambda)$.
 (The third derivatives of~$\ell(\Q, \, \cdot)$
 do not depend on~$\Q$, so we can choose~$L^\prime$
 to be independent of~$\Lambda$.)
 We evaluate explicitly the second
 derivatives of~$\ell$.
 Recalling~\eqref{minimisers_l}, we obtain
 \begin{equation}\label{eq:hess-ell-N}
  \begin{split}
   \D^2_{\M}\ell(\Q, \, \N)
   = \sqrt{2}\beta\rho \, \I
   + (\sqrt{2}\beta\rho + 2)\n\otimes\n + \sqrt{2}\beta\rho\,\m\otimes\m
   = 2\sqrt{2}\beta\rho \, \I + 2\n\otimes\n
  \end{split}
 \end{equation}
 and hence,
 \begin{equation} \label{ell23}
  \begin{split}
   \D^2_{\M} \ell(\Q, \, \N)(\M - \N)\cdot(\M - \N)
   \geq 2\sqrt{2}\,\beta \rho \abs{\M - \N}^2
   \geq \frac{2\sqrt{2}\,\beta}{\Lambda} \abs{\M - \N}^2
  \end{split}
 \end{equation}
 for any~$(\Q, \, \M)\in X(\Lambda)$.
 Combining~\eqref{ell21}, \eqref{ell22}, and~\eqref{ell23},
 and choosing, for instance, $\delta_0(\Lambda) := \frac{\sqrt{2} \beta}{\Lambda L'}$,
 the lemma follows.
\end{proof}

\begin{lemma} \label{lemma:ell1}
 For any~$\delta > 0$ there exist positive
 constants~$\eps_1(\delta) > 0$ and~$c_1(\delta)>0$
 (depending on~$\beta$ and~$\delta$, but not on~$\eps$)
 such that the following statement holds:
 for any~$\eps\in (0, \, \eps_1(\delta)]$ and
 any~$(\Q, \, \M)\in\Sz\times\R^2$ such that $\abs{\Q} \leq 3/4$
 or~$\dist(\M, \, \Sigma(\Q)) \geq \delta$, there holds
 \begin{equation} \label{ell1}
  f_\eps(\Q, \, \M) \geq c_1(\delta)\,\eps.
 \end{equation}
\end{lemma}
\begin{proof}
 Let~$\delta > 0$ be fixed.
 Should the lemma be false, there would exist sequences
 $(\eps_k)_{k\in\mathbb{N}}$, $(\Q_k)_{k\in\mathbb{N}}$, $(\M_k)_{k\in\mathbb{N}}$
 such that~$\eps_k\to 0$ as~$k\to+\infty$,
 \begin{equation} \label{ell11}
  \abs{\Q_k} \leq \frac{3}{4} \quad\textrm{or}\quad
  \dist(\M_k, \, \Sigma(\Q_k)) \geq \delta
  \qquad \textrm{for any } k\in\mathbb{N},
 \end{equation}
 and
 \begin{equation} \label{ell12}
  \frac{f_{\eps_k}(\Q_k, \, \M_k)}{\eps_k}
  = \frac{1}{4\eps_k}\left(\abs{\Q_k}^2 - 1\right)^2
   + \ell(\Q_k, \, \M_k) + \eps_k\,\chi_{\eps_k} \to 0
  \qquad \textrm{as } k\to+\infty.
 \end{equation}
 We can give an estimate for~$\abs{\Q_k}$
 by combining~\eqref{ell12} with Equation~\eqref{ell-minimum},
 which provides a lower bound for~$\ell_{\eps_k}(\Q_k, \, \M_k)$,
 and~Lemma~\ref{lemma:feps}, which shows that
 the sequence~$(\chi_{\eps_k})_{k\in\mathbb{N}}$
 is bounded. We obtain
 \begin{equation} \label{ell13}
  \begin{split}
    \frac{1}{4\eps_k}\left(\abs{\Q_k}^2 - 1\right)^2
    &\leq -\min\ell(\Q_k, \, \cdot) + \o_{k\to+\infty}(1) \\
    &\leq \frac{\beta}{2} \left(\abs{\Q_k} - 1\right)
     \left(\sqrt{2} + \beta + \beta\abs{\Q_k}\right)
     + \o_{k\to+\infty}(1).
  \end{split}
 \end{equation}
 This inequality implies that the sequence~$(\Q_k)_{k\in\mathbb{N}}$
 is bounded. Then, from~\eqref{ell12} we deduce
 \begin{equation*}
  \begin{split}
   \limsup_{k\to+\infty} \left(\frac{1}{4}\left(\abs{\M_k}^2 - 1\right)^2
    - \beta \, C\abs{\M_k}^2\right)
    \leq \limsup_{k\to+\infty} \ell(\Q_k, \, \M_k)
    \leq 0,
  \end{split}
 \end{equation*}
 where~$C$ is an upper bound for~$\abs{\Q_k}$.
 Therefore, $(\M_k)_{k\in\mathbb{N}}$, too, is bounded.
 Up to extraction of a subsequence, we can assume that~$\Q_k\to\Q$,
 $\M_k\to\M$. By passing to the limit in~\eqref{ell13}
 and~\eqref{ell12}, we see that~$\abs{\Q} = 1$ and
 $\ell(\Q, \, \M) \leq 0$.
 On the other hand,~\eqref{ell-minimum}
 implies that~$\ell(\P, \, \N)\geq 0$ for any~$(\P, \, \N)$
 such that~$\abs{\P} = 1$, with strict inequality unless~$\N\in\Sigma(\P)$.
 Therefore, we must have~$\abs{\Q} = 1$ and~$\M\in\Sigma(\Q)$,
 but this contradicts~\eqref{ell11}. The lemma follows.
\end{proof}

\subsection{First consequences of the Euler-Lagrange equations}

\paragraph*{The maximum principle.}

From now on, we denote by~$(\Q_\eps, \, \M_\eps)$
a solution to~\eqref{EL-Q}--\eqref{EL-M},
subject to either Dirichlet or mixed boundary conditions,
as in~\eqref{bc}--\eqref{hp:bc} or~\eqref{bcbis}--\eqref{hp:bcbis}
respectively. Either way, we assume
that~\eqref{hp:potential_bound} is satisfied.
We will denote the energy density as
\begin{equation} \label{energydensity}
 e_\eps(\Q, \, \M) := \frac{1}{2}\abs{\nabla\Q}^2
 + \frac{\eps}{2}\abs{\nabla\M}^2 + \frac{1}{\eps^2} f_\eps(\Q, \, \M)
\end{equation}
for~$(\Q, \, \M)\in W^{1,2}(\Omega, \, \Sz)\times W^{1,2}(\Omega, \, \R^2)$.

\begin{lemma} \label{lemma:max}
 The maps~$\Q^*_\eps$, $\M^*_\eps$ are smooth inside~$\Omega$
 and of class~$C^1$ up to the boundary of~$\Omega$.
 Moreover, there exist 
 a constant~$C_\beta$, depending only on $\beta$,
 and a constant $C_{\beta,\Omega}$, depending only on $\beta$ and $\Omega$,
 such that
 \begin{align}
  \norm{\Q^*_\eps}_{L^\infty(\Omega)}
   + \norm{\M^*_\eps}_{L^\infty(\Omega)} &\leq {C_\beta} \label{max-QM} \\
  \norm{\nabla\Q^*_\eps}_{L^\infty(\Omega)}
   + \norm{\nabla\M^*_\eps}_{L^\infty(\Omega)}
   &\leq \frac{{C_{\beta,\Omega}}}{\eps}. \label{max-gradients}
 \end{align}
\end{lemma}
\begin{proof}[Proof of Lemma~\ref{lemma:max}]
 Elliptic regularity theory (both for the Dirichlet and the Neumann
 problems) implies that~$\Q_\eps$ and~$\M_\eps$ are smooth in~$\Omega$
 and of class~$C^1$ up to~$\partial\Omega$.
 We focus on the proof of~\eqref{max-QM}; once~\eqref{max-QM} is proved,
 the gradient bounds~\eqref{max-gradients} will follow
 from elliptic estimates, by reasoning along the lines
 of~\cite[Lemma~A.2]{BBH0}.
 The bounds~\eqref{max-QM} are a consequence of the maximum principle.
 For the problem with Dirichlet boundary conditions,
 the details can be found in~\cite[Lemma~4.1]{CanevariMajumdarStroffoliniWang}.
 Here, we focus on the problem with mixed boundary conditions.

 As an intermediate step towards~\eqref{max-QM}, we show that
 \begin{equation} \label{max1}
  \max_{\overline{\Omega}} \abs{\M_\eps}^2
  \leq m_\eps := 1 + {\sqrt{2}}\beta \max_{\overline{\Omega}}\abs{\Q_\eps}
 \end{equation}
 Indeed, by taking the scalar product of~\eqref{EL-M}
 against~$\M_\eps$, we obtain
 \begin{equation*}
  -\frac{1}{2}\Delta\left(\abs{\M_\eps}^2\right)
  + \abs{\nabla\M_\eps}^2
  + \dfrac{1}{\eps^2}(\abs{\M_\eps}^2 - 1)\abs{\M_\eps}^2
  - \dfrac{2\beta}{\eps^2}\Q_\eps\M_\eps\cdot\M_\eps = 0.
 \end{equation*}
 and hence {(observing that $\Q \M \cdot \M \leq \frac{1}{\sqrt{2}} \abs{\Q} \abs{\M}^2$ for any $\Q \in \Sz$ and $\M \in \R^2$)}
 \begin{equation} \label{max2}
  -\frac{1}{2}\Delta\left(\abs{\M_\eps}^2\right)
   \leq - \dfrac{\abs{\M_\eps}^2}{\eps^2}
   \left(\abs{\M_\eps}^2 - 1 - {\sqrt{2}}\beta\abs{\Q_\eps}\right)
 \end{equation}
 at every point of~$\Omega$.
 Now, let~$x_0\in\overline{\Omega}$
 be a maximum point for~$\abs{\M_\eps}^2$
 and suppose, towards a contradiction,
 that~$\abs{\M_\eps(x_0)} > m_\eps$.
 Then, the right-hand side of~\eqref{max2} would attain
 a strictly negative value at the point~$x_0$
 and, by continuity of~$(\Q_\eps, \, \M_\eps)$,
 it would be strictly negative in a (smooth)
 neighbourhood~$V\subseteq\overline{\Omega}$ of~$x_0$.
 If~$x_0$ lies in the interior of~$\Omega$,
 we obtain a contradiction
 because~$-\Delta(\abs{\M_\eps}^2)(x_0) \geq 0$.
 On the other hand, if $x_0\in\partial\Omega$,
 then Hopf's lemma (applied on~$V$) implies
 that~$\partial_\nnu(\abs{\M_\eps}^2)(x_0) > 0$,
 which contradicts the boundary conditions~\eqref{bc}.
 Therefore, \eqref{max1} is proved.

 Now, \eqref{max-QM} follows by~\eqref{max1}
 by the maximum principle. Indeed, by taking the scalar
 product of~\eqref{EL-Q} against~$\Q_\eps$, we deduce
 \begin{equation} \label{maxQ}
  -\frac{1}{2}\Delta\left(\abs{\Q_\eps}^2\right)
  \leq -\dfrac{1}{\eps^2}\left(\abs{\Q_\eps}^4 - \abs{\Q_\eps}^2
  - {\frac{\beta}{\sqrt{2}}}\,\eps\abs{\Q_\eps}\abs{\M_\eps}^2\right)
 \end{equation}
 inside~$\Omega$. Let~$x_1\in\overline{\Omega}$
 be a maximum point for~$\abs{\Q_\eps}^2$.
 If~$x_1\in\partial\Omega$, then $\abs{\Q_\eps(x_1)}^2 = 1$
 because of the assumption~\eqref{hp:bc} on the boundary datum.
 If~$x_1\in\Omega$, then~\eqref{maxQ} and~\eqref{max1}
 imply that
 \begin{equation} \label{max4}
  \abs{\Q_\eps(x_1)}^3
  \leq \abs{\Q_\eps(x_1)} + {\frac{\beta}{\sqrt{2}}}\,\eps\abs{\M_\eps(x_1)}^2
  \leq \left(1 + {\beta^2\,\eps}\right)\abs{\Q_\eps(x_1)} + {\frac{\beta}{\sqrt{2}}}\,\eps
 \end{equation}
 and hence, $\abs{\Q_\eps}$ is uniformly bounded, in terms of~$\beta$ only.
\end{proof}

\begin{remark} \label{rk:max}
 The arguments above apply, with minor modifications, to the problem
 with Dirichlet boundary conditions.
 In both cases, from~\eqref{max1} and~\eqref{max4} we obtain
 \begin{equation}\label{eq:Q-M-s*}
  \abs{\Q_\eps} \leq s_*(\beta, \, \eps), \quad
  \abs{\M_\eps}^2 \leq 1 + {{\sqrt 2} }\beta \, s_*(\beta, \, \eps)
  \qquad \textrm{in } \Omega,
 \end{equation}
 where~$s_*(\beta, \, \eps) > 1$ is the largest root of the polynomial
 $P(X) = X^3 - (1 + {\beta^2\eps})X - {\frac{\beta}{\sqrt{2}}} \eps$.
 Elementary calculus shows that~$s_*(\beta, \, \eps)\to 1$
  as~$\eps\to 0$, for any given value of~$\beta$. Then, by differentiating
  the constraint $P(s_*(\beta, \, \eps)) = 0$ with respect to~$\eps$,
  we deduce
  \begin{equation}\label{eq:s_*}
   s_*(\beta, \, \eps) = 1 + {\eps \kappa_\star + {\rm O}(\eps^2)}
  \end{equation}
  as~$\eps\to 0$.
\end{remark}

\begin{remark} \label{rk:GLbound}
 Lemma~\ref{lemma:max} does not depend on
 the assumption~\eqref{hp:potential_bound}.
 However, combining~\eqref{hp:potential_bound}
 with the $L^\infty$-estimate~\eqref{max-QM}
 and the inequality~\eqref{potential_comparison},
 we obtain the bound
 \begin{equation} \label{GLbound}
  \frac{1}{\eps^2} \int_{\Omega} (\abs{\Q_\eps}^2 - 1)^2
   \lesssim \int_{\Omega}  \left(\frac{1}{\eps^2} f_\eps(\Q_\eps, \, \M_\eps)
    + \beta^2 \abs{\M_\eps}^4\right) \leq C,
 \end{equation}
 for some constant $C$ depending only on $\beta$ and $\Cpot$.
 This estimate will be useful in the sequel.
\end{remark}

\paragraph*{The Pohozaev identity.}

We define the \emph{stress-energy tensor} associated
to~$(\Q_\eps, \, \M_\eps)$ as
\begin{equation} \label{stressenergy}
 T_{jk}^\eps := \partial_j \Q_\eps\cdot\partial_k\Q_\eps
  + \eps \, \partial_j \M_\eps\cdot\partial_k\M_\eps
  - e_\eps(\Q_\eps, \, \M_\eps) \, \delta_{jk}
\end{equation}
for~$(j, \, k)\in\{1, \, 2\}^2$.
Here~$e_\eps(\Q_\eps, \, \M_\eps)$ is the energy density,
defined by~\eqref{energydensity}.
%
%
\begin{lemma} \label{lemma:stressenergy}
 The stress-energy tensor satisfies $\partial_j T_{jk}^\eps = 0$ in~$\Omega$.
\end{lemma}
\begin{proof}
 For ease of notation, we will write~$\Q$, $\M$
 instead of~$\Q_\eps$, $\M_\eps$.
 Let~$G$ be a smooth subdomain of~$\Omega$ and
 let~$\X\in C^\infty_{\mathrm{c}}(\R^2, \, \R^2)$
 be a test vector field, which may or may
 not be supported in~$\overline{\Omega}$.
 We test the equation~\eqref{EL-Q}
 against~$X_k\,\partial_k\Q$, test~\eqref{EL-M}
 against~$\eps \, X_k\,\partial_k\M$, then integrate over~$G$. We obtain
 \begin{equation*}
  -\int_G X_k\left(\partial_j\partial_j\Q\cdot\partial_k\Q
   + \eps \, \partial_j\partial_j\M\cdot\partial_k\M \right)
  + \frac{1}{\eps^2} \int_{G} X_k
   \left(\nabla_{\Q} f_\eps\cdot\partial_k\Q
   + \nabla_{\M} f_\eps\cdot\partial_k\M\right) = 0,
 \end{equation*}
 where~$\nabla_\Q f_\eps$, $\nabla_\M f_\eps$
 denote the (partial) gradients of~$f_\eps$
 with respect to the variables~$\Q$ and~$\M$,
 evaluated at~$(\Q_\eps, \, \M_\eps)$. By the chain rule,
 the previous equality can also be written as
 \begin{equation} \label{stren1}
  -\int_G X_k\left(\partial_j\partial_j\Q\cdot\partial_k\Q
   + \eps \, \partial_j\partial_j\M\cdot\partial_k\M \right)
  + \frac{1}{\eps^2} \int_{G} X_k
   \, \partial_k \left(f_\eps(\Q, \, \M)\right) = 0,
 \end{equation}
 Now, we integrate by parts repeatedly. The potential term can be written as
 \begin{equation} \label{stren2}
  \int_{G} X_k \, \partial_k \left(f_\eps(\Q, \, \M)\right)
  = - \int_{G} (\div\X) \, f_\eps(\Q, \, \M)
  + \int_{\partial G} (\X\cdot\nnu) \, f_\eps(\Q, \, \M) \, \d s,
 \end{equation}
 where~$\nnu$ is the outward unit normal to~$G$.
 We rewrite the first term by integrating by parts twice:
 \begin{equation} \label{stren3}
  \begin{split}
   -\int_G X_k\,\partial_j\partial_j\Q\cdot\partial_k\Q
   &= \int_G \partial_j X_k\,\partial_j\Q\cdot\partial_k\Q
     + \int_G X_k\,\partial_j\Q\cdot\partial_j\partial_k\Q
     - \int_{\partial G} (\nnu\cdot\nabla\Q)
      \cdot(\X\cdot\nabla\Q) \, \d s \\
   &= \int_G \partial_j X_k\,\partial_j\Q\cdot\partial_k\Q
     - \frac{1}{2} \int_{G} (\div\X) \abs{\nabla\Q}^2 \\
   &\hspace{2cm}  - \int_{\partial G} (\nnu\cdot\nabla\Q)
      \cdot(\X\cdot\nabla\Q) \, \d s
     + \frac{1}{2} \int_{\partial G} (\nnu\cdot\X)
      \abs{\nabla\Q}^2 \, \d s
  \end{split}
 \end{equation}
 A similar computation applies to the term
 containing the derivatives of~$\M$. By combining~\eqref{stren1}, \eqref{stren2} and~\eqref{stren3}, we obtain
 \begin{equation} \label{stren4}
  \begin{split}
   \int_G T_{jk}^\eps \, \partial_j X_k
   &= \int_{\partial G} \left( (\nnu\cdot\nabla\Q)\cdot(\X\cdot\nabla\Q)
    + \eps \, (\nnu\cdot\nabla\M)\cdot(\X\cdot\nabla\M)\right) \d s \\
   &\hspace{1.5cm} - \int_{\partial G} (\X\cdot\nnu)\, e_\eps(\Q, \, \M) \, \d s
  \end{split}
 \end{equation}
 By choosing~$G = \Omega$ and taking~$\mathbf{X}$ compactly
 supported in~$\Omega$,
 the boundary terms vanish, and the lemma follows.
\end{proof}

\begin{remark}
	Completely analogous computations to those above return
	the following fact, which will be used
	later on. Suppose that $\Q_\star\colon G \to \mathbb{S}^1$
    is a smooth harmonic map, i.e., a
	smooth solution to the equation
	\[
		\Delta \Q_\star + \abs{\nabla \Q_\star}^2 \Q_\star = 0
		\qquad \mbox{in } G,
	\]
	and let $\X$ be any smooth vector field with compact support in $G$. Then,
	\begin{equation}\label{eq:pohozaev-Q*}
		\int_G \left\{ \partial_j \Q_\star \cdot \partial_k \Q_\star \partial_j X_k - \frac{1}{2} (\div \X) \abs{\nabla \Q_\star}^2 \right\}\,{\d}x = 0.
	\end{equation}
\end{remark}

\begin{lemma} \label{lemma:pohozaev}
 For each ball~$B = B(x_0, \, R)\subseteq\Omega$,
 there holds
 \begin{equation}\label{eq:pohozaev-ball}
  \begin{split}
   \frac{2}{\eps^2} \int_{B} f_\eps(\Q_\eps, \, \M_\eps)\,{\d}x
    &+ \frac{R}{2}\int_{\partial B} \left( \abs{\partial_{\nnu}\Q_\eps}^2
     + \eps \abs{\partial_{\nnu}\M_\eps}^2\right) \d s\\
   &= \frac{R}{2}\int_{\partial B}
    \left( \abs{\partial_{\ttau}\Q_\eps}^2
	+ \eps \abs{\partial_{\ttau}\M_\eps}^2
	+ \frac{2}{\eps^2} f_\eps(\Q_\eps, \, \M_\eps)\right) \d s
  \end{split}
 \end{equation}
 where~$\nnu$ is the outward unit normal and~$\ttau$
 is the unit tangent field on~$\partial B$, oriented in such a way
 that $(\nnu, \, \ttau)$ is a positive basis.
\end{lemma}
\begin{proof}
 The lemma follows by taking~$G = B(x_0, \, R)$
 and~$\X(x) := \varphi(x)(x - x_0)$ in~\eqref{stren4}, where $\varphi : \R^2 \to \R$ is a smooth cut-off function such that $\varphi \equiv 1$ in a neighbourhood of $\Omega$.
\end{proof}

\paragraph*{The clearing-out lemma.}

%

This version of the ``clearing-out lemma'' is inspired
by~\cite[Theorem~2]{BBO}, but its proof is much simpler,
because we are working in a two-dimensional domain. Roughly stated,
it implies that, if the energy of~$(\Q_\eps, \M_\eps)$ on a ball
of radius~$R$ is sufficiently small with respect to~$\log(R/\eps)$,
then~$\Q_\eps$ is nonzero at the centre of the ball.

\begin{lemma}[Clearing-out] \label{lemma:clearingout}
 Let~$B = B(x_0, \, R)$ be a ball contained in~$\Omega$
 and let~$\eps$ be such that~$0 < \eps < R$. Then,
 \[
  \abs{\abs{\Q_\eps(x_0)} - 1}
  \lesssim \left(\frac{\F_\eps(\Q_\eps, \, \M_\eps; \, B)}
   {\log(R/\eps)}\right)^{1/4} + \eps^{1/2}
 \]
 The implicit constant on the right-hand side
 depends only on $\beta$ and $\Omega$.
\end{lemma}
\begin{proof}
 First of all, we claim that there exists~$\rho\in (\eps, \, R)$
 such that
 \begin{equation} \label{clearingout1}
  \int_{\partial B(x_0, \, \rho)} e_\eps(\Q_\eps, \, \M_\eps) \, \d s
  \leq\frac{\F_\eps(\Q_\eps, \, \M_\eps; B)}{\rho \, \log(R/\eps)}
 \end{equation}
 For otherwise, we would have
 \[
  \begin{split}
   \F_\eps(\Q_\eps, \, \M_\eps; \, B)
   \geq \int_{\eps}^R \left(\int_{\partial B(x_0, \, \rho)}
    e_\eps(\Q_\eps, \, \M_\eps) \, \d s\right) \, \d\rho
   &> \int_{\eps}^R \frac{\F_\eps(\Q_\eps, \, \M_\eps; B)}
    {\rho\, \log(R/\eps)} \d\rho\\
   &= \F_\eps(\Q_\eps, \, \M_\eps; \, B),
  \end{split}
 \]
 which is a contradiction. For such a value of~$\rho$,
 Lemma~\ref{lemma:pohozaev} gives
 \begin{equation} \label{clearingout2}
  \begin{split}
   \frac{2}{\eps^2} \int_{B(x_0, \, \eps)} f_\eps(\Q_\eps, \, \M_\eps)
   \leq \frac{2}{\eps^2} \int_{B(x_0, \, \rho)} f_\eps(\Q_\eps, \, \M_\eps)
   &\leq \rho \int_{\partial B(x_0, \, \rho)}
    e_\eps(\Q_\eps, \, \M_\eps) \, \d s \\
  & \leq\frac{\F_\eps(\Q_\eps, \, \M_\eps; \, B)}{\log(R/\eps)}
  \end{split}
 \end{equation}
 (see~\cite[Lemma~B.1]{CanevariMajumdarStroffoliniWang}).
 Therefore, taking~\eqref{potential_comparison}
 and Lemma~\ref{lemma:max} into account,
 from~\eqref{clearingout2} we deduce
 \begin{equation} \label{clearingout3}
  \begin{split}
   \frac{1}{\eps^2} \int_{B(x_0, \, \eps)} \left(\abs{\Q_\eps}^2 - 1\right)^2
   &\lesssim \int_{B(x_0, \, \eps)}
    \left( \frac{1}{\eps^2} f_\eps(\Q_\eps, \, \M_\eps)
    + \beta^2 \abs{\M_\eps}^4\right) \\
   &\lesssim \frac{\F_\eps(\Q_\eps, \, \M_\eps; \, B)}{\log(R/\eps)}
    + \beta^2\eps^2
  \end{split}
 \end{equation}
 Finally, the gradient estimate~\eqref{max-gradients}
 implies that
 \begin{equation} \label{clearingout4}
  \abs{\abs{\Q_\eps(x_0)} - 1}
  \lesssim \left(\frac{1}{\eps^2} \int_{B(x_0, \, \eps)}
   \left(\abs{\Q_\eps}^2 - 1\right)^2\right)^{1/4}
 \end{equation}
 --- see~\cite[Lemma~III.3]{BBO} for the details.
 (The statement in~\cite{BBO} only provides
 a one-sided bound for~$1 - \abs{\Q_\eps}$,
 but the very same argument can be used to
 prove the opposite inequality.)
 The lemma follows by combining~\eqref{clearingout3}
 with~\eqref{clearingout4}.
\end{proof}

\section{Estimates for~$\Q_\eps$ and $\M_\eps$}
\label{sect:compQ}

In this section, we derive several $\eps$-independent estimates
on $\Q_\eps$ and $\M_\eps$. For the
reader's convenience, such estimates are all collected
in Theorem~\ref{lemma:energy-est} at the end of the section.

Our arguments follow a classical path and
rely on testing the Euler-Lagrange
equations~\eqref{EL-Q},~\eqref{EL-M} against suitable
maps. In particular, starting from a basic clearing-out
lemma for the $\Q_\eps$-component (reminiscent of classical
results in Ginzburg-Landau theory), we show that, for boundary conditions
as in~\eqref{bc}--\eqref{hp:bc} or as in~\eqref{bcbis}--\eqref{hp:bcbis}
and under assumption~\eqref{hp:potential_bound}, the energy
of a sequence of critical points blows-up (at most) logarithmically
in $\eps$. Moreover, the largest part of the energy is carried
$\Q_\eps$-component; in fact, all terms in $\F_\eps$
different from $\int_\Omega \abs{\nabla \Q_\eps}^2$ are uniformly bounded;
furthermore, even the elastic energy of $\Q_\eps$ remains finite around points
at which $\abs{\Q_\eps} \geq 1/2$ for any $\eps$ small enough.

The results in this section will be crucially exploited all
along the rest of the paper. In particular, in Section~\ref{sect:compactness},
we use them to prove compactness properties for both
the $\Q_\eps$-component and $\M_\eps$-component of a sequence
$\{(\Q_\eps,\,\M_\eps)\}$ of critical points of $\F_\eps$.

\subsection{Bounds in~$W^{1,p}(\Omega)$ for $\Q_\eps$}
\label{sect:W1p}

The goal of this section is to prove the following result.

\begin{prop} \label{prop:Q_bound}
 Let~$(\Q_\eps, \, \M_\eps)$ be a sequence of solutions
 to~\eqref{EL-Q}--\eqref{EL-M}, subject to
 either~\eqref{bc}--\eqref{hp:bc} or~\eqref{bcbis}--\eqref{hp:bcbis}.
 Assume that the condition~\eqref{hp:potential_bound} is satisfied.
 Then, there exists 
 a constant~$C$, depending only on $\Omega$, $\beta$, $\Cpot$, and the boundary data,
 such that, for any~$\eps$ small enough, there holds
 \begin{equation} \label{Qbound-log}
  \int_\Omega \abs{\nabla\Q_\eps}^2 \leq C \abs{\log\eps} \! .
 \end{equation}
 Moreover, for any~$p\in [1, \, 2)$ there exists a constant~$C_p >0$,
 depending only on $\Omega$, $p$, $\beta$, $\Cpot$,
 and the boundary data
 such that, for any~$\eps$ small enough, there holds
 \begin{equation} \label{Qbound-p}
  \int_\Omega \abs{\nabla\Q_\eps}^p \leq C_p .
 \end{equation}
\end{prop}

\begin{remark}\label{rk:Q_bound}
	More precisely, the constants on the right-hand side of~\eqref{Qbound-log}
	and of~\eqref{Qbound-p} depend on the boundary data only through
	the $L^1(\partial \Omega)$- and the $L^2(\partial\Omega)$-norm of
	$\Qb \times \partial_\ttau \Qb$.
\end{remark}

As in~\cite{BBH, BBO}, the main step in the proof of
Proposition~\ref{prop:Q_bound} is to obtain bounds on~$j(\Q_\eps)$.
To this end, we write the equation satisfied by~$j(\Q_\eps)$
by taking the vector product of
Equations~\eqref{EL-Q}--\eqref{EL-M}
against~$\Q_\eps$ and~$\M_\eps$ respectively:
\begin{equation} \label{almostharmonic}
 -\div\left(j(\Q_\eps)\right) = \frac{\eps}{2} \,
   \partial_k\left(\M_\eps\times\partial_k\M_\eps\right)
  \qquad \textrm{in } \Omega
\end{equation}
(see~\cite[Lemma~4.8]{CanevariMajumdarStroffoliniWang} for details).
Lemma~\ref{lemma:max} implies that the right-hand side of
Equation~\eqref{almostharmonic} is bounded in~$W^{-1,\infty}(\Omega)$:
indeed,
\begin{equation} \label{almostharmonic-bound}
 \eps \norm{\M_\eps\times\partial_k\M_\eps}_{L^\infty(\Omega)}
  \leq \eps \norm{\M_\eps}_{L^\infty(\Omega)}
   \norm{\partial_k\M_\eps}_{L^\infty(\Omega)} \leq C
\end{equation}
for~$k\in\{1, \, 2\}$ and some constant~$C$ that
does not depend on~$\eps$ but only on $\beta$. This allows us to
obtain~$\eps$-independent bounds on~$j(\Q_\eps)$.

\begin{lemma} \label{lemma:boundj-p}
 For any~$p\in [1, \, 2)$, there exists an
 $\eps$-independent constant~$C_p > 0$ such that
 \begin{equation*}
  \norm{j(\Q_\eps)}_{L^p(\Omega)} \leq C_p
 \end{equation*}
 for any~$\eps$ small enough.
 The constant $C_p$ depends only on $p$, $\beta$, $\Cpot$,
 and the $L^1(\partial\Omega)$-norm of $\Qb \times \partial_\ttau \Qb$.
\end{lemma}
\begin{proof}
 \setcounter{step}{0}
 As in~\cite{BBH, BBO}, the strategy is to consider
 (a suitable variant of) the Helmoltz, or Hodge,
 decomposition of~$j(\Q_\eps)$. As a preliminary step,
 we will decompose~$j(\Q_\eps)$ in a suitable way.
 We use the notation~$\rho_\eps := \abs{\Q_\eps}$,
 $\j_\eps := j(\Q_\eps)$.
 \begin{step}[Splitting~$\curl\j_\eps$]
  Let $A\colon [0, \, +\infty)\to [0, \, 2]$
  be a smooth function such that
  \begin{equation} \label{A}
   A(0) = 1,\qquad
   A(t) = \frac{1}{t} \qquad \textrm{if } t \geq \frac{1}{2},
   \qquad
   \abs{A'(t)} \leq 4 \qquad \textrm{for any } t \geq 0,
  \end{equation}
  and let~$\alpha_\eps := A(\rho_\eps)$. We write
  \[
   \curl\j_\eps
   = \curl(\alpha^2_\eps\,\j_\eps)
    + \curl((1 - \alpha_\eps^2)\, \j_\eps).
  \]
  The first term, $\curl(\alpha^2_\eps\,\j_\eps)$,
  is supported in~$\{\rho_\eps\leq 1/2\}$, because
  of~\eqref{preJaccurl} and~\eqref{A}.
  Lemma~\ref{lemma:max} implies the pointwise inequality
  \begin{equation} \label{curl0}
   \abs{\curl(\alpha^2_\eps\,\j_\eps)}
   \lesssim \frac{\one_{\{\rho_\eps\leq 1/2\}}}{\eps^2}
   \lesssim \frac{1}{\eps^2} \left(\rho_\eps^2 - 1\right)^2,
  \end{equation}
  where~$\one_{\{\rho_\eps\leq 1/2\}}$ is the indicator function
  of the set~$\{\rho_\eps\leq 1/2\}$. Then, recalling~\eqref{GLbound},
  we obtain the $\eps$-independent bound
  \begin{equation} \label{curl1}
   \norm{\curl(\alpha^2_\eps\,\j_\eps)}_{L^1(\Omega)} \leq C.
  \end{equation}
  In fact, the constant $C$ depends only on $\beta$ and $\Cpot$.

  As for the second term, $\curl((1 - \alpha_\eps^2)\,\j_\eps)$,
  Lemma~\ref{lemma:max} and~\eqref{A} imply the pointwise bound
  \begin{equation*}
   \abs{(1 - \alpha^2_\eps)\,\j_\eps}^2
   \lesssim \frac{(\rho^2_\eps - 1)^2}{\eps^2 \, \rho_\eps^4}
    \one_{\{\rho_\eps\geq 1/2\}}
   + \frac{\one_{\{\rho_\eps\leq 1/2\}}}{\eps^2}
   \lesssim \frac{1}{\eps^2} (\rho_\eps^2 - 1)^2.
  \end{equation*}
  Therefore, \eqref{GLbound} gives
  \begin{equation} \label{curl2}
   \norm{(1 - \alpha^2_\eps)\j_\eps}_{L^2(\Omega)}
   \leq C,
  \end{equation}
  where~$C$ does not depend on~$\eps$ but only on $\beta$ and $\Cpot$.
 \end{step}

 \begin{step}[Hodge decomposition of~$\j_\eps$]
  We consider the following elliptic problems:
  \begin{equation} \label{phi_eps}
   \begin{cases}
    \Delta\phi_\eps \hspace{0.8mm} = \curl(\alpha_\eps^2 \, \j_\eps)
     & \textrm{in } \Omega \\
    \partial_{\nnu}\phi_\eps = \dfrac{1}{2}\Qb\times\partial_{\ttau}\Qb
     & \textrm{on } \partial\Omega
   \end{cases}
  \end{equation}
  and
  \begin{equation} \label{xi_eps}
   \begin{cases}
    \Delta\xi_\eps \hspace{0.8mm}
     = \curl((1 - \alpha_\eps^2) \, \j_\eps)
     & \textrm{in } \Omega \\
    \partial_{\nnu}\xi_\eps = 0
     & \textrm{on } \partial\Omega
   \end{cases}
  \end{equation}
  The right-hand sides of~\eqref{phi_eps} and~\eqref{xi_eps}
  satisfy the solvability conditions for Neumann problems.
  Indeed, if~$\ttau$ denotes the unit tangent vector
  to~$\partial\Omega$ in the anti-clockwise direction, we have
  \begin{equation} \label{curl-bc}
   \alpha_\eps^2\,\j_\eps\cdot\ttau
    = \frac{1}{2}\Qb\times\partial_{\ttau}\Qb,
   \quad (1 - \alpha^2_\eps)\,\j_\eps = 0
    \qquad \textrm{on } \partial\Omega
  \end{equation}
  because our assumptions on the boundary conditions
  (either~\eqref{hp:bc} or~\eqref{hp:bcbis})
  imply~$\abs{\Q_\eps} = \abs{\Qb} = 1$ on~$\partial\Omega$.
  As a consequence, we obtain
  \begin{align*}
      \int_\Omega \curl\left(\alpha_\eps^2\,\j_\eps\right)
    &= \int_{\partial\Omega} \alpha_\eps^2\,\j_\eps\cdot\ttau
     = \frac{1}{2} \int_{\partial\Omega} \Qb\times\partial_{\ttau}\Qb \\
      \int_\Omega \curl\left((1 -\alpha_\eps^2)\,\j_\eps\right)
    &= \int_{\partial\Omega} (1 - \alpha_\eps^2)\,\j_\eps\cdot\ttau
     = 0,
  \end{align*}
  and the solutions~$\phi_\eps$, $\xi_\eps$
  exist and are unique up to additive constants.
  Moreover, writing~$\nabla^\perp := (-\partial_2, \, \partial_1)$,
  we have
  \[
   \curl\left(\j_\eps - \nabla^\perp\phi_\eps
    - \nabla^\perp\xi_\eps\right)
   = \curl\j_\eps - \Delta\phi_\eps - \Delta\xi_\eps
   = 0 \qquad \textrm{in } \Omega.
  \]
  As~$\Omega$ is assumed to be simply connected,
  there exists a function~$H_\eps\colon\Omega\to\R$ such that
  \begin{equation} \label{Hodge}
   \j_\eps = \nabla H_\eps + \nabla^\perp\phi_\eps
    + \nabla^\perp\xi_\eps \qquad \textrm{in } \Omega.
  \end{equation}
  By taking the divergence of both sides of~\eqref{Hodge},
  and keeping~\eqref{almostharmonic} into account,
  we obtain
  \begin{equation} \label{H_eps}
   -\Delta H_\eps = \frac{\eps}{2} \,
   \partial_k\left(\M_\eps\times\partial_k\M_\eps\right)
   \qquad  \textrm{in } \Omega.
  \end{equation}
  In a similar way, by taking the \emph{tangential component}
  of the trace at the boundary in both sides of~\eqref{Hodge},
  and using~\eqref{phi_eps}, \eqref{xi_eps}, \eqref{curl-bc},
  we deduce
  \[
   \partial_{\ttau} H_\eps = 0 \qquad \textrm{on }
   \partial \Omega.
  \]
  The boundary of~$\Omega$ is connected, because~$\Omega$
  is simply connected. Therefore, up to an additive constant,
  we can assume that
  \begin{equation} \label{Heps-bc}
   H_\eps = 0 \qquad \textrm{on }
   \partial \Omega.
  \end{equation}
 \end{step}

 \begin{step}[Elliptic estimates on~$\phi_\eps$, $\xi_\eps$, $H_\eps$]
  We estimate~$\phi_\eps$ by applying
  estimates for the Poisson problem with integrable right-hand side
  (see e.g.~\cite[Proposition~A.2]{BBO}).
  Keeping~\eqref{curl1} into account, we obtain
  \begin{equation} \label{phieps-p}
   \begin{split}
    \norm{\nabla{\phi_\eps}}_{L^p(\Omega)}
    \leq C_p\left(\norm{\curl(\alpha_\eps^2 \, \j_\eps)}_{L^1(\Omega)}
     + \norm{\Qb\times\partial_{\ttau}\Qb}_{L^1(\partial\Omega)}\right)
    \leq C_p,
   \end{split}
  \end{equation}
  for any~$p\in [1, \, 2)$ and some constant~$C_p$
  that does not depend on~$\eps$ but only on $p$, $\beta$,
  $\Cpot$, and the boundary data. As for~$\xi_\eps$,
  we test Equation~\eqref{xi_eps} against~$\xi_\eps$
  and integrate by parts, recalling
  that~$(1 - \alpha_\eps^2)\, \j_\eps = 0$ on~$\partial\Omega$:
  \begin{equation*}
   \begin{split}
    \int_{\Omega} \abs{\nabla\xi_\eps}^2
    = - \int_{\Omega} (1 - \alpha_\eps^2)\, \j_\eps \times \nabla\xi_\eps.
   \end{split}
  \end{equation*}
  Then, Young's inequality and~\eqref{curl2} give
  \begin{equation} \label{xieps-2}
   \begin{split}
    \norm{\nabla\xi_\eps}_{L^2(\Omega)}
    \leq \norm{(1 - \alpha_\eps^2)\, \j_\eps}_{L^2(\Omega)}
    \leq C,
   \end{split}
  \end{equation}
  for the same constant $C$ as in~\eqref{curl2}.
  In a similar way, testing~\eqref{H_eps}--\eqref{Heps-bc}
  against~$H_\eps$, integrating by parts,
  and applying~\eqref{almostharmonic-bound}, we obtain
  \begin{equation} \label{Heps-2}
   \begin{split}
    \norm{\nabla H_\eps}_{L^2(\Omega)}
    \leq \eps \sum_{k=1}^2
     \norm{\M_\eps\times\partial_k\M_\eps}_{L^2(\Omega)}
    \leq C.
   \end{split}
  \end{equation}
  Combining~\eqref{phieps-p}, \eqref{xieps-2}
  and~\eqref{Heps-2}, the lemma follows.
  \qedhere
 \end{step}
\end{proof}

\begin{lemma} \label{lemma:boundj-2}
 Assume~\eqref{hp:potential_bound} holds.
 For any~$\eps$ small enough, we have
 \begin{equation*}
  \norm{j(\Q_\eps)}_{L^2(\Omega)}
   \lesssim \abs{\log\eps}^{1/2},
 \end{equation*}
 where the implicit constant on the right-hand side
 depends only on $\abs{\Omega}$, $\beta$, $\Cpot$, and
 the~$L^2(\partial\Omega)$-norm of~$\Qb \times \partial_\ttau \Qb$.
\end{lemma}
\begin{proof}
 We consider again the decomposition of~$j(\Q_\eps)$
 given in~\eqref{Hodge}. In view of~\eqref{xieps-2} and~\eqref{Heps-2},
 the lemma will follow if we prove that the solution~$\phi_\eps$
 to~\eqref{phi_eps} satisfies
 \begin{equation} \label{phieps2}
  \norm{\nabla\phi_\eps}_{L^2(\Omega)}
   \lesssim \abs{\log\eps}^{1/2}.
 \end{equation}
 We follow the same strategy as in~\cite[Lemma~X.5]{BBH}.
 To simplify the notation, we
 set~$h_\eps := \curl(\alpha_\eps^2 \, j(\Q_\eps))$
 and~$2g := \Qb\times\partial_{\ttau}\Qb$.
 By testing~\eqref{phi_eps} against~$\phi_\eps$,
 we obtain
 \begin{equation} \label{phi21}
  \int_{\Omega} \abs{\nabla\phi_\eps}^2
  = \int_{\partial\Omega} g \, \phi_\eps \, \d s
   - \int_{\Omega} h_\eps \, \phi_\eps
 \end{equation}
 We estimate the boundary term first. Since the
 solution~$\phi_\eps$ is only identified up to an
 additive constant, we can assume without loss of generality
 that~$\phi_\eps$ has zero average in~$\Omega$. Then,
 the trace inequality and the Poincar\'e inequality give
 \begin{equation} \label{phi22}
  \abs{\int_{\partial\Omega} g \, \phi_\eps \, \d s}
  \leq \norm{g}_{L^2(\partial\Omega)} \norm{\phi_\eps}_{L^2(\partial\Omega)}
  \lesssim \norm{\phi_\eps}_{H^1(\Omega)}
  \lesssim \norm{\nabla\phi_\eps}_{L^2(\Omega)}
 \end{equation}
 (the boundary datum~$g$ depends only on~$\Qb$,
 not on~$\eps$, and is continuous).
 Now, we estimate the term that depends on~$h_\eps$.
 Let~$Z_\eps := \{\abs{\Q_\eps}\leq 1/2\}$.
 We recall (see~\eqref{curl0}) that~$\abs{h_\eps}\lesssim \eps^{-2}$
 and~$h_\eps$ is supported in~$Z_\eps$. Moreover,
 the estimate~\eqref{GLbound}
 (which holds thanks to the assumption~\eqref{hp:potential_bound})
 for the Ginzburg-Landau potential
 implies
 \begin{equation} \label{phi22.5}
  \abs{Z_\eps} \lesssim \int_{\Omega}\left(\abs{\Q_\eps}^2 - 1\right)^2
  \lesssim \eps^2.
 \end{equation}
 As a consequence, we obtain
 \begin{equation} \label{phi23}
  \begin{split}
   \abs{\int_{\Omega} h_\eps \, \phi_\eps}
   \lesssim \frac{1}{\eps^2}
    \int_{Z_\eps} \abs{\phi_\eps}
   \lesssim \norm{\phi_\eps}_{L^\infty(Z_\eps)}
  \end{split}
 \end{equation}
 We can further estimate the~$L^\infty(Z_\eps)$-norm
 of~$\phi_\eps$ by applying Trudinger's inequality
 (see e.g~\cite[Theorem~7.15]{GilbargTrudinger}),
 which gives
 \begin{equation*}
  \int_\Omega \exp\left(\frac{\abs{\phi_\eps}^2}
   {\sigma_1^2 \norm{\nabla\phi_\eps}^2_{L^2(\Omega)}}\right)
  \leq \sigma_2 \abs{\Omega}
 \end{equation*}
 for some universal constants~$\sigma_1$, $\sigma_2$.
 It follows that
 \begin{equation*}
  \exp\left(\frac{\norm{\phi_\eps}^2_{L^\infty(Z_\eps)}}
   {\sigma_1^2 \norm{\nabla\phi_\eps}^2_{L^2(\Omega)}}\right) \abs{{Z_\eps}}
   \leq \sigma_2 \abs{\Omega}
 \end{equation*}
 and, hence,
 \begin{equation} \label{phi24}
   \norm{\phi_\eps}_{L^\infty({Z_\eps})}
    \leq \sigma_1 \norm{\nabla\phi_\eps}_{L^2(\Omega)}
   \left(\log\left(\frac{\sigma_2 \abs{\Omega}}
    {\abs{{Z_\eps}}}\right)\right)^{1/2}
   \lesssim \norm{\nabla\phi_\eps}_{L^2(\Omega)}
    \left(\abs{\log\eps}^{1/2} + 1\right)
 \end{equation}
 for any~$i\in\{1, \, \ldots, \, N_\eps\}$.
 Combining~\eqref{phi22} with~\eqref{phi23}
 and~\eqref{phi24}, we obtain
 \begin{equation*}
  \norm{\nabla\phi_\eps}_{L^2(\Omega)}^2
  \lesssim \left(\abs{\log\eps}^{1/2} + 1\right)
   \norm{\nabla\phi_\eps}_{L^2(\Omega)}
 \end{equation*}
 and~\eqref{phi21} follows.
\end{proof}

Next, we need to estimate~$\rho_\eps := \abs{\Q_\eps}$.
To this end, we write an equation for~$\rho_\eps$
by taking the scalar product of both sides of~\eqref{EL-Q}
against~$\Q_\eps$. We obtain
\begin{equation} \label{rhoeps}
 -\frac{1}{2} \Delta(\rho_\eps^2) + \abs{\nabla\Q_\eps}^2
 + \frac{1}{\eps^2}\left(\rho_\eps^2 - 1\right) \rho_\eps^2 - \frac{\beta}{\eps} \Q_\eps\M_\eps\cdot\M_{\eps} = 0 \qquad \textrm{in } \Omega.
\end{equation}
From this equation, we deduce the following bound:

\begin{lemma} \label{lemma:rhoeps2}
 For any~$\delta \in (0, \, 1/2)$, any~$p\in [1, \, 2]$
 and any~$\eps > 0$ small enough, there holds
 \[
  \int_\Omega\abs{\nabla\rho_\eps}^p
   \lesssim \left(\delta\int_{\Omega}\abs{\nabla\Q_\eps}^2\right)^{p/2}
   + \frac{\eps^{2-p}}{\delta^2} + 1,
 \]
 where the implicit constant in front of the
 right-hand side is independent of~$\delta$, $p$, $\eps$
 and depends only on $\Omega$, $\beta$, $\Cpot$,
 and the $L^1(\partial \Omega)$- and the $L^2(\partial\Omega)$-norm of
$\Qb \times \partial_\ttau \Qb$.
\end{lemma}
\begin{proof}
 For a given value of~$\delta\in (0, \, 1/2)$, let
 \[
  E_\eps := \left\{x\in\Omega\colon \abs{\rho_\eps(x) - 1}
   \leq \delta \right\} \! .
 \]
 The bound~\eqref{GLbound} implies
 \begin{equation} \label{rhoeps20}
  \abs{\Omega\setminus E_\eps} \lesssim \frac{1}{\delta^2}
   \int_{\Omega} {\left(\rho_\eps - 1\right)^2}
   \lesssim \frac{\eps^2}{\delta^2},
 \end{equation}
 where the implicit constant on the right-hand side depends
 only on $\beta$ and $\Cpot$,
 and hence, recalling~\eqref{max-gradients},
 \begin{equation} \label{rhoeps21}
  \int_{\Omega\setminus E_\eps} \abs{\nabla\rho_\eps}^2
  \lesssim \frac{1}{\delta^2},
 \end{equation}
 where, again, the implicit constant on the right-hand side depends
 only on $\beta$ and $\Cpot$.

 It remains to estimate~$\nabla\rho_\eps$
 on~$E_\eps$. To this purpose, we define
 \[
  \overline{\rho}_\eps :=
  \begin{cases}
   1 + \delta &\textrm{if } \rho_\eps \geq 1 + \delta \\
   \rho_\eps  &\textrm{if } 1-\delta < \rho_\eps < 1 + \delta \\
   1 - \delta &\textrm{if } \rho_\eps \leq 1 - \delta.
  \end{cases}
 \]
 We have~$\overline{\rho}_\eps = 1$ on~$\partial\Omega$,
 because~$\abs{\Q_\eps} = \abs{\Qb} = 1$ on~$\partial\Omega$
 by~\eqref{hp:bc}, \eqref{hp:bcbis}.
 To estimate~$\nabla\rho_\eps$ on~$E_\eps$,
 we test Equation~\eqref{rhoeps}
 against~$1 - \overline{\rho}_\eps$ and integrate by parts.
 We obtain
 \begin{equation} \label{rhoeps22}
  \begin{split}
   \int_{E_\eps} \rho_\eps \abs{\nabla\rho_\eps}^2
   &= \int_{\Omega}\left(\abs{\nabla\Q_\eps}^2
    + \frac{1}{\eps^2}(\rho_\eps^2 - 1) \rho_\eps^2
    - \frac{\beta}{\eps} \Q_\eps\M_\eps\cdot\M_{\eps}\right)
    (1 - \overline{\rho}_\eps) \\
   &\leq \delta\int_{\Omega}\abs{\nabla\Q_\eps}^2
    + \mathrm{I}_\eps + \mathrm{II}_\eps,
  \end{split}
 \end{equation}
 where
 \[
  \mathrm{I}_\eps := \frac{1}{\eps^2}\int_{\Omega}
   \abs{\rho_\eps^2 - 1} \abs{1 - \overline{\rho}_\eps}, \qquad
  \mathrm{II}_\eps := \frac{\beta}{\eps}\int_{\Omega}
   \abs{1 - \overline{\rho}_\eps} \,
   \abs{\Q_\eps\M_\eps\cdot\M_{\eps}}
 \]
 To estimate~$\mathrm{I}_\eps$, we observe that~$\abs{1 - \overline{\rho}_\eps}
 \leq \abs{1 - \rho_\eps}$ and hence,
 \begin{equation} \label{rhoeps23}
  \begin{split}
   \abs{\mathrm{I}_\eps} \leq \frac{1}{\eps^2} \int_{\Omega}
   \frac{(\rho_\eps^2 - 1)^2}{1 + \rho_\eps}
   \leq \frac{1}{\eps^2} \int_{\Omega}(\rho_\eps^2 - 1)^2
   \leq C,
  \end{split}
 \end{equation}
 because of~\eqref{GLbound}. To estimate~$\mathrm{II}_\eps$,
 we apply the~$L^\infty$-bound~\eqref{max-QM},
 the H\"older inequality, and~\eqref{GLbound} again:
 \begin{equation} \label{rhoeps24}
  \begin{split}
   \abs{\mathrm{II}_\eps} \lesssim
   \frac{1}{\eps} \int_{\Omega}\abs{1 - \rho_\eps}
   \lesssim \left(\frac{1}{\eps^2}
   \int_{\Omega} (\rho_\eps^2 - 1)^2 \right)^{1/2}
   \leq C,
  \end{split}
 \end{equation}
 where the constant $C$ depends only on $\Omega$, $\beta$, and $\Cpot$.
 From~\eqref{rhoeps22}, \eqref{rhoeps23} and~\eqref{rhoeps24},
 we deduce
 \begin{equation} \label{rhoeps25}
  \int_{E_\eps} \abs{\nabla\rho_\eps}^p
  \lesssim \left(\int_{E_\eps} \rho_\eps\abs{\nabla\rho_\eps}^2\right)^{p/2}
  \lesssim \left(\delta\int_{\Omega}\abs{\nabla\Q_\eps}^2\right)^{p/2}
   + C
 \end{equation}
 (we have used that~$\rho_\eps\geq 1 - \delta \geq 1/2$ on~$E_\eps$).
 Taking~\eqref{rhoeps21} into account, the lemma follows.
\end{proof}

We can now complete the proof of Proposition~\ref{prop:Q_bound}.

\begin{proof}[Proof of Proposition~\ref{prop:Q_bound}]
 Let~$\rho_\eps := \abs{\Q_\eps}$, and let
 \[
  F_\eps := \left\{x\in\Omega\colon \rho_\eps(x)\geq \frac{1}{2}\right\}
 \]
 The inequality~\eqref{GLbound} implies that~$\abs{\Omega\setminus F_\eps} \lesssim \eps^2$,
 so the $L^\infty$-bound~\eqref{max-gradients} on~$\nabla\Q_\eps$ gives
 \begin{equation} \label{Qbound1}
  \int_{\Omega\setminus F_\eps} \abs{\nabla\Q_\eps}^2 \leq C
 \end{equation}
 for some $\eps$-independent constant~$C$.
 To estimate the~$L^2$-norm of~$\nabla\Q_\eps$ in~$F_\eps$,
 we use the identity~\eqref{preJacnabla}, which gives
 \[
  \int_{F_\eps} \abs{\nabla\Q_\eps}^2
  = \int_{F_\eps} \left(\abs{\nabla\rho_\eps}^2
   + \frac{\abs{j(\Q_\eps)}^2}{\rho_\eps^2}\right)
  \lesssim \int_{F_\eps} \left(\abs{\nabla\rho_\eps}^2
   + \abs{j(\Q_\eps)}^2\right)
 \]
 Let~$\delta \in (0, \, 1/2)$ be a small parameter.
 Lemma~\ref{lemma:boundj-2} and Lemma~\ref{lemma:rhoeps2} imply
 \begin{equation} \label{Qbound2}
  \int_{F_\eps} \abs{\nabla\Q_\eps}^2
  \lesssim \delta\int_{\Omega} \abs{\nabla\Q_\eps}^2
   + \delta^{-2} + \abs{\log\eps}
 \end{equation}
 Combining~\eqref{Qbound1} with~\eqref{Qbound2},
 and choosing~$\delta$ small enough (but still
 independent of~$\eps$), we obtain~\eqref{Qbound-log}.
 Now, let~$p\in [1, \, 2)$. The identity~\eqref{preJacnabla},
 Lemma~\ref{lemma:boundj-p}, and Lemma~\ref{lemma:rhoeps2}
 (applied with the choice~$\delta = \abs{\log\eps}^{-2}$) imply
 \begin{equation*}
  \begin{split}
   \int_{F_\eps} \abs{\nabla\Q_\eps}^p
   &\lesssim \int_{F_\eps} \left(\abs{\nabla\rho_\eps}^p
    + \abs{j(\Q_\eps)}^p\right) \\
   &\lesssim \left(\frac{1}{\abs{\log\eps}^2}
    \int_{\Omega}\abs{\nabla\Q_\eps}^2\right)^{p/2}
    + \eps^{2-p}\abs{\log\eps}^4 + C_p
  \end{split}
 \end{equation*}
 where~$C_p$ is a constant that depends on~$p$, but not on~$\eps$.
 Therefore, keeping~\eqref{Qbound-log} into account,
 we obtain
 \begin{equation} \label{Qbound3}
  \begin{split}
   \int_{F_\eps} \abs{\nabla\Q_\eps}^p
   &\lesssim \frac{1}{\abs{\log\eps}^{p/2}}
    + \eps^{2-p}\abs{\log\eps}^4 + C_p
    \leq 1 + C_p
  \end{split}
 \end{equation}
 for~$\eps$ small enough. The estimate~\eqref{Qbound-p} follows
 from~\eqref{Qbound1} and~\eqref{Qbound3}.
\end{proof}

\subsection{$\eta$-ellipticity for $\Q_\eps$}\label{sec:eta-ellipt}

The goal of this section is to prove the following result.

\begin{prop} \label{prop:clearing-out}
	There exists positive numbers $\eps_*$, $\eta_*$ and,
	for any~$x_0\in\Omega$, $R > 0$, a constant~$C_*(x_0, \, R) > 0$
	such that the following statement holds.
	Let~$(\Q_\eps, \, \M_\eps)$ be a solution to~\eqref{EL-Q}--\eqref{EL-M},
	subject to boundary conditions either as in~\eqref{bc}--\eqref{hp:bc}
	or as in~\eqref{bcbis}--\eqref{hp:bcbis}.
	Assume that the condition~\eqref{hp:potential_bound} is satisfied,
	that $B(x_0, \, R) \csubset \Omega$ and that~$0 < \eps \leq \eps_* R$. If
	\begin{equation} \label{hp:eta*}
		\F_\eps(\Q_\eps, \M_\eps; \, B(x_0, \, R))
		 \leq \eta_* \log\frac{R}{\eps},
	\end{equation}
	then
	\begin{equation}\label{eta-nonzero}
		\mbox{ for any } x \mbox{ in }
		B\!\left(x_0, \, 3R/4\right)\!, \quad
		\abs{\Q_\eps(x)} \geq \frac{1}{2}.
	\end{equation}
	Moreover, whenever~\eqref{eta-nonzero} holds
	(irrespectively of~\eqref{hp:eta*}), there holds
	\begin{equation} \label{eta-energybound}
	 \int_{B(x_0, \, R/2)} \abs{\nabla\Q_\eps}^2 \leq C_*(x_0, \, R).
	\end{equation}
\end{prop}
\begin{proof}
 Assume that~$B(x_0, \, R)\csubset\Omega$
 and~$\eps\in (0, \, \eps_*R]$ are such that~\eqref{hp:eta*}
 is satisfied, for some positive constants~$\eta_*$, $\eps_*$
 to be determined later.

 \setcounter{step}{0}
 \begin{step}[Proof of~\eqref{eta-nonzero}]
  Assume that~$\eps_* \leq 1/8$, so that~$\eps\leq R/8$.
  Then, for any~$x\in B(x_0, \, 3R/4)$, Lemma~\ref{lemma:clearingout}
  implies
  \[
   \begin{split}
    \abs{\abs{\Q_\eps(x)} - 1}
    &\lesssim \left(\frac{\F_\eps(\Q_\eps, \, \M_\eps; \, B(x, \, R/4))}{\log\left(\dfrac{R}{4\eps}\right)}\right)^{1/4} + \eps^{1/2} \\
    &\hspace{-2mm} \stackrel{\eqref{hp:eta*}}{\lesssim}
    \left(\frac{\eta_*\log(R/\eps)}{\log(R/\eps) - \log 4}\right)^{1/4} + \eps^{1/2}
    \lesssim 3^{1/4} \, \eta_*^{1/4}  + \eps_*^{1/2} \, R^{1/2}
   \end{split}
  \]
  If~$\eta_*$ and~$\eps_*$ are chosen small enough,
  then~\eqref{eta-nonzero} follows.
 \end{step}

 \begin{step}[Bounds on~$j(\Q_\eps)$]
  Let~$\rho_\eps := \abs{\Q_\eps}$. Thanks to~\eqref{eta-nonzero},
  we can apply topological lifting results and write~$\Q_\eps$
  in ``polar coordinates'', as in~\eqref{Qpolar}, for some
  (smooth) scalar function~$\varphi_\eps\colon B(x_0, \, 3R/4)\to\R$.
  Up to an additive constant, we can assume without loss of generality that
  \begin{equation} \label{eta:averagephi}
   0 \leq \int_{B(x_0, \, 3R/4)} \varphi_\eps < 2\pi.
  \end{equation}
  The pre-Jacobian of~$\Q_\eps$ can the be written
  as~$j(\Q_\eps) = \rho_\eps^2 \, \nabla\varphi_\eps$.
  Then, Equation~\eqref{almostharmonic} reduces to
  \begin{equation} \label{almostharmonic-elliptic}
   -\div\left(\rho_\eps^2 \, \nabla\varphi_\eps\right) = \frac{\eps}{2} \,
   \partial_k\left(\M_\eps\times\partial_k\M_\eps\right)
   \qquad \textrm{in } B(x_0, \, 3R/4).
  \end{equation}
  The left-hand side of~\eqref{almostharmonic-elliptic}
  is uniformly elliptic, because of~\eqref{eta-nonzero}
  and the $L^\infty$-bound~\eqref{max-QM},
  while the right-hand side is uniformly bounded in $W^{-1,\infty}(\Omega)$,
  because of~\eqref{max-QM} and~\eqref{max-gradients}.
  Moreover, Proposition~\ref{prop:Q_bound} and~\eqref{gradQpolar}
  imply that
  \begin{equation*}
   \norm{\nabla\varphi_\eps}_{L^p(B(x_0, \, 3R/4))} \leq C_p, \qquad
   \mbox{for any } p\in [1, \, 2)
  \end{equation*}
  and for some constant~$C_p$
  that depends on~$p$, but not on~$\eps$. By applying
  Calderon-Zygmund estimates to Equation~\eqref{almostharmonic-elliptic},
  we obtain the interior estimate
  \begin{equation} \label{eta-phi}
   \norm{\nabla\varphi_\eps}_{L^p(B(x_0, \, 5R/8))} \leq C_p, \qquad
   \mbox{for any } p \in [1, \, +\infty),
  \end{equation}
  for some constant~$C_p$ that depends on~$p$, $x_0$, $R$
  but not on~$\eps$.
  (More precisely, the constant $C_p$ depends only on $\beta$, $\Omega$,
  $p$, $x_0$, $R$, $\Cpot$, and $\Qb$ but just through 	the $L^1(\partial \Omega)$- and the $L^2(\partial\Omega)$-norm of
	$\Qb \times \partial_\ttau \Qb$.)
 \end{step}

 \begin{step}[Bounds on the modulus]
  Let us consider the equation~\eqref{rhoeps} for~$\rho_\eps$.
  Thanks to~\eqref{eta-nonzero}, we can divide both sides by~$\rho_\eps$.
  Then, applying~\eqref{gradQpolar}, we obtain
  \begin{equation} \label{rhoepsbis}
   -\Delta\rho_\eps + 4 \rho_\eps\abs{\nabla\varphi_\eps}^2
   + \frac{1}{\eps^2}\left(\rho_\eps^2 - 1\right) \rho_\eps
   - \frac{1}{\eps}\sigma_\eps = 0 \qquad \textrm{in } B(x_0, \, 3R/4),
  \end{equation}
  where~$\sigma_\eps := \beta\rho_\eps^{-1} \, \Q_\eps\M_\eps\cdot\M_\eps$.
  Let~$B := B(x_0, 5R/8)$.
  Let~$\zeta\in C^\infty_{\mathrm{c}}(B)$
  be a cut-off function such that~$\zeta = 1$ in~$B(x_0, \, R/2)$
  and~$0 \leq \zeta \leq 1$ in~$B$. By testing Equation~\eqref{rhoepsbis}
  against~$(1 - \rho_\eps)\,\zeta^2$, integrating by parts,
  and applying Young's inequality,
  we obtain
  \begin{equation*}
   \begin{split}
    \frac{1}{2} \int_{B} \zeta^2 \abs{\nabla\rho_\eps}^2
    &\leq 2\int_{B}(1 - \rho_\eps)^2\abs{\nabla\zeta}^2
     + 4 \int_{B} \rho_\eps(1 - \rho_\eps) \abs{\nabla\varphi_\eps}^2\zeta^2\\
    &\hspace{3.67cm} + \frac{1}{\eps^2}
     \int_{B} \rho_\eps\left(\rho_\eps^2 - 1\right)(1 - \rho_\eps)\,\zeta^2
     - \frac{1}{\eps} \int_{B}\sigma_\eps(1 - \rho_\eps)\,\zeta^2\\
    &=: \mathrm{I}_\eps + \mathrm{II}_\eps + \mathrm{III}_\eps + \mathrm{IV}_\eps.
   \end{split}
  \end{equation*}
  We claim that all the terms at the right-hand side are bounded
  uniformly with respect to~$\eps$. For~$\mathrm{I}_\eps$,
  this is an immediate consequence of the~$L^\infty$-bound~\eqref{max-QM}.
  Boundedness of~$\mathrm{II}_\eps$ follows from~\eqref{max-QM} and~\eqref{eta-phi}.
  For~$\mathrm{III}_\eps$, we observe that
  \[
   \abs{\mathrm{III}_\eps} \leq \frac{1}{\eps^2}
     \int_{B} \frac{\rho_\eps\left(\rho_\eps^2 - 1\right)^2}{1 + \rho_\eps},
  \]
  then apply~\eqref{max-QM} and~\eqref{GLbound} to conclude
  that~$\abs{\mathrm{III}_\eps} \leq C$ for some~$\eps$-independent~$C$.
  Finally, we estimate~$\mathrm{IV}_\eps$ by applying the H\"older inequality:
  \[
   \abs{\mathrm{IV}_\eps} \leq 4 \left(\int_{B} \sigma_\eps^2\right)^{1/2}
    \left(\frac{1}{\eps^2} \int_{B} \left(\rho_\eps^2 - 1\right)^2\right)^{1/2}
  \]
  Since~$\sigma_\eps$ is bounded (because of~\eqref{max-QM}
  and~\eqref{eta-nonzero}), the estimate~\eqref{GLbound}
  implies that~$\abs{\mathrm{IV}_\eps}\leq C$. Therefore, we have proved that
  \begin{equation} \label{eta-rho}
   \int_{B} \zeta^2 \abs{\nabla\rho_\eps}^2 \leq C,
  \end{equation}
  where the constant~$C$ does not depend on~$\eps$
  (but does depend on~$x_0$, $R$, in general). Together with~\eqref{eta-phi},
  this inequality implies~\eqref{eta-energybound}.
  \qedhere
 \end{step}
\end{proof}

\begin{remark}\label{rk:dep-C*}
	As it can be easily checked from the proof above, the constant $C$ on the right-hand
	side of~\eqref{eta-rho} depends only on $x_0$, $R$, $\beta$, $\Omega$,
  	 $\Cpot$, and $\Qb$, but just through the $L^1(\partial \Omega)$- and the
  	 $L^2(\partial\Omega)$-norm of $\Qb \times \partial_\ttau \Qb$.
  	Keeping~\eqref{eta-phi} into account, it follows that the constant $C_*(x_0,\,R)$
  	on the right-hand side of~\eqref{eta-energybound} depends only on the quantities
  	listed
  	above.
\end{remark}

We now improve on Proposition~\ref{prop:clearing-out}
so to obtain the following refined bounds.
\begin{prop}\label{prop:improved-W1p-estimates}
	Let~$(\Q_\eps, \, \M_\eps)$ be a solution to~\eqref{EL-Q}--\eqref{EL-M},
	subject to boundary conditions either as in~\eqref{bc}--\eqref{hp:bc}
	or as in~\eqref{bcbis}--\eqref{hp:bcbis}.
	Assume that the condition~\eqref{hp:potential_bound} is satisfied.
	Let $B(x_0, \, R) \csubset \Omega$ be a ball such that
	$\abs{\Q_\eps} \geq 1/2$ on $B(x_0,\,R)$.
%
	Then,
	\begin{equation}\label{eq:improved-W1p-estimates}
		\norm{\nabla \Q_\eps}_{L^p(B(x_0,\,R/2)} +
		\eps^{-1}\norm{\abs{1-\abs{\Q_\eps}}}_{L^p(B(x_0,\,R/2))} \leq C_p(x_0,R).
	\end{equation}
\end{prop}

Before proceeding to the proof of
Proposition~\ref{prop:improved-W1p-estimates},
we point out the following observation.
\begin{remark}\label{rk:power-decay-GL}
	If $K$ is a compact set contained in $\{\abs{\Q_\eps} \geq 1/2\}$,
	then~\eqref{eq:improved-W1p-estimates} holds on $K$ (by a covering
	argument), with a constant $C_p(K)$ depending on $p$ and $K$ (but not on $\eps$).
	In particular, for any ball $B_R = B(x_0,\,R) \subseteq K$,
	applying the H\"{o}lder inequality, we get
	\begin{equation}
		\int_{B_R} \left( \frac{1}{2} \abs{\nabla \Q_\eps}^2 + \frac{1}{4\eps^2}\left( \abs{\Q_\eps}^2-1 \right)^2 \right)\,{\d}x \leq C_\alpha(K) R^\alpha
	\end{equation}
	for any $\alpha \in (0,2)$.
\end{remark}

Moreover, we state the following lemma, which replaces
\cite[Lemma~2]{BBH0}.
\begin{lemma}\label{lemma:bdd-GL-pot-p}
	Let $B = B(x_0,\,R)$ be a ball. Let $B'$ be the ball concentric with $B$
	and with radius $r < R$. Assume that $f$ is a continuous function that
	belongs to $L^p(B)$ for any $p \in [1,+\infty)$ and that
	$v \in C^2(B)$ is a non-negative function satisfying
	\begin{equation}\label{hp:potential-ineq-bis}
		-\frac{1}{2} \Delta v + \frac{1}{4\eps^2} v \leq f
		\qquad \mbox{in } B
	\end{equation}
	and
	\[
		\norm{v}_{L^\infty(B)} < +\infty.
	\]
	Then,
	\begin{equation}\label{eq:improved-potential-est-bis}
		\norm{v}_{L^p(B')} \lesssim \eps^2, \qquad \mbox{for any }
		p \in [1, +\infty),
	\end{equation}
	where the implicit constant in front of the right-hand side depends only on
	$p$, $r$, $R$, $\norm{v}_{L^\infty(B)}$, and $\norm{f}_{L^p(B)}$.
\end{lemma}

\begin{proof}
	We fix $\zeta \in C^\infty_c(B)$,
	a smooth cut-off function satisfying
	\begin{equation}\label{eq:zeta}
		0 \leq \zeta \leq 1 \,\,\mbox{ in } B, \qquad
		\zeta \equiv 1  \,\,\mbox{ in } B',
		\qquad \abs{\nabla \zeta} \lesssim \frac{1}{R-r}
		\,\,\mbox{ in } B.
	\end{equation}
	We observe that, by H\"{o}lder's inequality, it is enough
	to prove~\eqref{eq:improved-potential-est-bis} for any $p$ large enough.
	Therefore, we can assume that $p > 2$. We define
	\[
		\psi := \zeta^{2p} v^{p-1}.
	\]
	The function $\psi$ belongs to $C^1_c(B)$ and we can
	test the inequality~\eqref{hp:potential-ineq-bis} against $\psi$.
	Since $\psi \geq 0$, this yields
	\begin{equation}\label{eq:test-psi}
		\int_B \frac{1}{2} \nabla \psi \cdot \nabla v \,{\d}x
		+ \frac{1}{4\eps^2} \int_B \zeta^{2p} v^p \,{\d}x \leq
		\int_B  \zeta^{2p} v^{p-1} f\,{\d}x.
	\end{equation}
	Since $0 \leq \zeta \leq 1$, we have $\zeta^{2p} \leq \zeta^{2(p-1)}$ pointwise
	and therefore, by H\"{o}lder's inequality, 
	\begin{equation}\label{eq:est-RHS-test-psi}
		\int_B  \zeta^{2p} v^{p-1} f\,{\d}x
		\leq \int_B  \zeta^{2(p-1)} v^{p-1} \abs{f}\,{\d}x
		\leq \norm{f}_{L^p(B)}
		\left( \int_B \zeta^{2p} v^p \,{\d}x \right)^{1-1/p}. 
	\end{equation}
	On the other hand, we have
	\[
		\int_B \nabla \psi \cdot \nabla v \,{\d}x =
		\int_B \left( 2p \zeta^{2p-1} v^{p-1} \nabla \zeta \cdot \nabla v + (p-1) \zeta^{2p} v^{p-2} \abs{\nabla v}^2 \right)\,{\d}x,
	\]
	so that, by ~\eqref{eq:test-psi}, we obtain
	\begin{equation}\label{eq:test-psi-2}
	\begin{split}
		\int_B &\frac{1}{2}(p-1) \zeta^{2p} v^{p-2} \abs{\nabla v}^2 \,{\d}x
		+ \frac{1}{4\eps^2} \int_B \zeta^{2p} v^p\,{\d}x \\
		&\leq \norm{f}_{L^p(B)} \left( \int_B \zeta^{2p} v^p \,{\d}x \right)^{1-1/p}
		+ \int_B  p \zeta^{2p-1} v^{p-1} \abs{ \nabla \zeta} \abs{\nabla v} \,{\d}x.
		\end{split}
	\end{equation}
	Since $p > 2$, we may rewrite the integrand in the
	last term on the right-hand side as
	\[
		p \zeta^{2p-1} v^{p-1} \abs{\nabla \zeta}\abs{\nabla v} =
		p \left( \zeta^{p-1} v^{p/2} \abs{\nabla \zeta} \right)
		\left(\zeta^{p} v^{p/2-1} \abs{\nabla v} \right)
	\]
	Thus, by applying Young's inequality and using~\eqref{eq:zeta},
	we obtain the bound
	\begin{equation}\label{eq:Young-v-zeta}
		p \zeta^{2p-1} v^{p-1} \abs{\nabla \zeta}\abs{\nabla v}
		\leq \frac{C'_{p}}{(R-r)^2} \zeta^{2(p-1)}{v}^{p}
		+ \frac{1}{2}(p-1) \zeta^{2p} v^{p-2} \abs{\nabla v}^2,
	\end{equation}
	where $C'_{p}$ is a positive constant that depends only on $p$.
	Plugging~\eqref{eq:Young-v-zeta} into~\eqref{eq:test-psi-2}
	and recalling that $v^p \leq \norm{v}_{L^\infty(B)} v^{p-1}$ pointwise, we arrive at
	\[
		\frac{1}{4 \eps^2} \int_B \zeta^{2p} v^p \,{\d}x \leq
		C_p\left(r,\,R,\,\norm{v}_{L^\infty(B)},\,\norm{f}_{L^p(B)}\right)
		\left[ \left(\int_B \zeta^{2p} v^p\,{\d}{x} \right)^{1-1/p} + \int_B \zeta^{2(p-1)} v^{p-1} \,{\d}x \right],
	\]
	whence, by a further application of H\"{o}lder's inequality on the last
	term on the right-hand side, we obtain
	\[
	\frac{1}{4 \eps^2} \int_B \zeta^{2p} v^p \,{\d}x
	\leq C_p\left(r,\,R,\,\norm{v}_{L^\infty(B)},\,\norm{f}_{L^p(B)}\right)
	\left(\int_B \zeta^{2p} v^p \,{\d}{x} \right)^{1-1/p},
	\]
	for a possibly larger constant $C_p\left(r,\,R,\,\norm{v}_{L^\infty(B)},\,\norm{f}_{L^p(B)}\right)$
	(still depending only on $p$, $R$, and $\norm{f}_{L^p(B)}$),
	so that
	\[
		\left( \int_B \zeta^{2p} v^p \,{\d}x \right)^{1/p}
		\leq C_p\left(r,\,R,\,\norm{v}_{L^\infty(B)},\,\norm{f}_{L^p(B)}\right) \eps^2,
	\]
	and, recalling that $\zeta \equiv 1$ in $B'$, the
	desired inequality~\eqref{eq:improved-potential-est-bis} follows.
\end{proof}

\begin{remark}
	If we happen to know that $f \in L^\infty(B)$,
	then, instead of the above lemma,
	we may apply \cite[Lemma~2]{BBH0} and
	obtain the stronger conclusion
	\[
		\norm{v}_{L^\infty(B')} \lesssim \eps^2.
	\]
\end{remark}

\begin{proof}[Proof of Proposition~\ref{prop:improved-W1p-estimates}]
The proof is inspired by classical arguments in \cite{BOS-ARMA, BOS-Pisa}
and, to make it more transparent, it is split into several steps.
\setcounter{step}{0}
\begin{step}
	Set $\rho_\eps := \abs{\Q_\eps}$ and let
	$\varphi_\eps : B\left(x_0, {3R}/{4}\right) \to \R$ be the angle
	function given by~\eqref{Qpolar}. By assumption, $\rho_\eps \geq 1/2$
	on $B(x_0,\,3R/4)$ and then, by~\eqref{eta-phi}, we have
	$\norm{\nabla \varphi_\eps}_{L^p\left(B\left(x_0,\frac{5}{8}R\right)\right)} \leq C_p$ for any $p < +\infty$.
	From now on, we drop the subscript $\eps$, for ease of notation.
\end{step}

\begin{step}
	Writing the equation for $\rho$ in the form
	\begin{equation}\label{eq:Delta-rho}
		-\Delta \rho + 4 \rho \abs{\nabla \varphi}^2
		+ \frac{1}{\eps^2} \left( \rho^2 - 1 \right) \rho
		= \frac{\sigma}{\eps}
		\qquad \mbox{in } B\left(x_0, \frac{3R}{4}\right),
	\end{equation}
	where $\sigma := \frac{\beta \Q \M \cdot \M}{\rho}$ is bounded, and
	multiplying it by $(\rho - 1)$, we obtain
	\begin{equation}\label{eq:rho-1}
		(1-\rho) \Delta \rho + \frac{1}{\eps^2}\left(\rho^2-1\right)(\rho-1)
		\rho = 4 (1-\rho)\rho \abs{\nabla \varphi}^2
		+ \frac{\sigma}{\eps}(\rho-1).
	\end{equation}
	We manipulate the terms at the left-hand side of the above equation
	as follows:
	\begin{equation}\label{eq:rho-2}
		(1-\rho)\Delta \rho = \div\left( (1-\rho) \nabla \rho \right) +
		\abs{\nabla \rho}^2 = -\frac{1}{2}\Delta\left( (\rho-1)^2 \right)
		+ \abs{\nabla \rho}^2
	\end{equation}
	and
	\[
		\frac{1}{\eps^2}\left( \rho^2 - 1 \right)(\rho-1) \geq
		\frac{1}{2 \eps^2}\left(\rho-1\right)^2,
	\]
	while for the right-hand side we observe that, by Young's inequality,
	\[
		\abs{\frac{\sigma}{\eps}(\rho-1)} \leq
		\frac{\left(\rho-1\right)^2}{4\eps^2} + \sigma^2.
	\]
	Thus, we obtain
	\begin{equation}\label{eq:potential-ineq}
		-\frac{1}{2}\Delta\left( \left(\rho-1\right)^2 \right)
		+\frac{1}{4\eps^2}\left(\rho-1 \right)^2 \leq \widetilde{\sigma},
	\end{equation}
	where
	$\widetilde{\sigma} := \sigma^2 + 4 (1-\rho) \rho \abs{\nabla \varphi}^2$.
	Since $\sigma$ and $\rho$ are (uniformly) bounded in $L^\infty(\Omega)$
	and
	$\varphi$ satisfies~\eqref{eta-phi} in $B\left(x_0, \frac{5R}{8}\right)$,
	it follows that
	\begin{equation}\label{eq:sigma-tilde}
		\norm{\widetilde{\sigma}}_{L^p\left( B\left(x_0, \frac{5R}{8}\right) \right)} \leq C_p, \qquad
		\mbox{for any } p < +\infty,
	\end{equation}
	where, by~\eqref{eta-phi}, $C_p$ depends only on $p$, $x_0$, $R$,
	$\beta$, $\Omega$, $\Cpot$,
	and the~$L^1(\partial\Omega)$- and~$L^2(\partial\Omega)$-norms
	of~$\Qb \times \partial_\ttau \Qb$.

	We now claim that
	\begin{equation}\label{eq:improved-potential-est}
		\norm{\rho-1}_{L^p\left( B\left(x_0, \frac{9R}{16} \right) \right)}
		\leq C_p(x_0, R) \eps, \qquad \mbox{for any } p < +\infty,
	\end{equation}
	where $C_p(x_0, R)$ depends only on $p$, $x_0$, $R$, the coupling
	constant $\beta$,
	$\Omega$, $\Cpot$, and $\Qb$ (but just through 	the $L^1(\partial \Omega)$- and the $L^2(\partial\Omega)$-norm of $\Qb \times \partial_\ttau \Qb$).

	\vskip5pt

	\noindent\emph{Proof of~\eqref{eq:improved-potential-est}.}
	Let
	\[
		B := B\left(x_0, \frac{5}{8}R \right), \qquad
		B' := B\left(x_0, \frac{9}{16}R \right).
	\]
	Recasting~\eqref{eq:potential-ineq} as an inequality for the
	non-negative function $v := (\rho-1)^2$ leads to
	\begin{equation}\label{eq:potential-ineq-bis}
		-\frac{1}{2} \Delta v + \frac{1}{4\eps^2} v \leq \widetilde{\sigma}
		\qquad \mbox{in } B.
	\end{equation}
	Note that $\norm{v}_{L^\infty(\Omega)}$ is bounded only in terms of $\beta$,
	by~\eqref{max-QM}.
	From Lemma~\ref{lemma:bdd-GL-pot-p}, we have
	\[
		\norm{v}_{L^p(B')} \lesssim \eps^2, \qquad \mbox{for any }
		p \in [1, +\infty),
	\]
	where the implicit constant in front of the right-hand side depends only on
	$p$, $R$, $\beta$, and $\norm{\widetilde{\sigma}}_{L^p(B)}$.
	Since the latter norm is controlled
	by a constant depending only on $x_0$, $R$, $\beta$, $\Omega$,
  	 $\Cpot$, and $\Qb$ (but just through the $L^1(\partial \Omega)$- and the $L^2(\partial\Omega)$-norm of $\Qb \times \partial_\ttau \Qb$.),
	going back from $v$ to $(\rho-1)^2$, we
	obtain~\eqref{eq:improved-potential-est},
	where the implicit constant in front of the right-hand side depends only on
	$p$, $x_0$, $R$, and $\beta$, as claimed.
\end{step}

\begin{step}[$L^p$-estimates for $\nabla \rho$ and conclusion]
	Writing~\eqref{EL-Q} in the form
	\[
		-\Delta \Q = \frac{1}{\eps}\left( \frac{(1-\abs{\Q}^2)}{\eps} \Q + \beta\left( \M\otimes \M - \frac{\abs{\M}^2}{2}\Id \right) \right),
	\]
	the global $L^\infty$-bounds on $\Q$ and $\M$ provided by Lemma~\ref{lemma:max}
	and~\eqref{eq:improved-potential-est} tell us that the term in round
	brackets above is bounded in $L^p\left(B\left(x_0, \frac{9R}{16}\right) \right)$.
	Thus, by standard elliptic regularity estimates,
	\begin{equation}\label{eq:W2p-est-Qeps}
	\begin{split}
		\norm{\Q}_{W^{2,p}\left(B\left(x_0,\frac{R}{2}\right)\right)}
		&\lesssim
		\norm{\Delta \Q}_{L^p\left(B\left(x_0, \frac{9R}{16} \right) \right)}
		+ \norm{\Q}_{L^p\left(B\left(x_0, \frac{9R}{16} \right) \right)} \\
		&\lesssim C_p(x_0,R) \eps^{-1},
	\end{split}
	\end{equation}
	where $C_p(x_0,R)$ depends only on $p$, $x_0$, $R$, and the quantities listed in Remark~\ref{rk:dep-C*}. By interpolation,
	\begin{equation}\label{eq:nabla-rho-bdd}
	\begin{split}
		\norm{\nabla \rho}_{L^p\left(B\left(x_0,\frac{R}{2}\right)\right)}
		&\lesssim
		\norm{\Q}_{W^{2,p}\left(B\left(x_0,\frac{R}{2}\right)\right)}^{1/2}
		\norm{1-\rho}_{L^p\left(B\left(x_0,\frac{R}{2}\right)\right)}^{1/2} \\
		&\lesssim C_p(x_0,R) \eps^{-1/2} \eps^{1/2} = C_p(x_0,\,\beta,\,R),
	\end{split}
	\end{equation}
	whence
	\[
	\norm{\nabla \Q}_{L^p\left(B\left(x_0,\frac{R}{2}\right)\right)}
	\lesssim
		\norm{\nabla \rho}_{L^p\left(B\left(x_0,\frac{R}{2}\right)\right)} +
		\norm{\nabla \varphi}_{L^p\left(B\left(x_0,\frac{R}{2}\right)\right)}
	\leq C_p(x_0,\,\beta,\, R),
	\]
	where, once again, $C_p(x_0,\,\beta,\,R)$ depends only on
	$p$, $x_0$, $\beta$, $R$, and the quantities list in Remark~\ref{rk:dep-C*}.
	The conclusion follows.
	\qedhere
\end{step}
\end{proof}

\begin{corollary}\label{cor:linfty-est-Qeps}
	If $K \subset \Omega$ is any compact set such that
	$\abs{\Q_\eps} \geq 1/2$ in $K$ for any $\eps$ small enough,
	then
	\begin{gather}
		\norm{ \abs{\Q_\eps} - 1 }_{L^\infty(K)} \leq C_\alpha \eps^{1-\alpha} \label{eq:local-Linfty-est-Q-1} \\
		\norm{\nabla \Q_\eps}_{L^\infty(K)} \leq C_\alpha \eps^{-\alpha}\label{eq:local-Linfty-est-Q},
	\end{gather}
	for any $\alpha \in (0,1]$, where the number $C_\alpha$ depends only on
	$\alpha$, $\beta$, $K$, and the quantities listed in Remark~\ref{rk:dep-C*}.
\end{corollary}

\begin{proof}
	By a standard covering argument, it suffices to
	prove~\eqref{eq:local-Linfty-est-Q-1},~\eqref{eq:local-Linfty-est-Q} for
	a ball $B = B(x_0,\,R)$, with constants depending only on $\alpha$, $\beta$,
	$x_0$, $R$, and the quantities listed in Remark~\ref{rk:dep-C*}.
	Given $\alpha > 0$ and letting $p := \max\{2,2/\alpha\}$, by
	the Gagliardo-Nirenberg interpolation inequality we have
	\[
		\norm{\abs{\Q_\eps}-1}_{L^\infty(B)} \lesssim
		\norm{\Q_\eps}^{2/p}_{W^{1,p}(B)}
		\norm{\abs{\Q_\eps}-1}^{1-2/p}_{L^p(B)}.
	\]
	Then, by~\eqref{eq:improved-W1p-estimates} we obtain
	\[
		\norm{\abs{\Q_\eps} -1 }_{L^\infty(B)} \leq
		C_\alpha \eps^{1-\alpha},
	\]
	where the constant $C_\alpha$ on the right-hand side depends only on $\alpha$,
	$\beta$, $x_0$, and $R$.

	Similarly, again by the Gagliardo-Nirenberg interpolation inequality,
	we have
	\[
		\norm{\nabla \Q_\eps}_{L^\infty(B)} \lesssim
		\norm{\Q_\eps}^{2/p}_{W^{2,p}(B)}
		\norm{\nabla \Q_\eps}^{1-2/p}_{L^p(B)}
	\]
	and then, by~\eqref{eq:improved-W1p-estimates} and~\eqref{eq:W2p-est-Qeps},
	we obtain
	\[
		\norm{\nabla \Q_\eps}_{L^\infty(B)} \leq C_p \eps^{-2/p}
		\leq C_\alpha \eps^{-\alpha},
	\]
	where the constant $C_\alpha$ in the right-hand side depends only on $\alpha$,
	$\beta$, $x_0$, and $R$.
\end{proof}

\subsection{An estimate for the energy of~$\M_\eps$}

\begin{prop} \label{prop:M_bound}
 Assume that~\eqref{hp:potential_bound} holds.
 There exists a constant~$C>0$,
 depending only on $\beta$, ${\rm C}_{\rm pot}$, and the boundary data,
 such that, for any~$\eps$ small enough, there holds
 \begin{equation}\label{eq:M_bound}
  \eps \int_\Omega \abs{\nabla\M_\eps}^2 \leq C.
 \end{equation}
\end{prop}
\begin{proof}
 The strategy for the proof is to test Equation~\eqref{EL-M}
 against a suitable test function, then apply the assumption~\eqref{hp:potential_bound} on the potential
 and the estimate for~$\nabla\Q_\eps$ given by Proposition~\ref{prop:Q_bound}.
 Before going to the details, let us recall some notation.
 Let~$\ell$ be defined as in~\eqref{ell_eps}.
 We recall from Lemma~\ref{lemma:f-fixedQ} that,
 for any~$x\in\Omega$ such that~$\Q_\eps(x)\neq 0$,
 the function~$\ell(\Q_\eps(x), \, \cdot\,)$ has exactly
 two minimisers in~$\R^2$, that is
 \begin{equation} \label{Mpm-N}
   \M_{\pm}(x) := \pm\left(\sqrt{2}\beta\rho_\eps(x)
    + 1\right)^{1/2} \n_\eps(x)
 \end{equation}
 where~$\rho_\eps(x) := \abs{\Q_\eps(x)}$
 and~$\n_\eps(x)$ is a unit eigenvector corresponding to the (unique)
 positive eigenvalue of~$\Q_\eps(x)$.
 Let~$\Sigma_x := \Sigma(\Q_\eps(x)):=\{\M_{\pm}(x)\}\subseteq\R^2$.
 For any~$\N\in\R^2$ such that~$\dist(\N, \, \Sigma_x) \leq 1$,
 let~$\pi_x(\N) := \pi(\Q_\eps(x), \, \N)$ be the projection
 of~$\N$ to~$\Sigma_x$, as defined in~\eqref{projection}.

 \setcounter{step}{0}

 \begin{step}[Construction of~$\N_\eps$]
  Let~$\delta \in (0, \, 1/2)$ be a fixed parameter,
  to be chosen later (independently of~$\eps$).
  Let~$\xi$, $\zeta\colon [0, \, +\infty)\to [0, \, 1]$ be
  smooth functions such that
  \begin{align}
   \xi(t) = 0 \ \textrm{ if } 0 \leq t \leq \frac{1}{2}, \qquad
    \xi(t) = 1 \ \textrm{ if } t \geq \frac{3}{4} \label{xi} \\
   \zeta(t) = 1 \ \textrm{ if } 0 \leq t \leq \delta, \qquad
    \zeta(t) = 0 \ \textrm{ if } t \geq 2\delta \label{zeta}
  \end{align}
  For~$x\in\Omega$, we define
  \[
   \eta_\eps(x) := \xi\left(\rho_\eps(x)\right)
    \zeta\left(\dist(\M_\eps(x), \, \Sigma_x)\right)
  \]
  and
  \begin{equation} \label{N}
   \N_\eps(x) := \left(1 - \eta_\eps(x)\right) \M_\eps(x)
     + \eta_\eps(x) \, \pi_x(\M_\eps(x))
  \end{equation}
  The value~$\eta_\eps(x)$ is non-zero only if~$\rho_\eps(x)\geq 1/2$
  and~$\dist(\M_\eps(x), \, \Sigma_x)\leq 2\delta < 1$,
  therefore $\pi_x(\M_\eps(x))$
  is well-defined if~$\eta_\eps(x) \neq 0$.

  We use~$\N_\eps$ as a test function in Equation~\eqref{EL-M}.
  Assume for a moment we are imposing mixed boundary conditions,
  as in~\eqref{bcbis}, \eqref{hp:bcbis}.
  Then, by testing Equation~\eqref{EL-M}
  against~$\eps\M_\eps - \eps\N_\eps$, integrating by parts,
  and applying Young's inequality, we obtain
  \begin{equation} \label{N-testEL}
   \frac{\eps}{2}\int_{\Omega} \abs{\nabla\M_\eps}^2
    \leq \int_{\Omega} \left(\frac{\eps}{2}\abs{\nabla\N_\eps}^2
    + \frac{1}{\eps} \nabla_{\M} \ell(\Q_\eps, \, \M_\eps)\cdot(\N_\eps - \M_\eps)\right) \!,
  \end{equation}
  where~$\nabla_{\M}$ denotes the (partial) gradient with respect
  to the~$\M$-variable. Equation~\eqref{N-testEL} remains true if we
  are considering Dirichlet boundary conditions as
  in~\eqref{bc}, \eqref{hp:bc}. Indeed, the assumption~\eqref{hp:bc}
  implies that the boundary data satisfy~$\abs{\Qb(x)} = 1$,
  $\Mb(x)\in\Sigma_x$ and, hence, $\eta_\eps(x) = 1$,
  $\N_\eps(x) = \pi_x(\Mb(x)) = \Mb(x)$ for each~$x\in\partial\Omega$.
  As a consequence, $\eps\M_\eps - \eps\N_\eps = 0$ on~$\partial\Omega$
  and~\eqref{N-testEL} follows, irrespective of whether we
  are imposing Dirichlet or mixed boundary conditions.
 \end{step}

 \begin{step}[Estimate on~$\nabla\N_\eps$]
  We claim that
  \begin{equation}  \label{Ngradient}
   \eps \int_{\Omega} \abs{\nabla\N_\eps}^2
   \lesssim \int_\Omega \left(\eps\abs{\nabla\Q_\eps}^2
   + \frac{1}{\eps^2} f_\eps(\Q_\eps, \, \M_\eps) \right) \! .
  \end{equation}
  Here (and in the rest of this proof, wherever we use the notation~$\lesssim$)
  the implicit constant in front of the right-hand side
  depends on~$\delta$, but not on~$\eps$.
  In order to prove this claim, we consider the (open) set
  \begin{equation} \label{E-N}
   E_\eps := \left\{x\in\Omega\colon \rho_\eps(x) > \frac{3}{4},
    \ \dist(\M_\eps(x), \, \Sigma_x) < \delta  \right\} \! .
  \end{equation}
  By Lemma~\ref{lemma:ell1}, there exists
  a constant~$c_1(\delta) > 0$, depending on~$\delta$ and~$\beta$
  but not on~$\eps$, such that
  $f_\eps(\Q_\eps, \, \M_\eps)\geq c_1(\delta) \, \eps$
  on~$\Omega\setminus E_\eps$.
  As a consequence, we have
  \begin{equation} \label{Emeasure-N}
   \abs{\Omega\setminus E_\eps}
    \leq \frac{1}{c_1(\delta)\,\eps} \int_\Omega f_\eps(\Q_\eps, \, \M_\eps).
  \end{equation}
  On the other hand, the bound~\eqref{max-gradients}
  implies the pointwise estimate~$\abs{\nabla\N_\eps}\lesssim\eps^{-1}$.
  Therefore, from~\eqref{Emeasure-N} we deduce
  \begin{equation} \label{N-grad-offE}
   \eps \int_{\Omega\setminus E_\eps} \abs{\nabla \N_\eps}^2
    \lesssim \frac{{\abs{\Omega \setminus E_\eps}}}{\eps}
    \lesssim \frac{1}{\eps^2} \int_\Omega f_\eps(\Q_\eps, \, \M_\eps).
  \end{equation}
  It remains to estimate~$\abs{\nabla\N_\eps}$ on~$E_\eps$.
  For any~$x\in E_\eps$, we have~$\eta_\eps(x)=1$
  (because of~\eqref{xi}, \eqref{zeta}) and~$\N_\eps(x) = \pi_x(\M_\eps(x))$.
  By definition of~$\pi_x$ (see~\eqref{projection}),
  $\N_\eps(x)$ is a nonzero eigenvector of~$\Q_\eps(x)$
  corresponding to the (unique) positive eigenvalue.
  We claim that~$\N_\eps$ is smooth in the open set~$E_\eps$.
  Indeed, let~$B\subseteq E_\eps$ be an arbitrary open ball.
  The spectral theorem implies that~$\Q_\eps$ can be written in the form
  \[
   \Q_\eps = \frac{\rho_\eps}{\sqrt{2}} \left(\n_\eps\otimes\n_\eps
    - \m_\eps\otimes\m_\eps\right) \qquad \textrm{in } B,
  \]
  where~$(\n_\eps, \, \m_\eps)$ is a smooth orthonormal
  frame of eigenvectors for~$\Q_\eps$, and (see~\eqref{Mpm-N})
  \begin{equation} \label{Nsign}
   \N_\eps(x) = \pi_x(\M_\eps(x))
   = \pm\left(\sqrt{2}\beta\rho_\eps(x) + 1\right)^{1/2} \n_\eps(x)
   \qquad \textrm{at any point } x\in B,
  \end{equation}
  By definition of~$\pi_x$, the sign of~$\N_\eps(x)\cdot\n_\eps(x)$
  is the same as the sign of~$\M_\eps(x)\cdot\n_\eps(x)$
  and is nonzero, because~$\dist(\M_\eps(x), \, \Sigma_x)\leq 1/2$.
  Therefore, the sets~$B^+ := \{x\in B\colon \N_\eps(x)\cdot\n_\eps(x) > 0\}$
  and~$B^- := \{x\in B\colon \N_\eps(x)\cdot\n_\eps(x) < 0\}$
  are open, disjoint, and their union contains~$B$.
  As~$B$ is connected, it follows that either~$B = B^+$
  or~$B = B^-$ --- that is, the sign in~\eqref{Nsign} is constant.
  Either way, $\N_\eps$ is smooth in~$B$ and there holds
  \[
   \abs{\nabla\N_\eps} \lesssim \abs{\nabla\rho_\eps}
    + \abs{\nabla\n_\eps} \lesssim \abs{\nabla\Q_\eps}
  \]
  pointwise in~$B$. As the ball~$B\subseteq E_\eps$
  is chosen arbitrarily, we have
  \begin{equation} \label{N-grad-onE}
   \int_{E_\eps} \abs{\nabla \N_\eps}^2
    \lesssim \int_{\Omega} \abs{\nabla \Q_\eps}^2 .
  \end{equation}
  Now, \eqref{Ngradient} follows from~\eqref{N-grad-offE}
  and~\eqref{N-grad-onE}.
 \end{step}

 \begin{step}[Estimate on the potential term]
  We claim that
  \begin{equation} \label{Npotential}
   \frac{1}{\eps} \int_{\Omega} \nabla_{\M}
    \ell(\Q_\eps, \, \M_\eps)\cdot(\N_\eps - \M_\eps)
    \lesssim \frac{1}{\eps^2} \int_\Omega f_\eps(\Q_\eps, \, \M_\eps).
  \end{equation}
  To this end, we consider again the
  set~$E_\eps$ defined by~\eqref{E-N} above.
  The $L^\infty$-bound~\eqref{max-QM} and
  the estimate~\eqref{N-grad-offE} imply
  \begin{equation} \label{Npot-offE}
   \frac{1}{\eps} \int_{\Omega\setminus E_\eps} \nabla_{\M}
    \ell(\Q_\eps, \, \M_\eps)\cdot(\N_\eps - \M_\eps)
   \lesssim \frac{\abs{\Omega\setminus E_\eps}}{\eps}
   \lesssim \frac{1}{\eps^2} \int_\Omega f_\eps(\Q_\eps, \, \M_\eps).
  \end{equation}
  To estimate the contribution from~$E_\eps$,
  we observe that
  $\N_\eps(x) = \pi_x(\M_\eps(x))$
  at each point~$x\in E_\eps$. As a consequence,
  Lemma~\ref{lemma:ell2} and the $L^\infty$-bound~\eqref{max-QM} imply
  \begin{equation} \label{Npot-onE}
   \frac{1}{\eps} \int_{E_\eps} \nabla_{\M}
    \ell(\Q_\eps, \, \M_\eps)\cdot(\N_\eps - \M_\eps)
    \leq 0,
  \end{equation}
  so long as we choose~$\delta$ small enough.
  (We can choose a suitable value of~$\delta$ that depends
  only on~$\beta$ and the~$L^\infty$-norm
  of~$\Q_\eps$; however, the latter is bounded uniformly
  with respect to~$\eps$, so~$\delta$ can be chosen
  independently of~$\eps$.)
  Combining~\eqref{Npot-offE} with~\eqref{Npot-onE},
  we obtain~\eqref{Npotential}.
 \end{step}

 \begin{step}[Conclusion]
  From~\eqref{N-testEL}, \eqref{Ngradient} and~\eqref{Npotential},
  we deduce
  \begin{equation} \label{Nfinal}
   \eps \int_\Omega \abs{\nabla\M_\eps}^2
    \lesssim \int_\Omega \left(\eps\abs{\nabla\Q_\eps}^2
    + \frac{1}{\eps^2} f_\eps(\Q_\eps, \, \M_\eps) \right) \! ,
  \end{equation}
  The terms at the right-hand side are bounded
  uniformly with respect of~$\eps$, because of
  Proposition~\ref{prop:Q_bound} and
  Assumption~\eqref{hp:potential_bound}.
  The lemma follows.
  \qedhere
 \end{step}
\end{proof}

\begin{remark} \label{rk:M_bound}
 We stress that the estimate~\eqref{Nfinal}
 is completely independent of the assumption~\eqref{hp:potential_bound},
 although it relies on the assumptions~\eqref{hp:bc}, \eqref{hp:bcbis}
 on the boundary data and the maximum principle, Lemma~\ref{lemma:max}
 (which is itself independent of~\eqref{hp:potential_bound},
 see Remark~\ref{rk:max}).
\end{remark}

\begin{remark}
	More precisely, like in Proposition~\ref{prop:Q_bound},
	the constant $C$ on the right-hand side of~\eqref{eq:M_bound}
	depends on the boundary data just through
	the~$L^1(\partial\Omega)$- and~$L^2(\partial\Omega)$-norms
	of~$\Qb\times\partial_{\ttau}\Qb$.
\end{remark}

From Proposition~\ref{prop:M_bound}, we immediately draw the following
consequence.
\begin{corollary}\label{cor:est-prejac-M}
	Assume that~\eqref{hp:potential_bound} holds. Then, for any
	$\eps > 0$ small enough,
	\begin{equation}\label{eq:Lp-est-j(M)}
		\norm{j(\M_\eps)}_{L^p(\Omega)} \lesssim \eps^{-1/2}
	\end{equation}
	for any $p \in [1,2]$. As a consequence,
	\begin{equation}\label{eq:negative-sob-conv-j(M)}
		\eps \norm{\div j(\M_\eps)}_{W^{-1,p}(\Omega)} \lesssim \eps^{1/2}
	\end{equation}
	for any $p \in [1,2]$ and any $\eps$ small enough.
\end{corollary}

\begin{proof}
	Let $p \in [1,2]$ be arbitrary and assume $\eps$ is so small
	that~\eqref{eq:M_bound} holds.
	Since $j(\M_\eps) = \M_\eps \times \nabla \M_\eps$, by the uniform bound
	$\norm{\M_\eps}_{L^\infty(\Omega)} \lesssim 1$, the
	H\"{o}lder inequality, and \eqref{eq:M_bound} we have
	\[
	\begin{split}
		\int_\Omega \abs{j(\M_\eps)}^p\,{\d}x &= \int_\Omega \abs{\M_\eps \times \nabla \M_\eps}^p\,{\d}x
		\leq \int_\Omega \abs{\M_\eps}^p \abs{\nabla \M_\eps}^p \,{\d}x \\
		&\lesssim \left(\int_\Omega \abs{\nabla \M_\eps}^2\right)^{p/2}
		\stackrel{\eqref{eq:M_bound}}{\lesssim} \left(\frac{1}{\eps}\right)^{p/2},
	\end{split}
	\]
	which yield~\eqref{eq:Lp-est-j(M)}. Now, pick arbitrarily
	$\Psi \in W^{1,p'}_0(\Omega)$. Then, for any $\eps$ as before,
	\[
	\begin{split}
		\abs{\left\langle \eps \div(j(\M_\eps)),\, \Psi \right\rangle_{W^{-1,p}(\Omega), W^{1,p'}_0(\Omega)}}
		&= \eps \abs{\int_\Omega j(\M_\eps) \cdot \nabla \Psi\, {\d}x} \\
		&\leq \eps \norm{j(\M_\eps)}_{L^p(\Omega)} \norm{\nabla \Psi}_{L^{p'}(\Omega)} \\
		&\lesssim \eps^{1/2} \norm{\Psi}_{W^{1,p'}(\Omega)},
	\end{split}
	\]
	which implies~\eqref{eq:negative-sob-conv-j(M)}.
\end{proof}

\subsection{Summary of the energy estimates}

Combining assumption~\eqref{hp:potential_bound}
with Lemma~\ref{lemma:feps}, Proposition~\ref{prop:Q_bound},
Proposition~\ref{prop:clearing-out}, and
Proposition~\ref{prop:M_bound}, we obtain the following
lemma, summarising the various energy estimates at
our disposal.

\begin{theorem}\label{lemma:energy-est}
	Let $\{(\Q_\eps,\,\M_\eps)\}$ be a sequence of
	critical points of the functional $\F_\eps$, subject to boundary
	conditions as in~\eqref{bc},~\eqref{hp:bc} or
	as in~\eqref{bcbis},~\eqref{hp:bcbis}. Assume
	that~\eqref{hp:potential_bound} holds. Then,
	\begin{align}
		&\F_\eps(\Q_\eps,\,\M_\eps) \lesssim \abs{\log\eps}, \label{eq:log-bound} \\
		&\int_\Omega \abs{\nabla \Q_\eps}^2 \lesssim \abs{\log\eps}, \\
		&\int_\Omega \eps \abs{\nabla \M_\eps}^2 \lesssim 1, \\
		&\frac{1}{\eps^2}\int_\Omega f_\eps(\Q_\eps,\,\M_\eps) \lesssim 1, \label{eq:negligible-pot}\\
		&\frac{1}{\eps^2}\int_\Omega \left(1-\abs{\Q_\eps}^2\right)^2 \lesssim 1 \label{eq:bdd-GL-pot},
	\end{align}
	where the implicit constants on the right-hand side depend only
	on $\beta$, $\Cpot$, and
	the $L^1(\partial \Omega)$- and the $L^2(\partial\Omega)$-norm of
	$\Qb \times \partial_\ttau \Qb$.

	Moreover, on any ball $B = B(x_0,\,R) \csubset \Omega$ on which
	$\abs{\Q_\eps} \geq 1/2$,
	there hold
	\begin{gather}
		\int_B \abs{\nabla \Q_\eps}^2\,{\d}x \leq C(x_0,\,R,\,\beta,\,\Cpot), \\
		\F_\eps(\Q_\eps,\,\M_\eps;\,B) \leq C(x_0,\,R,\,\beta,\,\Cpot) \label{eq:bounded-F-clearing-out},
	\end{gather}
	where the constant
	$C(x_0,\,R,\,\beta,\,\Cpot,\,\Qb)$ depends only on $x_0$, $R$,
	$\beta$, $\Cpot$, and the $L^1(\partial \Omega)$- and the $L^2(\partial\Omega)$-norm of $\Qb \times \partial_\ttau \Qb$.
	Consequently, if $K \subset \Omega$ is any compact set such
	that $\abs{\Q_\eps} \geq 1/2$ on $K$ for any $\eps$ small enough,
	then there holds
	\begin{equation}
		\lim_{\eps \to 0}\frac{\F_\eps(\Q_\eps,\,\M_\eps;\,K)}{\abs{\log\eps}}
		= 0.
	\end{equation}
\end{theorem}

\begin{remark}
	It follows from~\eqref{eq:log-energy-q} below and classical results in Ginzburg-Landau theory
	(\cite{Jerrard, Sandier} that, under
	both~\eqref{bc}--\eqref{hp:bc} and~\eqref{hp:bcbis}, as soon as $\deg(\Qb,\,\partial\Omega) \neq 0$,
	the energy $\F_\eps$ grows up \emph{at least} logarithmically in $\eps$. 
	In particular, if~\eqref{hp:potential_bound} holds, then $\int_\Omega \abs{\nabla{\Q_\eps}}^2\,{\d}x$ grows up at least logarithmically in $\eps$.
\end{remark}

\section{Compactness}
\label{sect:compactness}

In this section, we prove compactness properties, in appropriate norms,
for both the $\Q_\eps$-com\-po\-nent and the $\M_\eps$-com\-po\-nent of
sequences $\{(\Q_\eps,\,\M_\eps)\}$ of critical
points of $\F_\eps$, subject to boundary condition either
as in~\eqref{bc}--\eqref{hp:bc} or as in~\eqref{bcbis}--\eqref{hp:bcbis}
and satisfying assumption~\eqref{hp:potential_bound}.
We first prove convergence (up to subsequences) for the
$\Q_\eps$-component to a limiting map
$\Q_\star : \Omega \to \NN$. For this task, we employ methods
reminiscent of the standard Ginzburg-Landau theory but appropriately
modified to take the coupling term into account. A key step consists in
proving that the energy densities $\mu_\eps$ defined in~\eqref{eq:def-mu-eps}
converge to a limiting measure $\mu_\star$ whose support is a finite set of points.
Then, as in \cite{CanevariMajumdarStroffoliniWang},
we observe that, away from the (finite) energy-concentration set
and up to a suitable change of variable (presented in Section~\ref{sec:u} below),
we can rewrite the functional $\F_\eps$ as
the sum of a Ginzburg-Landau type functional (with a slightly modified potential)
and a vectorial Allen-Cahn functional including a perturbation term which is small
in energy.
Taking advantage of this fact and
of classical compactness results for sequences with bounded Allen-Cahn energy
from~\cite{Baldo, FonsecaTartar}, we eventually show that a subsequence of
$\{\M_\eps\}$ converges in $L^p(\Omega)$, for any finite $p \geq 1$.

\subsection{Changes of variables: from $\Q$ to $\q$ and from $\M$ to $\u$}\label{sec:u}
Following \cite[Section~3]{CanevariMajumdarStroffoliniWang},
we introduce suitable types of ``change of variables'' for both the $\Q$-component
and the $\M$-component of a pair $(\Q,\,\M) \in \Sz \times \R^2$.
\vskip5pt

\noindent As for the $\Q$-component, we recall from~\eqref{eq:small-q} that the map
\[
	\Sz \ni \Q \mapsto \q := \sqrt{2}\left(Q_{11},\, Q_{12} \right) \in \R^2
\]
provides an isometric isomorphism between $\Sz$ (endowed with
the Frobenius scalar product) and $\R^2$ (with the standard
scalar product). Of course, we can also view $\q$ as a complex
number, in the obvious way.
This correspondence obviously extends to an isometric
correspondence between $\Q$-tensor fields and vector fields.
With the help of Lemma~\ref{lemma:feps}, we deduce that
\begin{equation}\label{eq:log-energy-q}
	\int_\Omega \left( \frac{1}{2}\abs{\nabla \q}^2 + \frac{1}{8\eps^2} \left(\abs{\q}^2 - 1 \right)^2\right)\,{\d}x
	\leq \mathcal{F}_\eps(\Q,\,\M) + \beta^2 \int_\Omega \abs{\M}^2\,{\d}x.
\end{equation}
Thus, in view of Theorem~\ref{lemma:energy-est} and of the uniform bound~\eqref{max-QM},
the Ginzburg-Landau energy of a sequence $\{\q_\eps\}$ associated with a sequence
of critical points $\{(\Q_\eps,\,\M_\eps)\}$ is logarithmically bounded.
This allows us to apply to $\Q$-tensor fields
several results typical of the Ginzburg-Landau theory with essentially no changes.
Usually,
we will understood the usage of the above correspondence, whenever
convenient, for the rest of this paper.

\vskip5pt

\noindent Let $\Q \in \Sz \setminus \{0\}$. 
Then, by the spectral theorem, we may write
\[
	\Q = \frac{\abs{\Q}}{\sqrt{2}}(\n \otimes \n - \m\otimes \m)
\]
for orthogonal vectors $\n$, $\m \in \mathbb{S}^1$, where $\n$ denotes
the eigenvector of $\Q$ relative to the positive eigenvalue while
$\m$, because of the zero-trace constraint, is relative to the
negative eigenvalue. In turn,
we may decompose any $\M \in \R^2$ as
\[
	\M = (\M\cdot \n)\, \n + (\M \cdot \m)\,\m.
\]
As in \cite[Equation~(3.6)]{CanevariMajumdarStroffoliniWang}, we define
\[
	u_1 := (\M\cdot \n), \qquad u_2 := (\M\cdot \m).
\]
Upon setting $\u := (u_1,\,u_2)$, we clearly have
\[
	\abs{\M} = \abs{\u}.
\]
The above discussion generalises to $\Q$-tensor field and
vector fields as follows.
Let $G \subseteq \Omega$ be a smooth,
simply connected domain. Let
\[
	\{(\Q_\eps,\,\M_\eps)\} \subset
	W^{1,2}\left(G,\,\Sz\right) \times W^{1,2}(G,\,\R^2)
\]
be a sequence satisfying, for any $\eps > 0$,
\begin{gather}
	\int_G \left( \frac{1}{2}\abs{\nabla \Q_\eps}^2 + \frac{1}{4\eps^2}\left( \abs{\Q_\eps}^2-1\right)^2 \right) \,{\d}x \lesssim \abs{\log\eps}, \label{eq:logbound+lift-assumptions-1}\\
	\abs{\Q_\eps(x)} \geq \frac{1}{2},\quad \abs{\M_\eps(x)} \leq A \qquad
	\mbox{for any } x \in G,\label{eq:logbound+lift-assumptions-2}
\end{gather}
where the implicit constant on the right-hand side
of~\eqref{eq:logbound+lift-assumptions-1} does not depend on $\eps$ and
$A$ is some positive constant that does not depend on $\eps$ as well.
By~\eqref{max-QM} and Proposition~\ref{prop:Q_bound},
these assumptions are verified for critical points
under either~\eqref{bc}-\eqref{hp:bc} or~\eqref{bcbis}-\eqref{hp:bcbis},
provided that assumption~\eqref{hp:potential_bound} is in force.
Then, 
since $G$ is simply connected and smooth, we may coherently decompose
\begin{equation}\label{eq:polar-dec-Q-G}
	\Q_\eps = \frac{\abs{\Q_\eps}}{\sqrt{2}}
	\left( \n_\eps \otimes \n_\eps - \m_\eps \otimes \m_\eps \right)
	\qquad \mbox{in } G,
\end{equation}
where $(\n_\eps,\,\m_\eps)$ is an orthonormal frame of eigenvectors
of $\Q_\eps$ in $G$, with $\n_\eps \in W^{1,2}(G,\,\mathbb{S}^1)$ and
$\m_\eps \in W^{1,2}(G,\,\mathbb{S}^1)$ --- in fact, since $\Q_\eps$ is
smooth, $\n_\eps$, $\m_\eps$ are also smooth. We will always denote with
$\n_\eps(x)$ the eigenvector of $\Q_\eps(x)$ relative to its
(strictly) positive eigenvalue.
As in the above, we then let
\begin{equation}\label{eq:def-u}
	(u_\eps)_1 := \M_\eps \cdot \n_\eps, \qquad
	(u_\eps)_2 := \M_\eps \cdot \m_\eps,
\end{equation}
so that
\[
	\M_\eps = (u_\eps)_1 \n_\eps + (u_\eps)_2 \m_\eps.
\]
Then, we define
\[
	\u_\eps \colon G \to \R^2, \qquad
	\u_\eps := ( (u_\eps)_1,\,(u_\eps)_2 ).
\]
Again, there holds
\begin{equation}\label{eq:linfty-bound-u}
	\abs{\u_\eps(x)} = \abs{\M_\eps(x)}, \qquad
	\mbox{for any } x \in G,
\end{equation}
so that any sequence $\{\u_\eps\}$ constructed in $G$ out of maps
$\M_\eps$ belonging to critical pairs
is bounded in $L^\infty(G)$ by a constant that does not depend
on $G$ (but only on the coupling parameter $\beta$).
For later purposes, we also notice that
\begin{equation}\label{eq:QMM}
	\frac{\Q}{\abs{\Q}}\M \cdot \M
	= \frac{1}{\sqrt{2}}\left(u_1^2 - u_2^2\right)
	\qquad \mbox{in } G.
\end{equation}
Next, we introduce the functions
(see~\cite[(3.7),~(3.8)]{CanevariMajumdarStroffoliniWang})
\begin{gather}
	g_\eps(\Q) := \frac{1}{4\eps^2}\left( \abs{\Q}^2 -1 \right)^2 -
	\frac{2 \kappa_\star}{\eps}\left( \abs{\Q}-1 \right) + \kappa_\star^2, \label{eq:g} \\
	h(\u) := \frac{1}{4}\left( \abs{\u}^2 - 1 \right)^2
	-\frac{\beta}{\sqrt{2}} \left( u_1^2 - u_2^2 \right)
	+ \frac{\beta^2 + \sqrt{2}\beta}{2} \label{eq:h}.
\end{gather}
By \cite[Lemma~3.3 and Lemma~3.4]{CanevariMajumdarStroffoliniWang},
$g_\eps$ and $h$ are non-negative functions. They are related
to $f_\eps$ as follows (just recall~\eqref{f-ell} or
cf.~\cite[Equation~(3.14)]{CanevariMajumdarStroffoliniWang}):
\begin{equation}\label{eq:fgh}
\begin{split}
	\frac{1}{\eps^2}f_\eps(\Q,\,\M) &= g_\eps(\Q) + \frac{1}{\eps}h(\u) +
	\frac{\abs{\Q}-1}{\eps}\left( 2 \kappa_\star - \frac{\beta}{\sqrt{2}}\left(u_1^2-u_2^2 \right)\right) \\
	&+\frac{\kappa_\eps}{\eps^2} - \frac{1}{2\eps}\left(\beta^2 + \sqrt{2}\beta\right) - \kappa_\star^2.
\end{split}
\end{equation}
By~\eqref{eq:chi-to-k*}, the last line in~\eqref{eq:fgh} vanishes as $\eps \to 0$.
Then, it follows from~\eqref{hp:potential_bound},~\eqref{eq:bdd-GL-pot},
\eqref{max-QM}, the nonnegativity of $g_\eps$, and an application of
H\"{o}lder's inequality on the third
term at the right-hand side (recalling also~\eqref{eq:k*}) that,
whenever $\{(\Q_\eps,\,\M_\eps)\}$ is a sequence of critical points
satisfying~\eqref{bc}--\eqref{hp:bc} or~\eqref{bcbis}--\eqref{hp:bcbis}
and~\eqref{hp:potential_bound}, there holds
\begin{equation}\label{eq:h-bdd}
	\frac{1}{\eps}\int_G h(\u_\eps) \,{\d}x \leq C,
\end{equation}
where $C$ depends only on $\beta$ and $\Cpot$.

For the reader's convenience and for later reference, we recall the
following result from \cite{CanevariMajumdarStroffoliniWang}.
\begin{lemma}[{\cite[Lemma~3.4]{CanevariMajumdarStroffoliniWang}}]
\label{lemma:h}
	The function $h : \R^2 \to \R$ defined in~\eqref{eq:h}
	has the following properties:
	\begin{enumerate}[(i)]
		\item For any $\u \in \R^2$, we have $h(\u) \geq 0$, with
		$h(\u) = 0$ if and only if $\u = \u_{\pm}$, where
		\begin{equation}\label{eq:def-upm}
			\u_{\pm} := \left( \pm\left(\sqrt{2}\beta + 1\right)^{1/2}, 0 \right).
		\end{equation}
		\item The Hessian matrix of $h$ at both $\u_+$ and $\u_-$
		is strictly positive definite. (In particular, $h$ behaves quadratically near the wells $\u_{\pm}$.)
	\end{enumerate}
\end{lemma}

With the functions $g_\eps$ and $h$ at hand,
we may rewrite the functional $\F_\eps(\Q_\eps,\,\M_\eps;\,G)$ as
\begin{equation}\label{eq:Feps-decoupling}
	\F_\eps(\Q_\eps,\,\M_\eps;\,G) = \int_G \left( \frac{1}{2} \abs{\nabla \Q_\eps}^2
	+ g_\eps(\Q_\eps)\right)\,{\d}x + \int_G \left( \frac{\eps}{2}\abs{\nabla \u_\eps}^2 + \frac{1}{\eps}h(\u_\eps) \right)\,{\d}x + R_\eps,
\end{equation}
where $R_\eps$ is a remainder term
(defined in \cite[Equation~(3.15)]{CanevariMajumdarStroffoliniWang}),
small in the sense that $R_\eps = {\rm o}_{\eps \to 0}(1)$
(cf. \cite[(3.9) and Lemma~4.10]{CanevariMajumdarStroffoliniWang}).
The definition of the remainder term $R_\eps$ involves \emph{both}
$\Q_\eps$ and $\u_\eps$, and therefore the change of variable
in~\eqref{eq:def-u} doest not yield a complete decoupling of the
functional $\F_\eps$ into a term depending only on $\Q_\eps$
and a term depending only on $\u_\eps$. However, it allows us to
consider the functional
\begin{equation}\label{eq:def-J}
	\mathcal{AC}_\eps(\u_\eps; \, G) := \int_G \left( \frac{1}{2}\abs{\nabla \u_\eps}^2
	+ \frac{1}{\eps}h(\u_\eps) \right)\,{\d}x.
\end{equation}
Since the function $\u \mapsto h(\u)$ has superlinear growth as
$\abs{\u} \to +\infty$, Lemma~\ref{lemma:h} implies that $h$ can be regarded
as an \emph{Allen-Cahn potential}, of the type considered
in \cite{FonsecaTartar} (see also \cite{Baldo, Sternberg}). Thus,
$\u \mapsto \mathcal{AC}_\eps(\u;\,G)$ can be seen as an
\emph{Allen-Cahn functional with multiple wells of equal depth}.%

In Lemma~\ref{lemma:bdd-AC-energy} below,
we show that $\{\u_\eps\}$ is a sequence with
bounded energy (i.e., $\mathcal{AC}_\eps(\u_\eps;\,G)$
is bounded uniformly with respect to $\eps$ and $G$).
Thanks to well-known compactness statements for
energy-bounded sequences in the (vectorial) Allen-Cahn theory (again, see
\cite{FonsecaTartar} and also \cite{Baldo, Sternberg}),
in the next subsection we will obtain
rather easily compactness in $L^p(\Omega)$, for any $p$
with $1 \leq p < +\infty$, for the sequence
$\{\u_\eps\}$ and, joining this with the compactness results
for the $\Q_\eps$-component, we obtain compactness in
$L^p(\Omega)$ for $\{\M_\eps\}$ as well.


\begin{lemma}
\label{lemma:bdd-AC-energy}
Let $\{(\Q_\eps,\,\M_\eps)\}$ be a sequence of critical points of $\F_\eps$
subject to either~\eqref{bc}-\eqref{hp:bc} or~\eqref{bcbis}-\eqref{hp:bcbis}
and suppose that assumption~\eqref{hp:potential_bound} holds.
Suppose that there exist a finite set $\Sigma_\star \subset \overline{\Omega}$
such that for any $K \subset \Omega \setminus \Sigma_\star$ there exist
$\eps_0(K) > 0$ such that $\abs{\Q_\eps}\geq 1/2$
in $\Omega \setminus \Sigma_\star$ for any $0 < \eps \leq \eps_0(K)$.
Then, for any simply connected open set
$G \subset K$ with smooth boundary and any $0 < \eps \leq \eps_0(K)$,
there holds
\begin{equation}\label{eq:AC-bound}
	\int_G \left( \frac{1}{2}\abs{\nabla \u_\eps}^2 + \frac{1}{\eps}h(\u_\eps) \right)\,{\d}x \leq C,
\end{equation}
where $C$ is a positive constant that depends \textbf{only} on $\Omega$,
$\beta$, the constant $C_{\rm pot}$ on the right-hand side
of~\eqref{hp:potential_bound}, and the $L^1(\partial\Omega)$- and the
$L^2(\partial\Omega)$-norm of $\Qb$.
\end{lemma}

\begin{remark}\label{rk:AC-bound-independent-G}
	We stress that the constant $C$ in~\eqref{eq:AC-bound}
	does \textbf{not} depend on the subdomain $G$.
\end{remark}

\begin{remark}
	We shall prove later on that the assumptions of Lemma~\ref{lemma:bdd-AC-energy}
	are in fact satisfied whenever $\{\u_\eps\}$ is obtained by a sequence of critical
	points $\{(\Q_\eps,\,\M_\eps)\}$ (see Lemma~\ref{lemma:mu*} below).
\end{remark}

\begin{proof}[{Proof of Lemma~\ref{lemma:bdd-AC-energy}}]
By Proposition~\ref{prop:M_bound} and
assumption~\eqref{hp:potential_bound},
we have
\begin{equation}\label{eq:AC-bound-compu1}
	\int_\Omega\left( \frac{\eps}{2}\abs{\nabla \M_\eps}^2 + \frac{1}{\eps^2}f_\eps(\Q_\eps,\,\M_\eps) \right) \,{\d}x \leq C,
\end{equation}
for some constant $C$ depending only on the bound $\Cpot$
in~\eqref{hp:potential_bound}, the
parameter $\beta$, and on the $L^1(\partial \Omega)$- and the $L^2(\partial\Omega)$-norm of $\Qb \times \partial_\ttau \Qb$.

Fix any compact set $K \subset \Omega \setminus \Sigma_\star$
and any simply connected open set $G \subset K$
with smooth boundary. Then, by assumption,
there exists a threshold value $\eps_0(K) > 0$ so that
$\abs{\Q_\eps} \geq 1/2$ for any $0 < \eps \leq \eps_0(K)$, so
that we can switch from $\M_\eps$ to $\u_\eps$ in $G$.

By the definition~\eqref{eq:def-u} of $\u_\eps$, there holds
\begin{equation}\label{eq:AC-bound-compu3}
	\abs{\nabla \M_\eps}^2 = \abs{\nabla \u_\eps}^2 - 2 j(\u_\eps) \cdot j(\n_\eps)
	 + \abs{\u_\eps}^2 \abs{\nabla \n_\eps}^2
\end{equation}
pointwise in $G$. Here, $j(\u_\eps)$ and $j(\n_\eps)$ are defined as in~\eqref{eq:def-prejac-vect}.
Since
\[
	j(\u_\eps)\cdot j(\n_\eps) \leq \abs{\u_\eps}\abs{\nabla \u_\eps}\abs{\nabla \n_\eps},
\]
it follows from Young's inequality,~\eqref{eq:linfty-bound-u},
and~\eqref{max-QM} that
\begin{equation}\label{eq:AC-bound-compu4}
	\abs{\nabla \u_\eps}^2 \lesssim \abs{\nabla \M_\eps}^2 + \abs{\nabla \Q_\eps}^2
\end{equation}
where the implicit constant on the right-hand side
depends only on 
the parameter $\beta$ and $\Omega$. 
Combining~\eqref{eq:AC-bound-compu3} and~\eqref{eq:AC-bound-compu4},
we have
\begin{equation}\label{eq:AC-bound-compu5}
	\eps \int_G \abs{\nabla \u_\eps}^2 \,{\d}x \lesssim
	\eps \int_G \left( \abs{\nabla \M_\eps}^2 + \abs{\nabla \Q_\eps}^2 \right) \,{\d}x,
\end{equation}
where the implicit constant on the right-hand side
depends only on the bound $\Cpot$
in~\eqref{hp:potential_bound}, the
parameter $\beta$, $\Omega$, and the $L^1(\partial \Omega)$- and the $L^2(\partial\Omega)$-norm of $\Qb \times \partial_\ttau \Qb$.

Putting~\eqref{eq:h-bdd} and~\eqref{eq:AC-bound-compu5} together
and recalling once again the bounds in Theorem~\ref{lemma:energy-est},
we obtain~\eqref{eq:AC-bound}.
\end{proof}

\vskip5pt
\noindent The strategy proceeds now as follows. By the
estimates in Theorem~\ref{lemma:energy-est},
the leading part of the energy is carried by the $\Q_\eps$-component.
Once divided by $\abs{\log\eps}$, all the terms in the energy apart from
the elastic energy of $\Q_\eps$ are negligible in the limit as
$\eps \to 0$. So, we will start by proving compactness
the for $\Q_\eps$-component of a sequence $\{(\Q_\eps,\,\M_\eps)\}$
of critical points. In particular, we will show that the energy
concentrates around at most finitely many points
(possibly located on the boundary). After possible extraction of
a subsequence, we have $\abs{\Q_\eps} \geq 1/2$ locally away from
the concentration points for any $\eps$ small enough,
and this allows us to employ the change of variables described at the beginning of this
section. The second
point of the strategy relies on the bound~\eqref{eq:AC-bound} and
it consists in proving compactness for the
$\u_\eps$-components through 
classical compactness results for sequences with equibounded
vectorial Allen-Cahn energy.
Finally, gathering the compactness results for $\{\Q_\eps\}$
and $\{\u_\eps\}$, we obtain a compactness theorem for the
$\M_\eps$-component.

\subsection{Compactness for $\Q_\eps$}\label{sec:Qeps}

In this section, we prove that the $\Q_\eps$-component of a
sequence $\{(\Q_\eps, \M_\eps)\}$ of critical points of $\F_\eps$
satisfying the boundary conditions~\eqref{bc}--\eqref{hp:bc}
or~\eqref{bcbis}--\eqref{hp:bcbis} and assumption~\eqref{hp:potential_bound}
converges, in various norms, 
to a limiting map $\Q_\star : \Omega \to \NN$ as $\eps \to 0$. Moreover,
we show that
$\Q_\star$ is a harmonic map with values into $\NN$ having only a finite number
of point singularities.
This extend to critical points some results that are proven
in \cite[Section~4]{CanevariMajumdarStroffoliniWang} for minimisers.
However, our arguments are substantially different from those
in~\cite{CanevariMajumdarStroffoliniWang}, since the latter
often take advantage of minimality.
Instead, we rely on PDE methods, adapting the existing theory about \emph{solutions}
of the Ginzburg-Landau system. In particular, we modify methods
from~\cite[Chapter~X]{BBH} and from~\cite{BBO} to keep the
perturbation into account. This task involves some extra challenges,
the chief example being, perhaps, Lemma~\ref{lemma:con-loc-mod+pot} below.

To keep the size of this section at the minimum, we address the
reader to the literature on the Ginzburg-Landau theory whenever
possible, exploiting the change of variables~\eqref{eq:small-q}
and the related energy bound~\eqref{eq:log-energy-q}.
\vskip5pt

\begin{remark}(Extension to a larger domain)\label{rk:extension}
	For technical reasons, it will be convenient to extend any critical pair
	$(\Q_\eps,\,\M_\eps)$
	to an ($\eps$-independent) open neighbourhood of $\Omega$,
	i.e., to some open, simply connected set
	$\Omega' \supset \overline{\Omega}$, while keeping~\eqref{Qbound-log},
	\eqref{eq:M_bound}, and~\eqref{hp:potential_bound}.

	In the case of pure Dirichlet
	boundary conditions~\eqref{bc}--\eqref{hp:bc}, this
	can be achieved by taking any $C^1$-smooth extension $\M$ of $\Mb$ to
	$\Omega' \setminus \overline{\Omega}$ satisfying the constraint
	$\abs{\M} = \left(\sqrt{2}\beta+1\right)^{1/2}$ and then letting
	$\Q = \sqrt{2}\left( \frac{\M \otimes \M}{\sqrt{2}\beta+1} - \frac{\I}{2}\right)$
	in $\Omega' \setminus \overline{\Omega}$. Up to shrinking $\Omega'$ a little bit,
	we may assume that $\M$ and $\Q$ have finite energy. Moreover,
	by Lemma~\ref{lemma:feps}, we have
	$\frac{1}{\eps^2}f(\Q,\,\M) = \kappa_\star^2 + \o_{\eps \to 0}(1)$ as $\eps \to 0$.

	In the case of the `mixed' boundary conditions~\eqref{bcbis}--\eqref{hp:bcbis},
	we can take $\delta > 0$ so small that in the $\delta$-neighbourhood
	$\Omega' := \Omega_\delta = \{ x \in \R^2\,:\,\dist(x,\partial \Omega) < \delta \}$
	of $\Omega$ we can extend $\Q_\eps$ and $\M_\eps$ by reflection across the
	boundary, keeping the energy estimates.
	For vector fields, this can be done, for instance, by exploiting
	\cite[Proposition~8.1 and Remark~8.2]{ABO2} while for $\Q$-tensor fields
	it can be done, e.g., as described in \cite[Section~2.2]{DMP1}.
	(However, differently
	from~\cite{DMP1}, here we do not need the extended maps to be solutions
	of an equation in the whole extended domain but only to control their
	energy. This goal can be achieved by requiring merely $C^2$-regularity
	of the boundary and $C^1$-regularity of the Dirichlet boundary condition
	for the $\Q_\eps$-component.)
	Note that these extensions are $\eps$-dependent, but preserve the boundary
	condition~\eqref{bcbis}--\eqref{hp:bcbis}. 
\end{remark}

\begin{notation}
In all the statements below, it is understood that
$\{(\Q_\eps, \M_\eps)\}$ is a given sequence of
critical points of $\F_\eps$ satisfying the boundary
conditions~\eqref{bc}--\eqref{hp:bc}
or~\eqref{bcbis}--\eqref{hp:bcbis} and
assumption~\eqref{hp:potential_bound}.
Whenever necessary, we also assume (without
changing notation and even without explicit mention)
that each pair $(\Q_\eps,\,\M_\eps)$
has been extended to a larger domain $\Omega'$
as explained in Remark~\ref{rk:extension}.
Finally, we set $\rho_\eps = \abs{\Q_\eps}$ and we denote by
\begin{equation}\label{eq:mu-eps}
	\mu_\eps := \frac{1}{\abs{\log\eps}}\left( \frac{1}{2}\abs{\nabla \Q_\eps}^2 +
	\eps \abs{\nabla \M_\eps}^2 + \frac{1}{\eps^2}f(\Q_\eps,\,\M_\eps) \right)
\end{equation}
the energy density of the pair
$\left(\Q_\eps,\,\M_\eps\right)$, rescaled by $\abs{\log\eps}$ and dually
seen as a measure in $\Omega$ (or in $\Omega'$).
\end{notation}

The first result of this section characterises the possible limiting
energy measures $\mu_\star$. 
\begin{lemma}\label{lemma:mu*}
	There exist a (not-relabelled) subsequence and a
	bounded, non-negative Radon measure $\mu_\star$ in $\overline{\Omega}$
	such that, as $\eps \to 0$,  the energy measures $\{\mu_\eps\}$
	converge to $\mu_\star$ weakly* as measures in $\overline{\Omega}$.
	Moreover,
	\begin{equation}\label{eq:mu_*}
		\mu_\star = \sum_{j=1}^{N_\star} \theta_j \delta_{c_j},
	\end{equation}
	where $c_1,\dots,c_{N_\star}$ are distinct points in $\overline{\Omega}$,
	$\theta_1, \dots, \theta_{N_\star}$ are real numbers
	so that $\theta_j \geq \eta_*$ for any $j \in \{1,\dots,N_\star\}$,
	and $\eta_*$ is the constant provided by
	Proposition~\ref{prop:clearing-out}
\end{lemma}

\begin{proof}
	For convenience, we extend each $\mu_\eps$ to a measure (still denoted
	$\mu_\eps$) on the whole $\R^2$ by defining $\mu_\eps$ as zero outside $\Omega'$.
	By the logarithmic energy bound in $\Omega'$, given by Theorem~\ref{lemma:energy-est}
	and Remark~\ref{rk:extension}, $\{\mu_\eps\}$ is a sequence of bounded,
	non-negative Radon measures in $\R^2$, whence the existence of a (not relabelled)
	subsequence and of a bounded, non-negative
	Radon measure $\mu_\star$ on $\R^2$ such that
	$\mu_\eps \rightharpoonup^* \mu_\star$ weakly*
	in the sense of measures in $\R^2$ and, in particular, on $\overline{\Omega}$,
	follow by standard arguments.
	Again by Theorem~\ref{lemma:energy-est} and Remark~\ref{rk:extension},
	$\int_{\Omega'} \eps \abs{\nabla \M_\eps}^2$
	and $\eps^{-2} \int_{\Omega'} f(\Q_\eps,\,\M_\eps)$ are both negligible
	in front of $\abs{\log\eps}$ as $\eps \to 0$.
	By Proposition~\ref{prop:clearing-out},
	there exists positive constants $\eta_*$, $\eps_*$ such that if
	$B(x_0,\,R)$ is a ball on which
	$\F_\eps(\Q_\eps,\,\M_\eps;\,B(x_0,\,R)) \leq \eta_* \log(R/\eps)$ and
	$0 < \eps \leq \eps_* R$, then
	$\int_{B(x_0,\,R/2)}\abs{\nabla \Q_\eps}^2 \leq C(x_0,\,R)$ independently
	of $\eps$, so that $B(x_0,\,R/2)$ is outside the support of $\mu_\star$.
	This implies that, for any ball $B \subset \R^2$,
	\[
		\mbox{either} \qquad \mu_\star(B) = 0 \qquad \mbox{or} \qquad
		\mu_\star(B) \geq \eta_*.
	\]
	Now, let
	\[
		S := \left\{ x \in \R^n \,\colon\, \mbox{for any } r > 0,\,\,\mu_\star(B(x,\,r)) > 0 \right\}.
	\]
	Then, $S$ is finite, with $\abs{S} \leq \mu_\star(\R^n) / \eta_*$, and
	$\mu_\star$ is supported by $S$.
\end{proof}

The support of $\mu_\star$ represents exactly
the energy-concentration set, on which therefore, in view
of the logarithmic bound on the energy,
weak convergence in $W^{1,2}$ and stronger types of convergence may
fail. The next lemma, analogous to the last part
of \cite[Lemma~4.4]{CanevariMajumdarStroffoliniWang} and reminiscent
of the results in \cite[Section~X.5]{BBH},
shows global weak compactness for $\{\Q_\eps\}$ in $W^{1,p}(\Omega)$
when $1 \leq p < 2$ and weak $W^{1,p}_{\rm loc}$-compactness outside $\spt \mu_\star$
for any finite $p \geq 1$.

\begin{notation}
All the arguments below involve extraction of subsequences;
however, to keep the statement shorter,
we do not emphasise this (very standard) fact in the forthcoming statements.
\end{notation}

\begin{lemma}\label{lemma:weak-conv-Q*}
There exists a map $\Q_\star \colon \Omega \to \NN$ such that
	\begin{align}
		&\Q_\eps \rightharpoonup \Q_\star \qquad \mbox{weakly in } W^{1,p}(\Omega)
		\mbox{ for any } p < 2 \label{eq:weak-conv-Q-p} \\
		&\Q_\eps \rightharpoonup \Q_\star \qquad \mbox{weakly in } W^{1,p}_{\loc}\left(\Omega \setminus \spt\mu_\star\right) \mbox{ for any } p \in [1,\,+\infty) \label{eq:local-weak-conv-Q}.
	\end{align}
Moreover, there exists $\n_\star \in W^{1,p}(\Omega, \mathbb{S}^1)$ for any $p < 2$
and belonging to $W^{1,p}_{\rm loc}(\Omega \setminus \spt\mu_\star)$ for any
$p \in [1,\,+\infty)$, so that
\begin{equation}\label{eq:Q*-oriented}
	\Q_\star(x) = \sqrt{2}\left( \n_\star(x) \otimes \n_\star(x) - \frac{\I}{2} \right)
\end{equation}
for a.e. $x \in \Omega$.
\end{lemma}

\begin{proof}
	The existence of a map $\Q_\star \in W^{1,p}\left(\Omega,\,\Sz\right)$ for any $p \in [1,\,2)$
	such that $\Q_\eps \rightharpoonup \Q_\star$ (on a subsequence) in $W^{1,p}(\Omega)$
	for any $p \in [1,\,2)$ follows from Proposition~\ref{prop:Q_bound}.
	Since $\partial \Omega$ is bounded and Lipschitz, by the Rellich-Kondrachov theorem
	it also follows that, up to extracting another subsequence,
	$\Q_\eps(x) \to \Q_\star(x)$ for a.e. $x \in \Omega$, hence $\abs{\Q_\star(x)} = 1$
	for a.e. $x \in \Omega$, so that $\Q_\star \in W^{1,p}(\Omega,\,\Sz)$
	for any $p \in [1,\,2)$, with $\abs{\Q_\star(x)} = 1$
	for a.e. $x \in \Omega$. Thus,~\eqref{eq:weak-conv-Q-p} is proven.

	Now, let $K \subset \Omega' \setminus \spt \mu_\star$ be any compact set.
	We can choose $R > 0$ and cover $K$ by finitely many balls
	$B(x_j,\,R/2) \csubset \Omega'$ which are out of the support
	of $\mu_\star$ and such that the balls $B(x_j,\,R)$ are still contained
	in $\Omega' \setminus \spt \mu_\star$.
	On each of the balls $B(x_j,\,R)$, by Proposition~\ref{prop:clearing-out} we must have
	$\F_\eps(\Q_\eps,\,\M_\eps;\,B(x_j,\,R)) \leq \eta_* {\log(R/\eps)}$
	for any $0 < \eps \leq \eps_* R$, where $\eps_*$ depends only on $K$,
	and hence $\abs{\Q_\eps} \geq 1/2$ on $B(x_j,\,3R/4)$ for any such $\eps$.
	Thus, by Proposition~\ref{prop:improved-W1p-estimates},
	up to extraction of a subsequence, interpolation, and a diagonal argument,
	we have $\Q_\eps \rightharpoonup \Q_\star$ in
	$W^{1,p}(\cup_j ( B(x_j,\,R/2) \cap \Omega))$
	for any $p$ with $1 \leq p < +\infty$,
	whence~\eqref{eq:local-weak-conv-Q} follows.

	Finally, let $K \subset \Omega \setminus \spt \mu_\star$ be any
	compact set with non-empty interior and let $G \subset K$ be any
	simply connected open set with smooth boundary. As in the above,
	in $\overline{G}$ we have $\abs{\Q_\eps(x)} \geq 1/2$ for any
	$x \in \overline{G}$ and any $\eps > 0$ small enough, depending
	only on $K$. Thus,
	in $G$ we can write $\Q_\eps$ in the form~\eqref{eq:polar-dec-Q-G},
	for vector fields $\n_\eps$, $\m_\eps \in W^{1,2}\left(G,\,\mathbb{S}^1\right)$
	globally defined on $G$, where, as usual, $\n_\eps$ is the eigenvector
	related to the positive eigenvalue of $\Q_\eps$. In fact, we may write
	$\I = \n_\eps \otimes \n_\eps + \m_\eps \otimes \m_\eps$ in $G$ and
	therefore,
	\[
		\Q_\eps = {\abs{\Q_\eps}}
		\sqrt{2} \left( \n_\eps \otimes \n_\eps - \frac{\I}{2} \right)
	\]
	in $G$. Since $\abs{\nabla \n} \lesssim \abs{\nabla \Q}$ pointwise,
	the estimates on $\norm{\nabla \Q}_{L^p(G)}$ obtained above
	for any $1 \leq p < +\infty$ yield estimates on $\norm{\nabla \Q}_{L^p(G)}$,
	and thus the weak convergence $\n_\eps \rightharpoonup \n_\star$ in $W^{1,p}(G)$
	for any $1 \leq p <+\infty$. In turn, by interpolation and possibly after
	extraction of subsequence,
	it follows that $\n_\eps \to \n_\star$
	in $L^p(G)$ for any $1 \leq p < +\infty$ as well as
	$\n_\eps(x) \to \n_\star(x)$ for a.e. $x \in G$, whence $\abs{\n_\star(x)} = 1$
	for a.e. $x \in G$ and
	\[
		\Q_\star(x)
		= \sqrt{2}\left( \n_\star(x) \otimes \n_\star(x) - \frac{\I}{2} \right)
	\]
	for a.e. $x\in G$.
	By letting $K$ and $G$ vary in $\Omega \setminus \spt \mu_\star$, it follows
	that~\eqref{eq:Q*-oriented} holds at a.e. $x \in \Omega$.
\end{proof}

\begin{remark}
	We shall see in Theorem~\ref{thm:Lp-conv-Meps-M*} that,
	in fact, $\Q_\star$ is oriented in $\Omega \setminus \spt \mu_\star$
	by (a multiple of) the vector field
	$\M_\star$, the $L^p(\Omega)$-limit of the maps $\M_\eps$.
\end{remark}

Next, we improve on the weak
$W^{1,p}_{\loc}$-convergence to
get strong $W^{1,p}_{\rm loc}$-convergence outside
$\spt\mu_\star$, for any $p$ with $1 \leq p < +\infty$.

\begin{prop}\label{prop:strong-conv-Qeps}
	As $\eps \to 0$, we have $\Q_\eps \to \Q_\star$ strongly
	in $W^{1,p}(\Omega)$ for any $p$ with $1 \leq p < 2$ and
	in $W^{1,p}_{\rm loc}\left(\Omega \setminus \spt\mu_\star\right)$
	for any $p$ with $1 \leq p < +\infty$.
\end{prop}

The proof of Proposition~\ref{prop:strong-conv-Qeps}
rests on the pivotal lemma below.

\begin{lemma}\label{lemma:W12-strong-conv-Qeps}
	As $\eps \to 0$, we have $\Q_\eps \to \Q_\star$ strongly
	in $W^{1,2}_{\rm loc}\left(\Omega \setminus \spt\mu_\star\right)$.
\end{lemma}

The proof of Lemma~\ref{lemma:W12-strong-conv-Qeps}
follows a similar path to that traced in
\cite[Step~1 and Step~2 of Theorem~X.2]{BBH}.
However, we have to face the additional difficulties posed
by the coupling term.
The key step in the proof Lemma~\ref{lemma:W12-strong-conv-Qeps}
is the following auxiliary result, which replaces
\cite[Equation~(127) in Chapter~X]{BBH}.

\begin{lemma}\label{lemma:con-loc-mod+pot}
	For any ball $B \csubset \Omega \setminus \spt\mu_\star$,
	we have
	\begin{equation}\label{eq:conv-loc-mod+pot}
		\int_{B} \left\{\abs{\nabla \rho_\eps}^2 + \left( \frac{\rho_\eps-1}{\eps} - \kappa_\star \right)^2  \right\}\,{\d}x \to 0
	\end{equation}
	as $\eps \to 0$.
\end{lemma}

\begin{proof}
	For notational convenience, we systematically drop the
	subscript $\eps$ within this proof. Also, we
	denote by $B'$ and $B''$ the balls concentric
	with $B$ whose radii are $3/4$ and $1/2$ of
	that of $B$, respectively.
	Since $B$ is out of $\spt\mu_\star$, we may
	assume that $\rho = \abs{\Q} \geq 1/2$ in $B'$ for
	any $\eps$ small enough. Since $B$ is arbitrary, 
	it is enough to obtain~\eqref{eq:conv-loc-mod+pot}  
	with $B''$ in place of $B$ (then, \eqref{eq:conv-loc-mod+pot} will 
	follow by a standard covering argument).
	\vskip5pt

	\noindent We consider (once again) the equation satisfied
	by $\rho$:
	\begin{equation}\label{gertrudo}
		-\Delta \rho + 4 \rho \abs{\nabla \varphi}^2
		+\frac{1}{\eps^2}(\rho-1)(\rho+1)\rho
		-\frac{\sigma}{\eps} = 0,
	\end{equation}
	which holds pointwise in $B'$.
	We recall that,
	thanks
	to the change of variables $\M \mapsto \u$ given
	by~\eqref{eq:def-u} and to~\eqref{eq:QMM}, the term
	$\sigma = \beta\frac{\Q\M\cdot \M}{\rho}$
	can be written in $B'$ in terms of the components
	$u_1, u_2$ of $\u$ as follows:
	\[
		-\sigma = \frac{\beta}{\sqrt{2}}
		\left(u_2^2 - u_1^2\right).
	\]
	As a consequence of the definition~\eqref{eq:h} of
	the Allen-Cahn potential $h(\u)$ and of the
	bound~\eqref{eq:h-bdd}, by explicit computation we have that
	\begin{equation}\label{eq:sigma-to-k*}
		\norm{2\kappa_\star-\sigma}_{L^2(B)}
		\lesssim \sqrt{h(\u)} = {\rm O}(\sqrt{\eps})
		\qquad \mbox{as } \eps \to 0,
	\end{equation}
	where the implicit constant depends only on $\beta$
	(through the uniform bound on $\norm{\M}_{L^\infty(\Omega)}$
	given by~\eqref{max-QM}) and on the constant $\Cpot$ on
	the right-hand side of~\eqref{hp:potential_bound}.
	Let $\zeta \in C^\infty_c(B',\,[0,1])$ be any smooth
	cut-off function such that
	\[
		\spt \zeta \subset B', \qquad
		\zeta \equiv 1 \quad \mbox{on } B'', \qquad
		\abs{\nabla \zeta} \lesssim 1.
	\]
	Multiplying~\eqref{gertrudo} by $\zeta^2 (\rho-1-\kappa_\star \eps)$,
	where $\kappa_\star$ is the constant defined in~\eqref{eq:k*},
	and then integrating over $B'$, we obtain
	\[
	\begin{split}
		\int_{B'} &\left\{ \zeta^2 \abs{\nabla \rho}^2 + \frac{1}{2}\zeta \nabla \zeta \cdot \nabla \left(\rho-1-\kappa_\star \eps\right)^2 + 4 \zeta^2 \rho(\rho-1-\kappa_\star \eps)\abs{\nabla \varphi}^2 \right.\\
		&\left.+ \zeta^2 \left( \frac{\rho-1}{\eps} - \kappa_\star \right) \left( \frac{\rho-1}{\eps}(\rho+1)\rho -\sigma \right) \right\} \,{\d}x = 0.
	\end{split}
	\]
	We rearrange this equality in the form
	\begin{equation}\label{eq:int-est-nabla-rho+pot-1}
	\begin{split}
		\int_{B'} &\left\{ \zeta^2 \abs{\nabla \rho}^2
		+ \zeta^2 \left( \frac{\rho-1}{\eps} - \kappa_\star \right) \left( \frac{\rho-1}{\eps}(\rho+1)\rho -\sigma \right) \right\} \,{\d}x \\
		&=
		-\int_{B'} \left\{\frac{1}{2}\zeta \nabla \zeta \cdot \nabla \left(\rho-1-\kappa_\star \eps\right)^2 + 4 \zeta^2 \rho(\rho-1-\kappa_\star \eps)\abs{\nabla \varphi}^2\right\}\,{\d}x.
	\end{split}
	\end{equation}
	By the $W^{1,p}$-estimates~\eqref{eq:improved-W1p-estimates}
	for $\Q$ and the uniform convergence $\rho \to 1$ on $B'$
	as $\eps \to 0$ provided by
	Corollary~\ref{cor:linfty-est-Qeps}, it follows that
	\begin{equation}\label{eq:int-est-nabla-rho+pot-2}
		\int_{B'} \zeta^2 \rho(\rho-1-\kappa_\star \eps)\abs{\nabla \varphi}^2 \,{\d}x
		= {\rm O}(\eps) \qquad \mbox{as } \eps \to 0.
	\end{equation}
	Moreover, by using Young's inequality, we have
	\begin{equation}\label{eq:int-est-nabla-rho+pot-3}
		-\int_{B'} \frac{1}{2}\zeta \nabla \zeta \cdot \nabla \left(\rho-1-\kappa_\star \eps\right)^2 \,{\d}x \leq \int_{B'} \frac{1}{2} \left(\rho-1-\kappa_\star \eps\right)^2 \abs{\nabla \zeta}^2 \,{\d}x + \int_{B'} \frac{1}{2}\zeta^2 \abs{\nabla \rho}^2 \,{\d}x.
	\end{equation}
	We further observe that the first integral on the right-hand
	side above is of order $\eps$ as $\eps \to 0$ so that
	plugging~\eqref{eq:int-est-nabla-rho+pot-2},
	\eqref{eq:int-est-nabla-rho+pot-3} into~\eqref{eq:int-est-nabla-rho+pot-1}
	we obtain
	\begin{equation}\label{eq:int-est-nabla-rho+pot-4}
		\int_{B'} \left\{ \frac{1}{2}\zeta^2 \abs{\nabla \rho}^2
		+ \zeta^2 \left( \frac{\rho-1}{\eps} - \kappa_\star \right) \left( \frac{\rho-1}{\eps}(\rho+1)\rho -\sigma \right) \right\} \,{\d}x = {\rm O}(\eps)
		\qquad \mbox{as } \eps \to 0.
	\end{equation}
	Recalling~\eqref{eq:sigma-to-k*}
	and~\eqref{eq:improved-W1p-estimates}, we notice that
	\[
		\int_{B'} \left( \frac{\rho-1}{\eps} - \kappa_\star \right)
		(2\kappa_\star - \sigma) \,{\d}x = {\rm O}\left(\sqrt{\eps}\right)
		\qquad\mbox{as } \eps \to 0,
	\]
	so that we can rewrite~\eqref{eq:int-est-nabla-rho+pot-4} as
	\[
		\int_{B'} \left\{ \frac{1}{2}\zeta^2 \abs{\nabla \rho}^2
		+ \zeta^2 \left( \frac{\rho-1}{\eps} - \kappa_\star \right) \left( \frac{\rho-1}{\eps}(\rho+1)\rho - 2\kappa_\star \right) \right\} \,{\d}x = {\rm O}\left(\sqrt{\eps}\right)
		\qquad \mbox{as } \eps \to 0.
	\]
	Once again by the uniform convergence $\rho \to 1$ on $B'$,
	we get
	\[
	\begin{split}
	\int_{B'} &\zeta^2 \left( \frac{\rho-1}{\eps} - \kappa_\star \right) \left( \frac{\rho-1}{\eps}(\rho+1)\rho - 2\kappa_\star \right)\,{\d}x \\
	&= \int_{B'} \zeta^2 \left( \frac{\rho-1}{\eps} - \kappa_\star \right) \left( \frac{\rho-1}{\eps}(2 + {\rm o}(1))(1+{\rm o}(1)) - 2\kappa_\star \right) \,{\d}x \qquad\mbox{as } \eps \to 0
	\end{split}
	\]
	and therefore, using again \eqref{eq:improved-W1p-estimates}
	and recalling that $\zeta \equiv 1$ on $B''$,
	we eventually obtain
	\[
		\int_{B''} \left\{ \frac{1}{2} \abs{\nabla \rho}^2
		+ 2 \left( \frac{\rho-1}{\eps} - \kappa_\star \right)^2 \right\} \,{\d}x
		= {\rm o}(1)
		\qquad \mbox{as } \eps \to 0,
	\]
	which, thanks to a standard covering argument,
	implies~\eqref{eq:conv-loc-mod+pot}.
\end{proof}

\begin{remark}\label{rk:geps-to-0}
	Let $B \csubset \Omega \setminus \spt\mu_\star$ be any ball.
	Then, by the uniform convergence $\rho_\eps \to 1$ in
	$B$, the bound~\eqref{eq:bdd-GL-pot},~\eqref{eq:conv-loc-mod+pot},
	and the definition of $g_\eps$, it also follows that
	\begin{equation}\label{eq:geps-to-0}
		\int_{B} g_\eps(\Q_\eps)\,{\d}x \to 0
	\end{equation}
	as $\eps \to 0$.
\end{remark}

\begin{remark}\label{rk:conv-pot-k*-Lp}
	It follows from~\eqref{eq:conv-loc-mod+pot} and
	Proposition~\ref{prop:improved-W1p-estimates} that
	\begin{equation}
		\norm{\nabla \rho_\eps}_{L^p(B)} + 
		\norm{\frac{\rho_\eps - 1}{\eps} -\kappa_\star}_{L^p(B)} \to 0
		\qquad \mbox{as } \eps \to 0
	\end{equation}
	in any ball $B \csubset \Omega \setminus \spt\mu_\star$.
\end{remark}

We are now ready for the proof of
Lemma~\ref{lemma:W12-strong-conv-Qeps}.
\begin{proof}[Proof of Lemma~\ref{lemma:W12-strong-conv-Qeps}]
	We argue similarly to as in Step~1 and Step~2 in the proof
	of \cite[Theorem~X.2]{BBH}, with the technical modifications
	needed to handle the coupling term in the
	Euler-Lagrange equations.

	\vskip5pt

	\noindent Let $K$ be any compact subset of
	$\Omega \setminus \spt\mu_\star$. For our purposes,
	it is enough to consider the case in which $K$ has non-empty
	interior, therefore we may assume
	that there exist open sets $A$, $U$ such that
	\[
		A \csubset K \subset U \csubset
		\Omega \setminus \spt\mu_\star.
	\]
	Moreover, we may assume that $K$ and $U$ have smooth boundary.
	\setcounter{step}{0}
	\begin{step}[Convergence of $j(\Q_\eps)$]
 	We consider again the Hodge-type decomposition
	of $\j_\eps := j(\Q_\eps)$ 	in Step~2 of
	Proposition~\ref{prop:Q_bound}, i.e., we write
	\[
		\j_\eps = \nabla H_\eps + \nabla^\perp \phi_\eps + \nabla^\perp \xi_\eps
		\qquad \mbox{in } \Omega,
	\]
	where $H_\eps$, $\phi_\eps$, and $\xi_\eps$ solve, respectively,
	\eqref{H_eps}-\eqref{Heps-bc},~\eqref{phi_eps}, and~\eqref{xi_eps}.
	We are going to show that
	\begin{align}
		& H_\eps \to H_\star \quad \mbox{in } W^{1,2}(K), \label{eq:H*} \\
		& \xi_\eps \to \xi_\star \quad \mbox{in } W^{1,2}(K) \label{eq:xi*},\\
		& \phi_\eps \to \phi_\star \quad \mbox{in } W^{1,2}(K)\label{eq:phi*}
	\end{align}
	as $\eps \to 0$. Of course, this implies that $\j_\eps \to \j_\star$
	strongly in $L^2(K)$, where
	\[
		\j_\star
		= \nabla H_\star + \nabla^\perp \phi_\star + \nabla^\perp \xi_\star,
	\]
	and, on the other hand,
	the strong $W^{1,2}$-convergence and Remark~\ref{rk:continuity-Jac}
	imply that $\j_\star = j(\Q_\star) = \frac{1}{2}\Q_\star \times \nabla \Q_\star$.

	\vskip5pt

	\noindent
	We start by noticing that
	the weak convergences
	$H_\eps \rightharpoonup H_\star$ and $\xi_\eps \rightharpoonup \xi_\star$
	in $W^{1,2}(\Omega)$ are consequences of~\eqref{Heps-2}
	and~\eqref{xieps-2}, respectively.
	Moreover, by \eqref{H_eps}, \eqref{Heps-bc},
	Corollary~\ref{cor:est-prejac-M}, and standard elliptic
	estimates, it follows that
	\[
		\norm{H_\eps}_{W^{1,2}(\Omega)} \to 0
	\]
	as $\eps \to 0$, whence~\eqref{eq:H*} follows, with $H_\star = 0$.
	Concerning $\xi_\eps$, we multiply Equation~\eqref{xi_eps} by
	$\zeta(\xi_\eps - \xi_\star)$, where $\zeta$ is an arbitrary smooth cut-off
	function satisfying
	\[
		0\leq \zeta \leq 1, \qquad \zeta \equiv 1 \mbox{ on } K,
		\qquad K \subset \spt \zeta \subset U,
	\]
	and we integrate over $\Omega$ to get
	\[
	\begin{split}
		&\int_\Omega \zeta \abs{\nabla \xi_\eps}^2
		-\int_\Omega \zeta \nabla \xi_\eps \cdot \nabla \xi_\star
		+ \int_\Omega (\xi_\eps - \xi_\star) \nabla \zeta \cdot \nabla \xi_\eps \\
		&= \int_\Omega -\zeta\left(1-\alpha_\eps^2\right) \j_\eps \times \nabla \xi_\eps
		+ \int_\Omega \zeta \left(1-\alpha_\eps^2\right) \j_\eps \times \nabla \xi_\star
		- \int_\Omega (\xi_\eps-\xi_\star)\left(1-\alpha_\eps^2\right) \j_\eps \times \nabla \zeta
	\end{split}
	\]
	By weak convergence,
	\begin{equation}\label{eq:conv-xi*}
		\int_\Omega \zeta \nabla \xi_\eps \cdot \nabla \xi_\star
		\to \int_\Omega \zeta \abs{\nabla \xi_\star}^2
	\end{equation}
	as $\eps \to 0$. Moreover, since weak convergence in
	$W^{1,2}(\Omega)$ implies strong convergence in $L^2(\Omega)$,
	from~\eqref{xieps-2} and~\eqref{curl2} it follows that
	\[
		\int_\Omega (\xi_\eps - \xi_\star) \nabla \zeta \cdot \nabla \xi_\eps \to 0,
		\qquad \int_\Omega (\xi_\eps-\xi_\star)\left(1-\alpha_\eps^2\right) \j_\eps \times \nabla \zeta \to 0,
	\]
	as $\eps \to 0$. By the uniform convergence $\rho_\eps \to 1$
	(and hence, $\alpha_\eps \to 1$) in $\overline{U}$ (see Corollary~\ref{cor:linfty-est-Qeps})
	and the weak convergence $\Q_\eps \rightharpoonup \Q_\star$ in $W^{1,2}(U)$ (implying
	$\j_\eps \rightharpoonup^* \j_\star$
	in the sense of distributions on $U$), we have
	\[
		\int_\Omega -\zeta\left(1-\alpha_\eps^2\right) \j_\eps \times \nabla \xi_\eps \to 0, \qquad
		\int_\Omega \zeta \left(1-\alpha_\eps^2\right) \j_\eps \times \nabla \xi_\star \to 0
	\]
	as $\eps \to 0$.
	Thus, weak convergence and~\eqref{eq:conv-xi*} imply
	$\xi_\eps \to \xi_\star$ strongly in $W^{1,2}(K)$.

	Finally, we focus on $\phi_\eps$. First of all,
	we notice that, by~\eqref{phieps-p} and Sobolev embedding,
	we have
	\begin{equation}\label{eq:phi-eps-bdd-L2}
		\norm{\phi_\eps}_{L^2(\Omega)} \leq C,
	\end{equation}
	where the constant $C$ does not depend on $\eps$.
	Moreover, there exists $\phi_\star \in W^{1,p}(\Omega)$
	for any $p$ with $1 \leq p < 2$, so that
	$\phi_\eps \rightharpoonup \phi_\star$ weakly in $W^{1,p}(\Omega)$
	and, up to a subsequence, strongly in $L^p(\Omega)$ and
	pointwise a.e. in $\Omega$, for any $p$ with
	$1 \leq p < 2$.

	Letting $\zeta$ be any smooth cut-off function as above
	and multiplying~\eqref{phi_eps} by $\zeta^2 \phi_\eps$, we
	obtain
	\[
		\int_\Omega \zeta^2 \abs{\nabla \phi_\eps}^2
		= \int_\Omega \phi_\eps \alpha_\eps^2 \j_\eps \times \nabla \zeta^2
		+ \int_\Omega \zeta^2 \alpha_\eps^2 \j_\eps \times \nabla \phi_\eps
		- \int_\Omega 2 \phi_\eps \zeta \nabla \zeta \cdot \nabla \phi_\eps.
	\]
	Using Young's inequality on the last two terms on the
	right-hand side, we obtain
	\[
		\int_\Omega \zeta^2 \abs{\nabla \phi_\eps}^2
		\lesssim \int_\Omega \phi_\eps \alpha_\eps^2 \j_\eps \times \nabla \zeta^2
		+ \int_\Omega \zeta^2 \alpha_\eps^4 \abs{\j_\eps}^2
		+ \int_\Omega \abs{\phi_\eps}^2 \abs{\nabla \zeta}^2
	\]
	In virtue of~\eqref{eq:phi-eps-bdd-L2}, of
	Proposition~\ref{prop:improved-W1p-estimates} and of
	the uniform convergence $\alpha_\eps \to 1$ on $\overline{U}$,
	the right-hand side above is bounded independently of $\eps$,
	whence
	\[
		\norm{\phi_\eps}_{W^{1,2}(U)} \leq C,
	\]
	where $C > 0$ is a constant that does not depend on $\eps$.
	Thus, $\phi_\eps \rightharpoonup \phi_\star$ weakly
	in $W^{1,2}(U)$.
	In order to obtain strong convergence in $W^{1,2}(K)$, it is
	now enough to show that
	\begin{equation}\label{eq:phi-eps-to-phi*}
		\int_K \abs{\nabla \phi_\eps}^2 \to
		\int_K \abs{\nabla \phi_\star}^2
	\end{equation}
	as $\eps \to 0$. To this purpose, we multiply~\eqref{phi_eps}
	by $\zeta^2(\phi_\eps - \phi_\star)$ and integrate over $\Omega$.
	This yields
	\[
		\int_\Omega \zeta^2 \abs{\nabla \phi_\eps}^2
		- \int_\Omega \zeta^2 \nabla \phi_\eps \cdot \nabla \phi_\star 
		= \int_\Omega (\phi_\star-\phi_\eps) \nabla \zeta^2 \cdot \nabla \phi_\eps
		+ \int_\Omega \zeta^2(\phi_\eps - \phi_\star)\curl\left(\alpha_\eps^2 \j_\eps\right).
	\]
	By the weak convergence $\phi_\eps \rightharpoonup \phi_\star$
	in $W^{1,2}(U)$, it follows that, as $\eps \to 0$,
	\[
		\int_\Omega \zeta^2 \nabla \phi_\eps \cdot \nabla \phi_\star
		\to \int_\Omega \zeta^2 \abs{\nabla \phi_\star}^2
	\]
	and, moreover, recalling also that weak convergence in $W^{1,2}(U)$ implies strong
	convergence in $L^2(U)$,
	\[
		\int_\Omega (\phi_\star-\phi_\eps) \nabla \zeta^2 \cdot \nabla \phi_\eps \to 0,
	\]
	as $\eps \to 0$. In addition, by the uniform convergence $\alpha_\eps \to 1$
	in $\overline{U}$,
	we have $\alpha_\eps > 1/2$ eventually in $U$, therefore,
	by~\eqref{preJaccurl} and~\eqref{A},
	the term $\curl(\alpha_\eps^2 \j_\eps)$ is identically zero in $K$
	for any $\eps > 0$ small enough. Thus,
	\[
		\int_\Omega \zeta^2(\phi_\eps - \phi_\star)\curl\left(\alpha_\eps^2 \j_\eps\right) = 0
	\]
	for any $\eps > 0$ small enough, and we obtain~\eqref{eq:phi-eps-to-phi*}.
	\end{step}

	\begin{step}[Conclusion]
		By~\eqref{eq:H*},~\eqref{eq:xi*},~\eqref{eq:phi*},
		it follows that, as $\eps \to 0$,
		\[
			\j_\eps \to \j_\star \qquad\mbox{strongly in } L^2(K).
		\]
		On the other hand, we have $\rho_\eps \to 1$ uniformly
		on $K$, and hence, by~\eqref{preJacnabla} and
		Lemma~\ref{lemma:con-loc-mod+pot} (and an elementary
		covering argument), it follows that, as $\eps \to 0$,
		\[
			\Q_\eps \to \Q_\star \qquad \mbox{strongly in }
			W^{1,2}(K).
		\]
		The conclusion follows.
		\qedhere
	\end{step}
\end{proof}

The proof of Proposition~\ref{prop:strong-conv-Qeps}
is now immediate.
\begin{proof}[Proof of Proposition~\ref{prop:strong-conv-Qeps}]
	The claim follows immediately by
	Lemma~\ref{lemma:weak-conv-Q*}, the
	$W^{1,p}$-estimates provided by
	Proposition~\ref{prop:improved-W1p-estimates},
	and Lemma~\ref{lemma:W12-strong-conv-Qeps}.
\end{proof}

In view of Proposition~\ref{prop:strong-conv-Qeps},
we also obtain the locally uniform convergence of $\Q_\eps$ towards
$\Q_\star$ away from $\spt\mu_\star$.
\begin{corollary}\label{cor:unif-conv-Qeps-Q*}
	Let $K \subset \Omega \setminus \spt\mu_\star$ be
	any compact set.
	As $\eps \to 0$, we have
	\begin{equation}\label{eq:unif-conv-Qeps-Q*}
		\Q_\eps \to \Q_\star \qquad \mbox{uniformly on } K.
	\end{equation}
\end{corollary}

\begin{proof}
	By a standard covering argument, it suffices to prove the
	claim for a ball $B \csubset \Omega \setminus \spt\mu_\star$.
	By the Gagliardo-Nirenberg inequality, we have
	\[
		\norm{\Q_\eps - \Q_\star}_{L^\infty(B)} \lesssim
		\norm{\Q_\eps - \Q_\star}_{W^{1,p}(B)}^{2/p} \norm{\Q_\eps - \Q_\star}_{L^p(B)}^{1-2/p}
	\]
	for any $p > 2$
	and, by Proposition~\ref{prop:strong-conv-Qeps} and the
	$L^\infty$-bound~\eqref{max-QM}, the right-hand side above
	tends to zero as $\eps \to 0$, so the conclusion follows.
\end{proof}

Now, we show that the pre-jacobian
$\j_\star := j(\Q_\star) = \frac{1}{2}\Q_\star\times \nabla \Q_\star$ of
$\Q_\star$ is divergence-free in $\Omega$ in the sense of distributions in $\Omega$,
i.e., that there holds
\begin{equation}\label{eq:divQ*}
		\div(\Q_\star \times \nabla \Q_\star) = 0
\end{equation}
in the sense of distributions in $\Omega$, that is, in $\mathscr{D}'(\Omega)$.
As in the usual Ginzburg-Landau theory, we shall draw several important consequences
from this equation.

\begin{prop}\label{prop:Q*-harm}
	The map $\Q_\star \colon \Omega \to \NN$
	satisfies~\eqref{eq:divQ*} in the sense of distributions in
	$\Omega$.
	Moreover, $\Q_\star$ is a smooth harmonic map in $\Omega \setminus \spt\mu_\star$
	with values into $\NN$, it is continuous
	in $\overline{\Omega}\setminus \spt \mu_\star$,
	and it satisfies the boundary condition
	$\Qb$ in the sense of $W^{1-1/p,p}(\partial \Omega)$, for any $p \in (1,2)$.
\end{prop}

\begin{proof}
	By taking the vector product of~\eqref{EL-Q},~\eqref{EL-M} with
	$\Q_\eps$ and $\M_\eps$, respectively, we obtain that any critical
	pair $(\Q_\eps,\,\M_\eps)$ satisfies
	\[
		-\div\left(\Q_\eps \times \nabla \Q_\eps\right) =
		\frac{\eps}{2} \div\left(\M_\eps \times \nabla \M_\eps\right)
		\qquad \mbox{in } \Omega.
	\]
	Let $p \in (1,\,2)$.
	By Lemma~\ref{lemma:weak-conv-Q*}, we have $\Q_\eps \rightharpoonup \Q_\star$
	weakly in $W^{1,p}(\Omega)$ as $\eps \to 0$. Thus, by the trace theorem,
	$\Q_\star$ has a well-defined trace in $W^{1-1/p,p}(\partial\Omega)$, and in fact,
	$\Q_\star = \Qb$ on $\partial \Omega$ in the sense of traces, i.e., in
	$W^{1-1/p,p}(\partial\Omega)$.
	Moreover, by the uniform bound in $L^\infty(\Omega)$ for $\Q_\eps$ and
	Lebesgue's dominated convergence theorem, we also have $\Q_\eps \to \Q_\star$
	strongly in $L^q(\Omega)$, for any $q < +\infty$. As a consequence,
	\begin{equation}\label{eq:Q*-harm-compu1}
		\div\left(\Q_\eps \times \nabla \Q_\eps\right)
		{\rightharpoonup}^* \div\left(\Q_\star \times \nabla \Q_\star\right)
		\qquad \mbox{in } \mathscr{D}'(\Omega), \quad \mbox{as } \eps \to 0.
	\end{equation}
	On the other hand, by Proposition~\ref{prop:M_bound} and the uniform
	bound in $L^\infty(\Omega)$ for $\M_\eps$ given by~\eqref{max-QM}, it
	follows that
	\[
		\eps\norm{\M_\eps \times \nabla \M_\eps }_{L^2(\Omega)}
		\lesssim \eps \norm{\M_\eps}_{L^\infty(\Omega)}
		\norm{\nabla \M_\eps}_{L^2(\Omega)}
		\lesssim \eps^{1/2} \to 0
	\]
	as $\eps \to 0$, whence
	\begin{equation}\label{eq:Q*-harm-compu2}
		\eps \div(\M_\eps \times \nabla \M_\eps) \to 0
		\qquad \mbox{in } W^{-1,2}(\Omega)\quad \mbox{as } \eps \to 0
	\end{equation}
	Combining~\eqref{eq:Q*-harm-compu1} and~\eqref{eq:Q*-harm-compu2},
	we deduce that $\Q_\star$ satisfies~\eqref{eq:divQ*}
	in the sense of distributions in $\Omega$.

	To check that $\Q_\star$ is smooth in $\Omega \setminus \spt\mu_\star$
	and continuous in $\overline{\Omega} \setminus \spt\mu_\star$, take
	$G \csubset \overline{\Omega} \setminus \spt\mu_\star$, an arbitrary
	simply connected domain with smooth boundary.
	Then, $\Q_\star \in W^{1,2}(G,\,\NN)$, so that we can apply
	lifting results (see e.g. \cite[Theorem~1]{BethuelChiron}) and write
	\[
		\Q_\star = \frac{1}{\sqrt{2}}
		\begin{pmatrix}
			\cos\theta_\star & \sin\theta_\star \\
			\sin\theta_\star & -\cos\theta_\star
		\end{pmatrix},
	\]
	where $\theta_\star \in W^{1,2}(G)$ is a scalar function satisfying
	\[
		\Delta \theta_\star = 0 \qquad \mbox{as distributions in } G.
	\]
	Therefore, $\theta_\star$ is smooth in $G$, and so is $\Q_\star$.
	Moreover, if $G$ touches $\overline{\Omega}$, $\theta_\star$ is
	continuous up to $\partial \Omega$ and hence $\Q_\star$ is.
	By letting $G$ vary among all simply connected domain with smooth boundary
	compactly contained in $\overline{\Omega} \setminus \spt\mu_\star$, the conclusion
	follows.
\end{proof}

Next, we state an additional convergence property for $\Q_\eps$.
To this purpose, we first recall that
for any map $\Q \in W^{1,1}(\Omega,\,\NN)$ having only a finite
number of singular points, the distributional Jacobian $J(\Q)$
writes as a sum of Dirac deltas,
whose multiplicities are the degrees of $\Q$ computed on small circles
around the singular points. The singularities with degree different from
zero are called the \emph{topological singularities} of $\Q$;
these are captured, in location and charge, by
$J(\Q)$,
cf., e.g., \cite[Section~2.0]{BrezisMironescu}.

Also, we recall that,
given a closed ball $\overline{B}_\rho(z) \subset \Omega'$ and a
$\Q$-tensor field $\Q$ on $\overline{B}_\rho(z)$ such
that $\abs{\Q} \geq 1/2$ on $\partial B_\rho(z)$, the map
\[
	\frac{\Q}{\abs{\Q}} \colon \partial B_\rho(z) \simeq \mathbb{S}^1
	\to \NN \simeq \RP^1
\]
is well-defined and continuous and hence its topological degree is
well-defined as an element of $\frac{1}{2}\Z$ (i.e., it is either
an integer or a half-integer). In particular, being
$\Q_\star \in W^{1,1}(\Omega,\,\NN)$, it has a well-defined topological degree
around each point in $\spt\mu_\star$, its Jacobian has the structure
described above, and since its topological singular set is the support of
$J(\Q_\star)$, it also follows that the sum of the degrees of its topological
singularities equals the degree of $\Qb$.

Without loss of generality, we may suppose that first
$N_{\rm top} \in \mathbb{N} \cup \{0\}$ points in $\spt\mu_\star$
are the topological singularities of $\Q_\star$.

\begin{prop}\label{prop:J*}
	We have
	\begin{equation}\label{eq:cpt-jacobians}
		J(\Q_\eps) \to J(\Q_\star) \qquad \mbox{in } W^{-1,1}(\Omega)
	\end{equation}
	as $\eps \to 0$. Moreover,
	\begin{equation}\label{eq:J*-subset-mu*}
		 \abs{ J(\Q_\star)} \leq \pi \mu_\star
	\end{equation}
	and, in fact,
	\begin{equation}\label{eq:JacQ*}
		J(\Q_\star) = \pi \sum_{j=1}^{N_{\rm top}} d_j \delta_{c_j}
		\qquad \mbox{in } \mathscr{D}'(\Omega),
	\end{equation}
	where $N_{\rm top} \leq N$ denotes the number of topological singularities
	of $\Q_\star$ and the numbers $d_1$, \dots, $d_{N_{\rm top}} \in \frac{1}{2}\Z \setminus \{0\}$
	are the degrees of $\Q_\star$ around the points in $\spt J(\Q_\star)$.
\end{prop}

\begin{proof}
	The convergence in $W^{-1,1}(\Omega)$ claimed by~\eqref{eq:cpt-jacobians} follows
	exactly as in the proof of \cite[Lemma~4.7]{CanevariMajumdarStroffoliniWang},
	using Lemma~\ref{lemma:weak-conv-Q*} above and the uniform bound~\eqref{max-QM}
	for $\norm{\Q_\eps}_{L^\infty(\Omega)}$.

	Clearly, \eqref{eq:J*-subset-mu*} holds, because $\Q_\star$ is a smooth
	$\NN$-valued map outside $\spt\mu_\star$, by Proposition~\ref{prop:Q*-harm}.
	(Hence, $\q_\star$ is a
	smooth $\mathbb{S}^1$-valued map outside $\spt\mu_\star$.)

	Finally,~\eqref{eq:JacQ*} follows since 
	$\Q_\star \in W^{1,1}(\Omega,\,\NN)$ and it is smooth in $\Omega$ outside
	the finite set $\spt \mu_\star$,
	hence $\q_\star \in W^{1,1}(\Omega,\,\mathbb{S}^1)$ and it is
	smooth in $\Omega$ outside
	$\spt \mu_\star$, so that we can apply classical structure
	theorems for Jacobians in two-dimensions
	(see, e.g., \cite[Section~2.1]{BrezisMironescu}), jointly with
	the further observation that
	$\Q_\star$ has degree zero around points in
	$(\spt \mu_\star \setminus J(\Q_\star)) \cap \Omega$
	(because these are either continuity points for $\Q_\star$
	or singularities of degree zero).
\end{proof}

\begin{remark}\label{rk:J*-subset-mu*}
	In the case of minimisers, considered in \cite{CanevariMajumdarStroffoliniWang},
	equality occurs in~\eqref{eq:J*-subset-mu*}, because of matching lower and upper
	energy bounds, which prevent energy concentration on boundary points and on points
	which are not topological singularities of $\Q_\star$. However, this is unclear,
	for general critical points. In general, $\spt\mu_\star$ contains the topological
	singularities of $\Q_\star$ and may contain, in addition, boundary singular
	points of $\Q_\star$, singularities of degree
	zero of $\Q_\star$, and even points (possibly located on the boundary)
	which are \emph{not} singularities of $\Q_\star$.
\end{remark}

In the case of pure Dirichlet boundary conditions and assuming in addition
(as in \cite{BBH})
that $\Omega$ is star-shaped, we can add information to
the conclusion of Proposition~\ref{prop:Q*-harm} and Proposition~\ref{prop:J*},
obtaining a result much closer to \cite[Theorem~X.4]{BBH} as well as to
\cite[Proposition~4.9]{CanevariMajumdarStroffoliniWang}.

\begin{corollary}\label{cor:Q*-harm-Dirichlet}
	Assume that $\Omega$ is star-shaped and that
	$\{(\Q_\eps,\,\M_\eps)\}$ is a sequence of critical points of
	$\F_\eps$, subject to the Dirichlet boundary conditions~\eqref{bc}--\eqref{hp:bc},
	and that the assumption~\eqref{hp:potential_bound} holds. Then, $\Q_\star$ is the
	canonical harmonic map with singularities
	$\left(c_1,\dots,c_{N_{\rm top}}\right)$, degrees
	$\left(d_1,\dots,d_{N_{\rm top}}\right)$, and boundary datum $\Qb$.
\end{corollary}

\begin{proof}
	We already know from Proposition~\ref{prop:Q*-harm} that
	$\Q_\star \colon \Omega \to \NN$ satisfies~\eqref{eq:divQ*}
	in the sense of distributions in $\Omega$, it is
	smooth harmonic in $\Omega \setminus \spt\mu_\star$, and continuous up to
	$\overline{\Omega} \setminus \spt\mu_\star$.
	The set of the topological singularities
	of $\Q_\star$ coincides with the support of the limiting Jacobian $J(\Q_\star)$
	and therefore such singularities carry degrees
	$d_1,\dots,d_{N_{\rm top}} \in \frac{1}{2}\Z \setminus \{0\}$.
	The points in $\spt\mu_\star \setminus \spt J(\Q_\star)$, if any, are
	continuity points for $\Q_\star$ or singular points located on the boundary
	or interior singularities of degree zero. We let $N' \in \mathbb{N} \cup \{0\}$
	be the number of points in $\spt\mu_\star$ which are either topological singularities
	of $\Q_\star$ or interior singularities of degree zero.

	We now show that, in fact, all the singularities of $\Q_\star$ are contained
	in the interior of $\Omega$. Once this is done, since $\Q_\star$ solves~\eqref{eq:divQ*}
	in $\mathscr{D}'(\Omega)$, it follows that $\Q_\star$ is the canonical harmonic map
	associated with the boundary condition $\Qb$, the singular points
	$c_1,\dots,c_{N'}$, and the degrees $d_1,\dots,d_{N'}$. However, since
	canonical harmonic maps are proper harmonic maps (see \cite[Definition on p.~14]{BBH}),
	it also follows the singularities of degree zero are removable, whence
	we may always assume that only the topological singularities are left.

	As in \cite[Theorem~X.4]{BBH}, full boundary regularity follows immediately
	once the hypotheses of \cite[Lemma~X.14]{BBH} are verified. (In the following, for
	brevity's sake, we understood
	the switching from $\Q_\star$ to $\q_\star$ (and back) where
	necessary to apply the results in \cite{BBH} in the form they are written there.)
	To this purpose, in view of Proposition~\ref{prop:Q*-harm}, it only remains
	to prove that
	\begin{equation}\label{eq:normal-der-Q*}
		\partial_\nnu \Q_\star \in L^2(\partial \Omega).
	\end{equation}
	To obtain~\eqref{eq:normal-der-Q*},
	we argue as in \cite{BBH}. Combining the Pohozaev identity with the
	($\eps$-independent) boundary conditions~\eqref{bc}--\eqref{hp:bc},
	we deduce that
	\begin{equation}\label{eq:Q*-can-harm-compu1}
		\norm{\partial_\nnu \Q_\eps}_{L^2(\partial\Omega)} \lesssim 1
	\end{equation}
	independently of $\eps$. On the other hand, we also infer
	in Proposition~\ref{prop:strong-conv-bdry-Qeps} in
	Appendix~\ref{app:dir} that
	(similarly to as in \cite[Theorem~X.3]{BBH})
	\begin{equation}\label{eq:Q*-can-harm-compu2}
		\Q_\eps \to \Q_\star \qquad \mbox{in }
		W^{1,2}_{\rm loc}\left(\overline{\Omega} \setminus \spt\mu_\star\right)
		\mbox{ as } \eps \to 0
	\end{equation}
	Combining~\eqref{eq:Q*-can-harm-compu1} and~\eqref{eq:Q*-can-harm-compu2},
	we obtain~\eqref{eq:normal-der-Q*}. Thus, by \cite[Lemma~X.4]{BBH}, $\Q_\star$ is
	smooth in a neighbourhood of $\partial \Omega$. But then $\Q_\star$ is the
	canonical harmonic map with singularities at
	$\left(c_1,\dots,c_{N_{\rm top}}\right)$, degrees
	$\left(d_1,\dots,d_{N_{\rm top}}\right)$
	around them, and boundary datum $\Qb$.
\end{proof}

\begin{remark}
	Even under the hypotheses of Corollary~\ref{cor:Q*-harm-Dirichlet}, it is generally
	unclear whether equality holds in~\eqref{eq:J*-subset-mu*} or whether
	$\spt\mu_\star \subset \Omega$.
\end{remark}

\subsection{Compactness for $\u_\eps$}\label{sec:cpt-u}

Here, we prove that any sequence $\{\u_\eps\}$
of functions $\u_\eps : G \to \R^2$
satisfying
\[
	\sup_{\eps > 0} \mathcal{AC}_\eps(\u_\eps;\,G) \leq C < +\infty
\]
is relatively compact in $L^p(G,\,\R^2)$, for any $p$ with
$1 \leq p < +\infty$.
The argument relies on the following specialisation to our case
of a classical compactness and lower-semicontinuity theorem
for sequences with bounded Allen-Cahn energy.
\begin{theorem}[{\cite[Theorem~4.1 and Theorem~3.4]{FonsecaTartar}}]\label{thm:FonsecaTartar}
	Any family $\{\u_\eps\}$ with $\mathcal{AC}_\eps(\u_\eps;\,G) \leq C < +\infty$
	for all $\eps > 0$ is relatively compact in $L^1(G,\,\R^2)$, i.e., there exists
	$\u_\star \in L^1(G,\,\R^2)$ and a (non-relabelled subsequence) such that
	$\u_\eps \to \u_\star$ a.e. in $G$ and strongly in $L^1(G,\,\R^2)$. Moreover,
	\[
		c_\beta \H^1\left( \S_{\u_\star} \cap G \right) \leq 
		\liminf_{\eps \to 0} \mathcal{AC}_\eps(\u_\eps;\,G),
	\]
	where
	\begin{equation} \label{c_beta}
		c_\beta := \frac{2\sqrt{2}}{3} \left(\sqrt{2}\beta + 1\right)^{3/2} .
	\end{equation}
\end{theorem}
In fact, \cite[Theorem~3.4]{FonsecaTartar} is a full $\Gamma$-convergence
statement in the $L^1$-topology.
Similar results, allowing for more general energies of Allen-Cahn type
having an arbitrary (finite) number of
wells of equal depth, can be found in \cite{Baldo}.

\begin{remark}\label{rk:u*}
In particular, $h(\u_\star) = 0$. Therefore, the limiting map
$\u_\star$ necessarily takes the form
\begin{equation}\label{eq:u*}
	\u_\star(x) = \left( \tau(x) \left( \sqrt{2}\beta + 1 \right)^{1/2},\, 0 \right)
	\qquad \mbox{for a.e. } x \in G
\end{equation}
where $\tau(x) \in \{1,-1\}$ is a sign. Since $\u_\star$ takes values
into a finite set, its distributional derivative
${\D}\u_\star$ must be concentrated on $\S_{\u_\star}$, implying
that $\u_\star \in \SBV(G,\,\R^2)$.
\end{remark}

Using Lemma~\ref{lemma:mu*} together with the
bounds~\eqref{eq:AC-bound}, \eqref{max-QM},
and the $L^1$-compactness in
Theorem~\ref{thm:FonsecaTartar}, we immediately get the
following compactness statement in all $L^p$-spaces with
$1 \leq p < +\infty$. The measure $\mu_\star$ appearing
in the statement below is exactly the limiting energy measure
given by~\eqref{eq:mu_*}.

\begin{lemma}\label{lemma:cpt-u}
	Let $\{(\Q_\eps,\,\M_\eps)\}$
	be a sequence of critical points of $\F_\eps$ subject
	to either \eqref{bc}--\eqref{hp:bc} or \eqref{bcbis}--\eqref{hp:bcbis} and
	suppose that assumption~\eqref{hp:potential_bound} holds.
	Let $G \csubset \Omega \setminus \spt \mu_\star$ be any simply connected open set.
	Then, the
	sequence $\{\u_\eps\}$ associated with $\{\M_\eps\}$ according to~\eqref{eq:def-u}
	is relatively compact in $L^p(G,\,\R^2)$, for any $p$ with $1 \leq p < +\infty$.
\end{lemma}

\begin{proof}
	Since the sequence $\{(\Q_\eps,\, \M_\eps)\}$ is made up of critical
	points of $\F_\eps$, it follows that $\{(\Q_\eps,\,\M_\eps)\}$
	satisfies the energy estimates in Theorem~\ref{lemma:energy-est}. Moreover,
	since $\overline{G}$ stays at strictly positive distance from $\spt\mu_\star$, by
	Corollary~\ref{cor:linfty-est-Qeps} and a covering argument
	it also follows that $\abs{\Q_\eps} \to 1$
	uniformly on $\overline{G}$, so that we have, eventually, $\abs{\Q_\eps} \geq 1/2$
	on $\overline{G}$.
	Let $\eps_0 > 0$ be the corresponding threshold value for $\eps$.
	Since $G$ is simply connected, we can lift $\Q_\eps$ in $G$ as
	in~\eqref{eq:polar-dec-Q-G}, at least for any $0 < \eps \leq \eps_0$.
	As a consequence, the sequence $\{\u_\eps\}$ in the statement is well-defined,
	at least for any $0 < \eps \leq \eps_0$.
	By~\eqref{eq:linfty-bound-u} and~\eqref{max-QM}, it follows that
	$\{\u_\eps\}$ is bounded in
	$L^\infty(G,\,\R^2)$ (by a constant independent of $G$ and, in fact, depending only
	the parameter $\beta$). 
	Therefore, Lemma~\ref{lemma:bdd-AC-energy} holds
	and, by Theorem~\ref{thm:FonsecaTartar}, $\{\u_\eps\}$ is relatively
	compact in $L^1(G,\,\R^2)$. Then, the conclusion follows from~\eqref{eq:linfty-bound-u}
	and~\eqref{max-QM}, by interpolation.
\end{proof}

\subsection{Compactness for $\M_\eps$}

In this section, putting together the results of Section~\ref{sec:Qeps}
and of Section~\ref{sec:u}, we prove a compactness result for
the $\M_\eps$-component
of a sequence of critical points $\{(\Q_\eps,\,\M_\eps)\}$
of $\F_\eps$ satisfying the boundary conditions \eqref{bc}-\eqref{hp:bc}
or \eqref{bcbis}-\eqref{hp:bcbis} and the
assumption~\eqref{hp:potential_bound}.

\begin{theorem}\label{thm:Lp-conv-Meps-M*}
	Let $\{(\Q_\eps,\,\M_\eps) \}$ be a sequence of critical points of $\F_\eps$,
	subject to either \eqref{bc}-\eqref{hp:bc} or \eqref{bcbis}-\eqref{hp:bcbis},
	and assume that~\eqref{hp:potential_bound} holds. Then, there exists a map
	$\M_\star \in \SBV(\Omega,\,\R^2)$ and a (non-relabelled) subsequence such that
	$\M_\eps \to \M_\star$ a.e. and strongly in $L^p(\Omega,\,\R^2)$ for any $p$
	with $1 \leq p < +\infty$ as $\eps \to 0$.
	Moreover, $\H^1(\S_{\M_\star}) < +\infty$ and $\M_\star$ satisfies
	\begin{align}
		\abs{\M_\star} &= \left( \sqrt{2}\beta + 1 \right)^{1/2}, \label{eq:mod-M*}\\
		\Q_\star &= \sqrt{2}\left( \frac{\M_\star \otimes \M_\star}{\sqrt{2}\beta+1} - \frac{\I}{2} \right),\label{eq:M*-Q*}
	\end{align}
	a.e. in $\Omega$. In particular, $f(\Q_\star,\,\M_\star) = 0$.
\end{theorem}

\begin{proof}
The proof goes exactly as in \cite{CanevariMajumdarStroffoliniWang}.
Therefore, we only sketch the main points, addressing the reader
to \cite[Proposition~4.11]{CanevariMajumdarStroffoliniWang}
for full details.

Up to extraction of a subsequence (and keeping Remark~\ref{rk:extension}
into account), by Lemma~\ref{lemma:mu*} the energy densities $\mu_\eps$
converge weakly* in the sense of measures
in $\overline{\Omega}$ to a limiting energy density $\mu_\star$, whose
support $\spt\mu_\star$ is a finite set contained in $\overline{\Omega}$.

By Corollary~\ref{cor:linfty-est-Qeps}, we know that, up to extraction of a
(further and non-relabelled) subsequence and a covering argument,
$\abs{\Q_\eps} \to 1$ uniformly in compact sets of $\Omega \setminus \spt\mu_\star$
as $\eps \to 0$ so that,
given any compact set $K \subset \Omega \setminus \spt\mu_\star$,
there exists $\eps_K > 0$ so that $\abs{\Q_\eps} \geq 1/2$ for any
$0 < \eps \leq \eps_K$. Let $K \subset \Omega \setminus \spt\mu_\star$ be
any compact set and let $G \subset K$
be any simply connected open set with smooth boundary.

By Lemma~\ref{lemma:cpt-u} applied in $G$, the sequence $\{\u_\eps\}$
associated with $\M_\eps$ as prescribed by~\eqref{eq:def-u} is
relatively compact in $L^p(G,\,\R^2)$, for any $p$ with $1 \leq p < +\infty$.
As $\eps \to 0$, besides the $L^p(G)$-convergence of the maps $\{\u_\eps\}$,
by Lemma~\ref{lemma:weak-conv-Q*} we also have that
$\Q_\eps \rightharpoonup \Q_\star$ weakly in $W^{1,2}(G)$ as $\eps \to 0$.
In particular, we have $\n_\eps \rightharpoonup \n_\star$ and
$\m_\eps \rightharpoonup \m_\star$
weakly in $W^{1,2}(G)$ and, up to extracting a subsequence and
using the Rellich-Kondrachov theorem, almost everywhere in $G$,
for some pair $(\n_\star,\,\m_\star)$ forming an
orthonormal frame for $\Q_\star$ in $G$. Thus, letting
\begin{equation}\label{eq:def-M*}
	\M_\star := (\u_\star)_1 \n_\star + (\u_\star)_2 \m_\star
	= \tau \left( \sqrt{2}\beta + 1 \right)^{1/2} \n_\star \quad \mbox{in } G,
\end{equation}
it follows that $\M_\star$ is well-defined, it does not depend on the choice
of the orientation of $\n_\eps$, $\m_\eps$ (so long as it is chosen consistently
as $\eps \to 0$, so that $\n_\eps \to \n_\star$ and $\m_\eps \to \m_\star$
a.e. in $G$). By letting $K$ and $G$ vary in $\Omega \setminus \spt\mu_\star$,
we can define $\M_\star$ almost everywhere on $\Omega$.
Then,~\eqref{eq:mod-M*} is obvious from~\eqref{eq:def-M*}, which also shows that 
\begin{equation}\label{eq:SM*=Su*}
	\S_{\M_\star} = \S_{\u_\star}.
\end{equation}
Moreover, since the sequence $\{\M_\eps\}$ is uniformly
bounded in $L^\infty(\Omega)$, Lebesgue's dominated convergence theorem
implies that, as $\eps \to 0$, $\M_\eps \to \M_\star$ in $L^p(\Omega)$, for any $p$
with $1 \leq p < +\infty$. 

From here on, continuing to argue as in the proof of~\cite[Proposition~4.11]{CanevariMajumdarStroffoliniWang},
one checks that $\H^1(\S_{\M_*}) < +\infty$ and that $\M_\star$ belongs to
$\SBV(\Omega, \,\R^2)$. Finally,~\eqref{eq:M*-Q*} is an obvious consequence
of~\eqref{eq:Q*-oriented} and of~\eqref{eq:def-M*}. This concludes the proof.
\end{proof}

As in \cite{CanevariMajumdarStroffoliniWang}, a regularity
property of $\M_\star$ outside of $\overline{\S_{\M_\star}}$ follows immediately.

\begin{prop}
	The map $\M_\star$ is locally harmonic in
	$\Omega \setminus \overline{\S_{\M_\star}}$,
	with values in the unit circle of radius $\left( \sqrt{2} \beta + 1 \right)^{1/2}$.
	In particular, $\M_\star$ is smooth in $\Omega \setminus \overline{\S_{\M_\star}}$.
\end{prop}

\begin{proof}
The proof follows exactly as in \cite[Proposition~4.12]{CanevariMajumdarStroffoliniWang}.
\end{proof}

\begin{remark}\label{rk:size-SM*}
	At this stage, however, we do not know yet whether $\overline{\S_{\M_\star}}$
	is a small set or not. We shall see in \cite[Remark~3.4]{CDS2}
	that
	the set $\overline{\S_{\M_\star}}$
	has locally finite $\H^1$-measure.
\end{remark}
The question of the size of the set $\overline{\S_{\M_\star}}$ 
is connected with that of the uniform convergence of $\M_\eps$ 
towards $\M_\star$ away from $\spt\mu_\star \cup \spt \nu_\star$. 
Both these questions will be addressed in the companion paper~\cite{CDS2}, 
as they need a clearing-out result for $\nu_\star$, which will be obtained 
in~\cite{CDS2} at the end of a quite long route. 
With the purpose of keeping~\cite{CDS2} at reasonable size, 
we prefer to avoid introducing again the change of variables 
$\M \mapsto \u$ in that paper, so we anticipate here a couple of lemmata 
that will be useful to prove the uniform convergence mentioned above. 

The first lemma expresses $V(\M)$ in terms of $h(\u)$, $\u$, and $\abs{\Q}$, 
and it follows immediately from the definitions.   
\begin{lemma}\label{lemma:V-h}
Given any pair $(\Q,\,\M) \in \Sz \times \R^2$ with $\Q \neq 0$,
there holds
\begin{equation}\label{eq:V-h}
	V(\M) = h(\u) + \frac{\beta}{\sqrt{2}}(1-\abs{\Q})\left(u_1^2 - u_2^2 - 1 - \frac{\beta +\beta \abs{\Q}}{\sqrt{2}}\right),
\end{equation}
where $\u = (u_1,\,u_2) \in \R^2$ is defined in~\eqref{eq:def-u}.
\end{lemma}

\begin{proof}
	Recalling \eqref{eq:def-u} and~\eqref{eq:QMM},
	identity~\eqref{eq:V-h} is immediately recognised as
	an obvious consequence of the definitions.
\end{proof}

In the second lemma, we take advantage of Lemma~\ref{lemma:V-h}, 
Lemma~\ref{lemma:con-loc-mod+pot}, and Corollary~\ref{cor:linfty-est-Qeps} 
to show that, away from $\spt\mu_\star$, the energies $\mathcal{AC}_\eps(\u_\eps)$ 
and $E_\eps(\M_\eps)$ differ only by terms that vanish as $\eps \to 0$. 
Then, we exploit this fact to obtain an upper bound for the length 
of $\H^1(\S_{\M_\star})$ (i.e.,~\eqref{eq:fail} below).
\begin{lemma}\label{lemma:V-h-limit}
	Let $\left\{\left( \Q_\eps,\,\M_\eps \right)\right\}$ be a sequence of critical
	points of $\F_\eps$, subject to boundary conditions either as
	in~\eqref{bc}--\eqref{hp:bc} or as in \eqref{bcbis}--\eqref{hp:bcbis}, and
	assume that~\eqref{hp:potential_bound} holds.
	Let $B := B(x_0,\,R) \csubset \Omega \setminus \spt \mu_\star$ be any ball.
	Then,
	\begin{equation}\label{eq:V-h-limit}
		\lim_{\eps \to 0}  \int_B \left\{ \frac{1}{\eps}V(x,\,\M_\eps(x)) - \frac{1}{\eps} h(\u_\eps)\right\}\,{\d}x = 0.
	\end{equation}
	Moreover,
	\begin{equation}\label{eq:AC-Eeps-limit}
		\mathcal{AC}_\eps(\u_\eps,\,B) = E_\eps(\M_\eps,\,B) + \o_{\eps \to 0}(1),
	\end{equation}
	and, for any simply connected open set $G \csubset \Omega \setminus \spt \mu_\star$
	with smooth boundary,
	\begin{equation}\label{eq:fail}
		c_\beta \H^1(\S_{\M_\star} \cap G)
		\leq \liminf_{\eps \to 0} E_\eps(\M_\eps;\,G),
	\end{equation}
	where $c_\beta = \frac{2\sqrt{2}}{3}\left( \sqrt{2}\beta + 1 \right)^{3/2}$
	is the constant given by~\eqref{c_beta}. 
\end{lemma}

\begin{proof}
	Given the assumptions, for any $\eps > 0$ small enough we can apply
	Lemma~\ref{lemma:V-h} with $(\Q,\,\M) = (\Q_\eps(x),\,\M_\eps(x))$ for
	any $x \in B$, so that~\eqref{eq:V-h-limit} follows
	by integration of both sides of~\eqref{eq:V-h} over $B$ and taking the
	limit $\eps \to 0$, and then
	using Lemma~\ref{lemma:con-loc-mod+pot},
	Corollary~\ref{cor:linfty-est-Qeps}, and~\eqref{eq:h-bdd} to show
	that the integral of the second term on the right-hand side of~\eqref{eq:V-h}
	tends to zero as $\eps \to 0$.

	Concerning~\eqref{eq:AC-Eeps-limit}, it follows from~\eqref{eq:V-h-limit} and
	integrating both sides of~\eqref{eq:AC-bound-compu3} over $B$,
	using H\"older's inequality
	on the middle term in the right-hand side and the bound~\eqref{eq:M_bound}.

	About~\eqref{eq:fail}, it follows from~\eqref{eq:AC-Eeps-limit}, a standard
	covering argument, Theorem~\ref{thm:FonsecaTartar} and the fact that, 
	by~\eqref{eq:SM*=Su*}, $\S_{\u_\star} = \S_{\M_\star}$. 
\end{proof}

\section{Proof of Theorem~\ref{mainthm:asymp}}
\label{sect:mainthm}

We are finally ready for the proof of
Theorem~\ref{mainthm:asymp}. In fact,
all the claims have been already proved
in the previous sections, but, for the reader's convenience,
we collect the appropriate references
below.

\begin{proof}[{Proof of Theorem~\ref{mainthm:asymp}}]
	The first part (i.e., the part relative to the $\Q$-component)
	of statement~\ref{item:mainthm-asymp-conv-QM} follows from
	Proposition~\ref{prop:strong-conv-Qeps}, while the second
	part comes from Theorem~\ref{thm:Lp-conv-Meps-M*}.
	About~\ref{item:mainthm-asymp-Q*M*-j}, it follows from
	Lemma~\ref{lemma:weak-conv-Q*}, Proposition~\ref{prop:Q*-harm},
	and (again) Theorem~\ref{thm:Lp-conv-Meps-M*}.
	From Lemma~\ref{lemma:mu*}, we have the 
	statement~\ref{item:mainthm-asymp-conv-mu*nu*}.
	Concerning~\ref{item:mainthm-asymp-supports}, 
	it follows from Lemma~\ref{lemma:mu*} 
	Finally,
	statement~\ref{item:mainthm-asymp-strong-conv}
	is part of Proposition~\ref{prop:strong-conv-Qeps}. 
\end{proof}

\appendix
\section{The pure Dirichlet case: $W^{1,2}_{\loc}$-convergence up to the boundary for $\Q_\eps$}\label{app:dir}
In this appendix, we assume that $\{\left(\Q_\eps,\,\M_\eps\right)\}$ is a sequence
of critical points of $\F_\eps$ subject to the pure Dirichlet boundary
conditions~\eqref{bc},~\eqref{hp:bc} and, arguing along the lines of
\cite[Theorem~X.3]{BBH}, we prove that $\Q_\eps \to \Q_\star$ strongly in
$W^{1,2}_{\rm loc}\left(\overline{\Omega} \setminus \spt\mu_\star\right)$, i.e.,
locally up to the boundary but away from $\spt\mu_\star$.
More precisely, we have the following proposition.

\begin{prop}\label{prop:strong-conv-bdry-Qeps}
Let $x_0 \in \partial \Omega$ and let $B(x_0,\,R)$ be any ball satisfying
\begin{equation}\label{eq:ball-far-from-mu*}
	\overline{B(x_0,\,2R)} \cap \spt\mu_\star = \emptyset.
\end{equation}
Then, as $\eps \to 0$, we have $\Q_\eps \to \Q_\star$ strongly in
$W^{1,2}\left(B(x_0,\,R) \cap \Omega\right)$.
\end{prop}

As a first step towards the proof of
Proposition~\ref{prop:strong-conv-bdry-Qeps}, we state a variant of
the Pohozaev identity in Lemma~\ref{lemma:pohozaev}.

\begin{lemma}\label{lemma:bdry-pohozaev}
	Assume that $\Omega$ is star-shaped about a point $x_\star \in \Omega$.
	Then, there holds
	\begin{equation}\label{eq:bdry-pohozaev}
	\begin{split}
		\frac{2}{\eps^2} \int_{\Omega} f_\eps(\Q_\eps, \M_\eps)\,{\d}x
		&+ \frac{R}{2} \int_{\partial \Omega} \left( \abs{\partial_{\nnu} \Q_\eps }^2 + \eps\abs{\partial_{\nnu} \M_\eps}^2   \right)\,{\d}s \\
		&\lesssim \frac{R}{2} \int_{\partial \Omega} \left( \abs{\partial_\ttau \Q_\eps}^2 + \eps \abs{\partial_\ttau \M_\eps}^2 + \frac{2}{\eps^2}f_\eps(\Q_\eps,\,\M_\eps) \right)\,{\d}s
	\end{split}
	\end{equation}
	where the implicit constant on the right-hand side depends only on $\Omega$
	and the boundary data. 
\end{lemma}

\begin{proof}
	Take $\X := x - x_\star$ in~\eqref{stren4}. Since $\Omega$ is star-shaped, there
	is a constant $\alpha_\Omega > 0$ such that
	\[
		(x-x_\star) \cdot \nnu \geq \alpha_\Omega, \qquad \forall x \in \partial \Omega.
	\]
	On the other hand, precisely as in \cite[(12) of Chapter~{III}]{BBH}), we have
	\[
		(\nnu \cdot \nabla \Q_\eps) \cdot (\X \cdot \nabla \Q_\eps )
		=  \left((x-x_\star) \cdot \nnu\right) \abs{\nnu \cdot \nabla \Q_\eps}^2 +
		\left((x-x_\star) \cdot \ttau \right) (\partial_\ttau \Q_\eps \cdot \partial_\nnu \Q_\eps)
	\]
	and the same holds for $\M_\eps$.
	Combing these facts and applying suitably Young's inequality,
	we obtain~\eqref{eq:bdry-pohozaev}, with a constant depending only on
	$\alpha_\Omega$
	and the boundary data.
\end{proof}

\begin{remark}\label{rk:poho-dir}
	If $\{(\Q_\eps,\,\M_\eps)\}$ satisfies the pure Dirichlet boundary
	conditions~\eqref{bc}--\eqref{hp:bc}, then, in view of~\eqref{bc}--\eqref{hp:bc},
	and item~\ref{item:good-data} of Lemma~\ref{ell_eps}, the right-hand side
	of~\eqref{eq:bdry-pohozaev} is bounded by a constant depending only on
	$\alpha_\Omega$, $\beta$, and the boundary data.
\end{remark}

\begin{proof}[{Proof of Proposition~\ref{prop:strong-conv-bdry-Qeps}}]
	As in the interior case, we argue in two steps, proving first that $\rho_\eps \to 1$
	and then that $\j_\eps \to \j_\star$ strongly in
	$W^{1,2}\left(\Omega \cap B(x_0,\,R)\right)$ as $\eps \to 0$.

	Let $\zeta \in C^\infty_c(B(x_0,\,2R))$ be a smooth-cut off function
	such that
	\[
		0 \leq \zeta \leq 1, \qquad \zeta \equiv 1 \quad \mbox{in } B(x_0,\,R).
	\]
	\setcounter{step}{0}
	\begin{step}[$\rho_\eps \to 1$ in $W^{1,2}\left(B(x_0,\,R) \cap \Omega\right)$]
	\label{step:bdry-rho}
	This step involves the Pohozaev-type inequality~\eqref{eq:bdry-pohozaev}
	and we will crucially rely on the Dirichlet boundary
	conditions~\eqref{bc}--\eqref{hp:bc} to control its right-hand side.

	In view of~\eqref{bc}--\eqref{hp:bc}, we may extend each pair $(\Q_\eps,\,\M_\eps)$
	as explained in Remark~\ref{rk:extension} to $\Omega' \supset \Omega$, an open set
	containing a $(2R)$-neighbourhood of $\partial \Omega$.
	Since $\rho_\eps \equiv 1$ on $\partial \Omega$, we then have
	$\rho_\eps > 1/2$ eventually on $B(x_0,\,3R/2)$.
	Therefore, arguing as in the proof of Lemma~\ref{lemma:con-loc-mod+pot},
	we can take the equation~\eqref{eq:Delta-rho} and multiply it by
	$\zeta(\rho_\eps -1 -\kappa_\star\eps)$. Then, recalling that $\spt \zeta$ is far
	from $\spt\mu_\star$, we integrate over $\Omega$
	to yield (we drop the subscripts $\eps$ for notational simplicity)
	\[
	\begin{split}
		\int_{\Omega} &\left\{ \zeta \abs{\nabla \rho}^2 + \frac{1}{2}\nabla \zeta \cdot \nabla \left(\rho-1-\kappa_\star \eps\right)^2 + \zeta \rho(\rho-1-\kappa_\star \eps)\abs{\nabla \varphi}^2 \right.\\
		&\left.+ \zeta \left( \frac{\rho-1}{\eps} - \kappa_\star \right) \left( \frac{\rho-1}{\eps}(\rho+1)\rho -\sigma \right) \right\} \,{\d}x = \int_{\partial \Omega} \zeta (\rho -1 -\kappa_\star \eps) \partial_{\nnu}\rho\,{\d}s.
	\end{split}
	\]
	Since $\rho \equiv 1$ on $\partial\Omega$, from the Pohozaev
	identity~\eqref{eq:bdry-pohozaev}, the imposed boundary
	conditions~\eqref{bc}--\eqref{hp:bc}, and Remark~\ref{rk:poho-dir} it follows that
	\[
	\int_{\partial \Omega} \zeta (\rho -1 -\kappa_\star \eps) \partial_{\nnu}\rho\,{\d}s
	= - \eps \kappa_\star \int_{\partial \Omega} \zeta \partial_{\nnu} \rho = \o_{\eps\to 0}(1).
	\]
	By Remark~\ref{rk:extension}, Theorem~\ref{lemma:energy-est},
	and an argument via Gagliardo-Nirenberg interpolation completely analogous to that
	in Corollary~\ref{cor:unif-conv-Qeps-Q*}, we have
	\begin{equation}\label{eq:unif-rho-bdry}
		\rho_\eps \to 1 \qquad \mbox{uniformly on } \overline{B(x_0,\,2R)},
		\,\mbox{ as } \eps \to 0.
	\end{equation}
	Using these facts, and recalling that $\spt \zeta \subset B(x_0,\,2R)$,
	we can argue, from now on, as in Lemma~\ref{lemma:con-loc-mod+pot},
	with no changes.
	\end{step}

	\begin{step}[$\j_\eps \to \j_\star$ in $W^{1,2}\left(B(x_0,\,R) \cap \Omega\right)$]
	\label{step:bdry-j}
	As in the interior case dealt with in Lemma~\ref{lemma:W12-strong-conv-Qeps}, we consider
	a suitable, `Hodge-type' decomposition of $\j_\eps := j(\Q_\eps)$, i.e., we write
	\[
		\j_\eps = \nabla H_\eps + \nabla^\perp \phi_\eps + \nabla^\perp \xi_\eps
		\qquad \mbox{in } \Omega,
	\]
	where $H_\eps$, $\phi_\eps$, and $\xi_\eps$ solve, respectively,
	\eqref{H_eps}-\eqref{Heps-bc},~\eqref{phi_eps}, and~\eqref{xi_eps}, and we
	study the convergence of $H_\eps$, $\phi_\eps$, and $\xi_\eps$ separately.

	Since $H_\eps = 0$ and $\partial_{\nnu}\xi_\eps = 0$ on $\partial \Omega$
	for any $\eps > 0$, in these cases there are no boundary terms and we can argue
	exactly as in the interior case, using the cut-off function $\zeta$
	above and~\eqref{eq:unif-rho-bdry}, without other changes.

	Concerning $\phi_\eps$, by~\eqref{eq:ball-far-from-mu*},~\eqref{eq:unif-rho-bdry},
	and~\eqref{preJaccurl}, we may assume that $\curl(\alpha_\eps^2 \j_\eps) \equiv 0$
	on $\spt \zeta$ for any $\eps > 0$ small enough.
	Thus, multiplying~\eqref{phi_eps} by $\zeta(\phi_\eps - \phi_\star)$,
	and integrating over $\Omega$, we have
	\begin{equation}\label{eq:bdry-phi-eps-1}
		\int_\Omega \zeta \nabla(\phi_\eps - \phi_\star) \cdot \nabla \phi_\eps \,{\d}x
		= \frac{1}{2} \int_{\partial\Omega} \zeta( \phi_\eps - \phi_\star)\left(\Qb \times \partial_{\ttau}\Qb\right)\,{\d}s + \o_{\eps \to 0}(1).
	\end{equation}
	Next, we observe that from~\eqref{phieps-p} and Sobolev embedding (since
	$\partial \Omega$ is Lipschitz and compact), we have
	\[
		\phi_\eps \rightharpoonup \phi_\star \qquad \mbox{weakly in }
		W^{1,p}(\Omega),
	\]
	for any $1 \leq p < 2$ as $\eps \to 0$, so that, again by Sobolev embedding
	and the Lipschitz regularity of $\partial \Omega$, we obtain
	\begin{equation}\label{eq:bdry-phi-eps-2}
		\phi_\eps \to \phi_\star \qquad \mbox{strongly in }
		L^p(\partial \Omega).
	\end{equation}
	for any $1 \leq p < 2$ as $\eps \to 0$.
	From here, it follows that
	\[
		\norm{\nabla \phi_\eps}_{L^2(\Omega \cap B(x_0,\,R))} \leq C,
	\]
	where $C$ does not depend on $\eps$, and thus
	\[
		\phi_\eps \rightharpoonup \phi_\star \qquad \mbox{weakly in }
		W^{1,2}(\Omega \cap B(x_0,\,R))
	\]
	as $\eps \to 0$. On the other hand,
	using~\eqref{eq:bdry-phi-eps-2} together with H\"older's inequality,
	we see that the right-hand side of~\eqref{eq:bdry-phi-eps-1} vanishes as
	$\eps \to 0$, so that
	\[
		\phi_\eps \to \phi_\star \qquad \mbox{strongly in }
		W^{1,2}(\Omega \cap B(x_0,\,R))
	\]
	as $\eps \to 0$.
	\end{step}

	\begin{step}[Conclusion]
	Putting together Step~\ref{step:bdry-rho} and Step~\ref{step:bdry-j}
	leads directly to the conclusion of the proof.
	\qedhere
	\end{step}
\end{proof}

\paragraph{Acknowledgments}
The authors are members of GNAMPA-INdAM. Their work has been
partially supported by GNAMPA projects CUP\_E53C22001930001{,} CUP\_E53C23001670001{,}
and CUP\_E53C25002010001.
B.S.'s research
is part of the project ``Geometric Evolution Problems and Shape
Optimizations'', PRIN  Project 2022E9CF89.
G.C. is part of the ANR project ``Singularities of energy-minimizing vector-valued maps'', ref. ANR-22-CE40-0006.
Part of the reserach that lead to the present work was carried
out while the authors participated in the summer school
``Variational and PDE Methods in Nonlinear Science'', organised by
the C.I.M.E. Foundaton in Cetraro (Italy), July~2023, and in the workshop
``24w5249 --- Mathematical Analysis of Soft Matter'',
organised by the Banff International Research Station
in Banff (Canada) in July 2024.
F.L.D. would like to thank the University of Verona
for hospitality during the last part of this work.
Likewise, G.C. is grateful to the University of Napoli Federico~II
for hospitality during several visits.

\bibliographystyle{plain}
\bibliography{UnifConv}

\begin{flushright}
\Addresses
\end{flushright}

\end{document}